\documentclass[intlim,righttag,10pt]{amsart}
\usepackage{amscd}
\usepackage{amssymb}
\usepackage[all]{xy}
\def\invlim{\mathop{\vtop{\ialign{##\crcr$\hfill{\lim}\hfil$\crcr
\noalign{\kern1pt\nointerlineskip}\leftarrowfill\crcr\noalign
{\kern -3pt}}}}\limits}
\def\dirlim{\mathop{\vtop{\ialign{##\crcr$\hfill{\lim}\hfil$\crcr
\noalign{\kern1pt\nointerlineskip}\rightarrowfill\crcr\noalign
{\kern -3pt}}}}\limits}
\def\lomapr#1{\smash{\mathop{\relbar\joinrel\longrightarrow}\limits^{#1}}}
 \def\verylomapr#1{\smash{\mathop{\relbar\joinrel\relbar\joinrel\relbar\joinrel\longrightarrow}\limits^{#1}}}
\def\veryverylomapr#1{\smash{\mathop{\relbar\joinrel\relbar\joinrel\relbar
\joinrel\relbar\joinrel\relbar\joinrel\longrightarrow}\limits^{#1}}}
\def\phi{\varphi}
\def\epsilon{\varepsilon}

\newtheorem{theorem}{Theorem}[section]
 \newtheorem{lemma}[theorem]{Lemma}
 \newtheorem{proposition}[theorem]{Proposition}
 \newtheorem{corollary}[theorem]{Corollary}

\theoremstyle{definition}
\newtheorem{definition}[theorem]{Definition}
\theoremstyle{remark}
\newtheorem{remark}[theorem]{Remark}

\newtheorem{example}[theorem]{Example}
\newtheorem*{acknowledgments}{Acknowledgments}

 \newcommand{\Nis}{\operatorname{Nis}}
 \newcommand{\Frac}{\operatorname{Frac} }
 
\newcommand{\ovk}{\overline{K} }
 
\newcommand{\keet}{\operatorname{K\acute{e}t}}
\newcommand{\ord}{\operatorname{ord} }

\newcommand{\gp}{\operatorname{gp} }

 \newcommand{\cl}{\operatorname{cl} } 
 \newcommand{\dr}{\operatorname{dR} }

 \newcommand{\cosk}{\operatorname{cosk} }
 \newcommand{\skl}{\operatorname{sk} }

 \newcommand{\holim}{\operatorname{holim} }
 \newcommand{\hocolim}{\operatorname{hocolim} }
\newcommand{\cofiber}{\operatorname{cofiber} }
\newcommand{\fiber}{\operatorname{fiber} }

 \newcommand{\eet}{\operatorname{\acute{e}t} }

 \newcommand{\dlog}{\operatorname{dlog} }

 \newcommand{\Zar}{\operatorname{Zar} }

 \newcommand{\Spec}{\operatorname{Spec} }

 \newcommand{\cd}{\operatorname{cd} } 
 \newcommand{\Hom}{\operatorname{Hom} }

 \newcommand{\Gal}{\operatorname{Gal} }
 \newcommand{\tr}{ \operatorname{tr} }

 \newcommand{\id}{ \operatorname{Id} }
\newcommand{\synt}{ \operatorname{syn} }
 \newcommand{\res}{ \operatorname{res} }
 \newcommand{\Cone}{\operatorname{Cone} }

 \newcommand{\kker}{\operatorname{Ker} }
 \newcommand{\crr}{\operatorname{cr} }
 \newcommand{\gr}{\operatorname{gr} }
 \newcommand{\im}{\operatorname{Im} }

 \newcommand{\ve}{ \varepsilon  }
  \newcommand{\kr}{^{\scriptscriptstyle\bullet}}

 \newcommand{\sff}{{\mathcal{F}}}

 \newcommand{\sh}{{\mathcal{H}}}

 \newcommand{\scc}{{\mathcal{C}}}
 \newcommand{\sk}{{\mathcal{K}}}
 \newcommand{\sll}{{\mathcal{L}}}
 \newcommand{\sn}{{\mathcal N}}
 \newcommand{\so}{{\mathcal O}}
 \newcommand{\sj}{{\mathcal J}}
 \newcommand{\se}{{\mathcal{E}}}

 \newcommand{\sss}{{\mathcal{S}}}
\newcommand{\sd}{{\mathcal{D}}} 
 
\newcommand{\spp}{{\mathcal{P}}}

 \newcommand{\wt}{\widetilde}

\newcommand{\Q}{\mathbf{Q}}
\newcommand{\Z}{\mathbf{Z}}

\newcommand{\R}{{\mathrm R}}

\newcommand{\M}{{\mathrm M}}
\newcommand{\A}{{\mathbf{A}}}
\numberwithin{equation}{section}
 
\begin{document}
 \title[On syntomic regulators I: constructions]
 {On syntomic regulators I: constructions}
 \author{Wies{\l}awa Nizio{\l}}
 \address{CNRS, UMPA, \'Ecole Normale Sup\'erieure de Lyon, 46 all\'ee d'Italie, 69007 Lyon, France}
\email{wieslawa.niziol@ens-lyon.fr} \date{\today}
\thanks{This research was supported in part by  the NSF grant DMS0703696 and  the grant ANR-14-CE25.}
 \maketitle 
 \begin{abstract}
 We show that  classical Chern classes from higher ($p$-adic) $K$-theory to syntomic cohomology extend to logarithmic syntomic cohomology. These Chern classes are compatible -- in a suitable sense -- with addition, products, and $\lambda$-operations. They are also compatible with the canonical Gysin sequences and, via period maps, with logarithmic \'etale Chern classes. Moreover, they induce logarithmic crystalline Chern classes. This uses as a critical new ingredient the recent comparison of syntomic cohomology with  $p$-adic nearby cycles \cite{CN} and $p$-adic motivic cohomology \cite{EN}.
 \end{abstract}
 \tableofcontents
 \section{Introduction}
  The study of syntomic regulators,  a $p$-adic analogue of Deligne's regulators,   is important in $p$-adic Hodge Theory \cite{N4,N2} and in computations of special values of $p$-adic $L$-functions \cite{Sw,BDR}. In this paper we show that syntomic regulators have a well-behaved logarithmic version that is compatible, via the period map,  with logarithmic \'etale regulators. In the sequel to this paper we will use it to extend results from proper schemes to open schemes with normal crossing compactification.
 
 Let $p$ be a prime. Let $K$ be a complete discrete valuation field of mixed characteristic $(0,p)$ with perfect residue field; let $\so_K$ be its ring of integers. 
 Let $X$ be a semistable scheme over $\so_K$ and let  $D$ be the  canonical horizontal normal crossings divisor on $X$. We equip $X$ with the   log-structure associated to $D$ and the special fiber. Alternatively, $X$ could be a smooth scheme over $\so_K$ with the log-structure associated to a divisor $D$ as above.  Denote by $X_0$ the special fiber of $X$ and by $X_n$ -- the reduction mod $p^n$ of $X$.

 For $n\geq 1,i\geq 0$, let $$
 \sss^{\prime}_n(i)_X:=(U\to X) \mapsto \Cone(F^i\R\Gamma_{\crr}(U_n)\lomapr{p^i-\phi}\R\Gamma_{\crr}(U_n))[-1]
 $$
 be the syntomic cohomology complex of sheaves on the \'etale site of $X$. The crystalline cohomology used is absolute, i.e. over $\Z/p^n$.
  Let $N$ be a constant as in \cite[Theorem 1.1]{CN}.  Recall that, for $p\geq 3$,  if  $K$ contains enough roots of unity\footnote{See Theorem \ref{input1} for what this means.} then $N$ depends only on $p$; in general it depends on $p$ and $e$ -- the absolute ramification index of $K$.

   The main goal of this paper is to prove the following theorem. 
 \begin{theorem}
 \label{boston1} Fix $m\geq N$. Let $U=X\setminus D$. 
 There exists a functorial compatible family of  logarithmic syntomic Chern classes\footnote{They vary with $m$ in the obvious way.}
 $$
 \overline{c}^{\synt}_{ij}:\quad K_j(U, \Z/p^n)\to H^{2i-j}_{\synt}(X,\sss^{\prime}_n(i)),\quad j\geq 2,
 $$
 such that 
 \begin{enumerate}
 \item they 
  are compatible -- in a suitable sense -- with addition, products, and $\lambda$-operations.
  \item   they are twisted, i.e., in the case of trivial divisor $D$ the Chern class $ \overline{c}^{\synt}_{ij}$ is equal to the $p^{mi}$-multiple of the classical syntomic Chern class. 
 \item 
  they  are compatible with the canonical Gysin sequences, i.e., the Gysin sequences associated to $D$.
\item  they are compatible, via the period maps $\alpha_{*,*}$ of Fontaine-Messing, 
  with \'etale Chern classes, i.e., the following diagram commutes
   $$
   \xymatrix@C=50pt{
   K_j(U, \Z/p^n)\ar[d]^{\jmath^*} \ar[r]^-{\overline{c}^{\synt}_{ij}}  & H^{2i-j}_{\synt}(X,\sss^{\prime}_n(i)) \ar[d]^{\alpha_{2i-j,i} } \\
    K_j(U_K, \Z/p^n) \ar[r]^-{p^{(m+1)i}\overline{c}^{\eet}_{ij}}  & H^{2i-j}_{\eet}(U_K,\Z/p^n(i)),   }
   $$
   where $\jmath:U_K\hookrightarrow U$ is the natural open immersion.
   \end{enumerate} 
   Similarly, there exists a functorial compatible family of  logarithmic syntomic Chern classes
 $$
 c^{\synt}_{ij}:\quad K_j(U)\to H^{2i-j}_{\synt}(X,\sss^{\prime}_n(i)),\quad j\geq 0,
 $$
 with the above listed properties.
 \end{theorem}
 \begin{remark} From the above theorem we obtain   $p$-adic logarithmic Chern classes. More specifically, 
 set 
 \begin{align*}
  K_j(U, \Q_p) :=\Q\otimes \invlim_n K_j(U,\Z/p^n),\quad 
   H^{2i-j}_{\synt}(X,\sss^{\prime}_{\Q}(i)) := \Q\otimes\invlim_n  H^{2i-j}_{\synt}(X,\sss^{\prime}_n(i)).
\end{align*}
The limit of Chern classes $\overline{c}^{\synt}_{ij}$ divided by $p^{mi}$ yields  (untwisted) logarithmic Chern classes 
$$
 \overline{c}^{\synt}_{ij}:\quad K_j(U, \Q_p)\to H^{2i-j}_{\synt}(X,\sss^{\prime}_{\Q}(i)),\quad j\geq 2.
$$
\end{remark}

 The strategy for proving the above theorem is well-known.  In the simplest case,  logarithmic cohomology is equal to the cohomology of the open set, where the log-structure is trivial. The definition of the logarithmic Chern classes is then immediate. To treat compatibility with Gysin sequences, one proves  a version of Grothendieck-Riemann-Roch to extend  the classical universal Chern classes to the logarithmic ones that are compatible (by definition) with the Gysin diagram and then one  uses cohomological purity to show  uniqueness of such an extension. This works, for example, for logarithmic $\ell$-adic \'etale cohomology.

 In general, logarithmic cohomology does not have the above mentioned property. This is the case, for example, for crystalline cohomology. Then 
one constructs a well-behaved new cohomology that dominates the given cohomology and one constructs the universal logarithmic Chern classes into the new  cohomology having all the necessary properties. For crystalline cohomology over a field, the new cohomology is the logarithmic de Rham-Witt cohomology.

  Logarithmic syntomic cohomology $\sss^{\prime}_n(*)$ by itself does not behave well enough to employ the above strategy. In its classical form  it satisfies a version of the projective space theorem and gives  the correct even-degree cohomology groups of the classifying spaces. Hence it has well-behaved classical Chern classes.  But it is too weak to allow those to be extended to  logaritmic cohomology: it is not equal to  syntomic cohomology of the open set, where the log-structure is trivial, nor does it satisfy purity. 
  
   To obtain a well-behaved logarithmic cohomology dominating syntomic cohomology one replaces syntomic cohomology by (Nisnevich) syntomic-\'etale cohomology $\se^{\prime}_n(*)_{\Nis}$. The latter is defined by gluing  \'etale cohomology of the Tate twist $\Z/p^n(i)$ on the generic fiber  with the syntomic cohomology $\sss^{\prime}_n(i)$ of the formal special fiber, projecting the result  down to the Nisnevich site, and truncating it at $i$. We have a natural map $\se^{\prime}_n(i)_{\Nis}\to \sss^{\prime}_n(i)_{\Nis}$, where the target is defined by projecting syntomic cohomology $\sss^{\prime}_n(i)$ to the Nisnevich site and truncating it at $i$.  Hence it suffices to construct logarithmic universal Chern classes with values in  syntomic-\'etale cohomology. This cohomology  has the properties necessary for the above strategy to work. To get an idea why this could be the case, let us look at the simpler case when $X$ is semistable. Recall that, by $p$-adic comparison theorems \cite{CN}, the logarithmic syntomic cohomology $\sss^{\prime}_n(i)$  is approximated (via the Fontaine-Messing period map) by   the $p$-adic nearby cycles $i^*\tau_{\leq j}\R j_*\Z/p^n(i)$, $i: X_0\hookrightarrow X$, $j: U_K\hookrightarrow X$. It follows that the syntomic-\'etale cohomology $\se^{\prime}_{n}(i)_{\Nis}$ is approximated by
the  truncated \'etale cohomology projected onto the Nisnevich site: $\tau_{\leq i}\R (j\varepsilon)_*\Z/p^n(i)_{U_K}$, $\varepsilon: U_{K,\eet}\to U_{K,\Nis}$. But by the Beilinson-Lichtenbaum conjecture we have  a quasi-isomorphism
 $\tau_{\leq j}\R \varepsilon_*\Z/p^n(i)_{U_K}\simeq \Z/p^n(j)_{\M,U_K}$, where $\Z/p^n(i)_{\M}$ is the motivic cohomology. Hence the syntomic-\'etale cohomology $\se^{\prime}_n(i)_{\Nis}$ is approximated 
 by $\tau_{\leq i}\R j_*\Z/p^n(i)_{\M,U_K}=\R j_*\Z/p^n(i)_{\M,U_K}$. In particular it is equal to syntomic-\'etale cohomology of the open set, where the log-structure is trivial. 
Moreover, since motivic cohomology has purity so does  syntomic-\'etale cohomology.   We got all the properties we wanted. In the more difficult case of good reduction  the role of the motivic cohomology $\Z/p^n(i)_{\M,U_K}$ is played  by the motivic cohomology $\Z/p^n(i)_{\M,U}$; that this can be done follows  from the comparison between syntomic cohomology and $p$-adic motivic nearby cycles proved in \cite{EN}. 

  The approximation mentioned above is done up to  universal constants (that can be controlled). This results in the twisting of Chern classes in the above theorem. Keeping track of those constants is the most tedious part of the paper. For small Tate twists $i$ these constants can be taken to be trivial.
\begin{remark}
It seems like an overkill to use a very difficult theorem like the Beilinson-Lichtenbaum conjecture to prove properties of such a seemingly simple object as the truncated \'etale cohomology $\tau_{\leq i}\R j_*\Z/p^n(i)$. Yet, even to prove the projective space theorem  to define  classical Chern classes with values in this cohomology and to construct Gysin sequences one uses \cite[Theorem 4.1]{Sa} the computations of mod $p$ nearby cycles via symbols due to Bloch-Kato \cite{BK} which are closely related to the Bloch-Kato conjecture (and hence to the Beilinson-Lichtenbaum conjecture)\footnote{It is interesting that it is much easier to prove projective space theorem and to construct Gysin sequences for syntomic cohomology: it simply reduces to the same for filtered crystalline cohomology, where it is immediate.}.
\end{remark}
   \begin{remark}
 There is another strategy that we could have employed to construct logarithmic syntomic Chern classes. One starts with the classical Chern classes, proves the Grothendieck-Riemann-Roch theorem for them, and then uses it and Gysin seuqences to induce logarithmic Chern classes on the complement of the divisor. This has to be done one irreducible divisor at a time; in particular, the logarithmic Chern classes have to have all the properties necessary for the standard proof of the Grothendieck-Riemann-Roch theorem to work (which, basically means, that they have to be compatible with the action of $K_0$-groups in a suitable sense).
 
  In \cite{Sw}, Somekawa tried to make this strategy work. His arguments work for removing one irreducible divisor but fail on the inductive step: he was not able to show that so obtained logarithmic Chern classes have good properties. This seems highly nontrivial. This paper shows that this strategy actually works but with syntomic-\'etale cohomology in place of syntomic cohomology. Purity of syntomic-\'etale cohomology is the key property that allows the inductive step. But this property also implies that logarithmic cohomology is equal to   cohomology of the open set, where the log-structure is trivial. Hence we have chosen this (instead of Gysin sequences) as the starting point of  the construction of logarithmic syntomic Chern classes. Compatibility with Gysin sequences is then a theorem. 
 \end{remark}
\subsubsection{Structure of the paper}    In Section $2$ we review the definitions and basic properties of syntomic and syntomic-\'etale cohomologies and recall their relationship to $p$-adic nearby cycles and $p$-adic motivic cohomology. In Section $3$ we study cohomology of classifying spaces: we compute their (Nisnevich) syntomic cohomology and we show that their Nisnevich syntomic-\'etale  and   syntomic cohomologies agree  in even degrees. 
In Section $4$ we review the basic facts concerning higher $K$-theory of (simplicial) schemes and  operations on $K$-theory. In Sections $5$ we define and study properties of classical Nisnevich syntomic and syntomic-\'etale Chern classes -  by standard arguments this builds on the computations done  in Section $3$. Then we introduce logarithmic syntomic-\'etale Chern classes and discuss purity. In Section $6$ we study compatibility of these Chern classes with Gysin sequences. We start with proving the  properties of syntomic-\'etale cohomology that are needed for the proof of the Grothendieck-Riemann-Roch theorem. The proof of the theorem itself follows along standard lines. 
\begin{acknowledgments}I would like to thank Fr\'ed\'eric D\'eglise and Marc Levine for answering   questions concerning motivic cohomology.
Parts of this paper were written during my visits to the Institut Henri Poincar\'e in Paris, the Institut de Math\'ematiques de Jussieu,
Columbia University, and the IAS at Princeton. 
I would like to thank these institutions  for their support, hospitality,
 and for the  wonderful working conditions they have supplied. Much of this paper was written in Caf\'e Nizio{\l} in downtown Ko{\l}obrzeg, Poland. I would like to thank Jaros{\l}aw  Nizio{\l} for creating such a pleasant coffee house  to work at.
 \end{acknowledgments}
 \subsubsection{Notation and Conventions}Unless stated otherwise we work in the category of fine log-schemes. 

 \begin{definition}
\label{1saint}
Let $N\in {\mathbf N}$. For a morphism $f: M\to M^{\prime}$ of ${\mathbf Z}_p$-modules, we say that $f$ is 
{\it $p^N$-injective} (resp. {\it $p^N$-surjective}) if its kernel (resp. its cokernel) is annihilated by $p^N$ 
and we say that $f$ is {\it $p^N$-isomorphism} if it is $p^N$-injective and $p^N$-surjective. 
We define in the same way the notion of {\it $p^N$-exact sequence} or $p^N$-acyclic complex 
(complex whose cohomology groups are annihilated by $p^N$) as well as the notion of {\it $p^N$-quasi-isomorphism}
 (map in the derived category that induces a $p^N$-isomorphism on cohomology). 
\end{definition}
 
\section{Syntomic cohomology}Let $\so_K$ be a complete discrete valuation ring with fraction field
$K$  of characteristic 0 and with perfect
residue field $k$ of characteristic $p$. Let $\varpi$ be a uniformizer of $\so_K$. Let $\overline{\so}_K$ denote the integral closure of $\so_K$ in $\ovk$. Let
$W(k)$ be the ring of Witt vectors of $k$ with 
 fraction field $F$ (i.e, $W(k)=\so_F$); let $e$ be the ramification index of $K$ over $F$.  Set $G_K=\Gal(\overline {K}/K)$, and 
let $\sigma$ be the absolute
Frobenius on $W(\overline {k})$. 
For a $\so_K$-scheme $X$, let $X_0$ denote
the special fiber of $X$. We will denote by $\so_K$ and 
${\so_K}^{\times}$ the scheme $\Spec ({\so_K})$ with the trivial and the canonical
(i.e., associated to the closed point)
log-structure, respectively.

  In this section we will briefly review the definitions of syntomic and syntomic-\'etale cohomologies and their basic properties.
 For details we refer the reader to \cite[2]{Ts}, \cite[2]{EN}.
\subsection{Syntomic cohomology}
 For a log-scheme $X$ we denote by $X_{\synt}$ the small log-syntomic site of $X$. 
For a log-scheme $X$ log-syntomic over $\Spec(W(k))$, define 
$$
\so^{\crr}_n(X) =H^0_{\crr}(X_n,\so_{X_n}),\qquad 
\sj_n^{[r]}(X) =H^0_{\crr}(X_n,\sj^{[r]}_{X_n}),
$$
where $\so_{X_n}$ is the structure sheaf of the absolute log-crystalline site (i.e., over $W_n(k)$), 
$\sj_{X_n}=\kker(\so_{X_n/W_n(k)}\to \so_{X_n})$, and $\sj^{[r]}_{X_n}$ is its $r$'th divided power of $\sj_{X_n}$.
Set $\sj^{[r]}_{X_n}=\so_{X_n}$ if $r\leq 0$.
 There is a  canonical, compatible with Frobenius,  and 
functorial isomorphism 
$$
H^*(X_{\synt},\sj_n^{[r]})\simeq H^*_{\crr}(X_n,\sj^{[r]}_{X_n}).
$$
It is easy to see that $\phi(\sj_n^{[r]} )\subset p^r\so^{\crr}_n$ for $0\leq r\leq p-1$. This fails in general and we modify
$\sj_n^{[r]}$:  $$\sj_n^{<r>}:= \{x\in \sj_{n+s}^{[r]}\mid \phi(x)\in p^r\so^{\crr}_{n+s}\}/p^n ,$$
for some $s\geq r$.
This definition is independent of $s$. We can define
the divided  Frobenius $\phi_r="\phi/p^r": \sj_n^{<r>} \to \so^{\crr}_n$.
Set $$\sss_n(r):=\Cone(\sj_n^{<r>} \stackrel{1-\phi_r}{\longrightarrow}\so^{\crr}_n)[-1].$$

     We will write $\sss_n(r)$ for the syntomic sheaves on $X_{m,\synt}$, $m\geq n$,  as well as on $X_{\synt}$. We will also need the "`undivided"' version of syntomic complexes of sheaves:
$$\sss'_n(r):=\Cone(\sj_n^{[r]} \stackrel{p^r-\phi}{\longrightarrow}\so^{\crr}_n)[-1]$$
as well as their twists
$$\sss^{i}_n(r):=\Cone(\sj_n^{[r]} \verylomapr{p^{r+i}-p^i\phi}\so^{\crr}_n)[-1], \quad  i\geq 0.
$$
We note that the following sequence is exact for $r\geq 0$
\begin{equation}
\label{exact}
\begin{CD}
0@>>> \sss_n(r)@>>> \sj_n^{<r>} @>1-\phi_r >>\so^{\crr}_n @>>> 0.
\end{CD}
\end{equation}
So, actually,
$$\sss_n(r):=\kker(\sj_n^{<r>} \stackrel{1-\phi_r}{\longrightarrow}\so^{\crr}_n).$$

 The natural map $\omega: \sss^i_n(r)\to \sss_n(r)$ induced by the maps $p^{r+i}: \sj_n^{[r]}\to \sj_n^{<r>}$ and $\id : \so^{\crr}_n \to \so^{\crr}_n $ has kernel and cokernel  killed by $p^{r+i}$. 
 So does the map $\tau: \sss_n(r)\to \sss^i_n(r)$ induced by the maps $\id : \sj_n^{<r>}\to \sj_n^{[r]}$ and $p^{r+i} : \so^{\crr}_n \to \so^{\crr}_n $. We have $\tau\omega=\omega\tau=p^{r+i}$. 
 There are also natural maps $\omega^a:\sss^a_n(r)\to \sss_n^{a+1}(r)$ induced by $\id$ on $\sj_n^{[r]}$ and by multiplication by $p^{a+1}$ on $\so^{\crr}_n$. We set $\omega^{\prime}=\omega^0$. 
 Write $\omega_a: \sss^{a+1}_n(r)\to \sss^a_n(r)$ for the map induced by $\id$ on $\so_n^{\crr}$ and by multiplication by $p^{a+1}$ on $\sj_n^{[r]}$. We have $\omega^a\omega_a=\omega_a\omega^a=p^{a+1}$.

   We have versions of complexes $\sss_n(r)$ and $\sss^{i}_n(r)$  on the large syntomic sites \cite[4.3]{Ts1}. 
If it does not cause confusion, we will  write $\sss_n(r)$, $\sss^{i}_n(r)$ for all these complexes as well as  for $\R\ve_*\sss_n(r)$, $\R\ve_*\sss^{i}_n(r)$, respectively, 
where $\varepsilon: X_{n,\synt}\to X_{n,\eet}$ is the canonical projection to the \'etale site (or sometimes to the Nisnevich  site)

   For $X$ a fine and saturated log-smooth log-scheme over $\so_K$ and $0\leq r\leq p-2$, the natural map of complexes of sheaves on the \'etale site of $X_0$  
    $$
    \tau_{\leq r}\sss_n(r)\to \sss_n(r)
$$
is a quasi-isomorphism. 
For $X$ semistable over $\so_K$  and $r\geq 0$, the  natural map of complexes of sheaves on the \'etale site of $X_0$  
    $$
    \tau_{\leq r}\sss^{\prime}_n(r)\to \sss^{\prime}_n(r)
$$
is  a $p^{Nr}$-quasi-isomorphism for a universal constant $N$   \cite[Prop. 3.12]{CN}. 
\subsubsection{Syntomic cohomology and differential forms}
Let $X$ be a
syntomic  scheme over $W(k)$. Recall the
differential definition
\cite{K} of syntomic cohomology. Assume first
that we have an immersion $\iota:X\hookrightarrow Z$ over $W(k)$ such that
$Z$ is a smooth $W(k)$-scheme endowed with a compatible system of liftings of the Frobenius
$\{F_n:Z_n\to Z_n\}$. Let $D_n=D_{X_n}(Z_n)$ be the PD-envelope of $X_n$ in $Z_n$ (compatible with the canonical PD-structure on $pW_n(k)$) and
$J_{D_n}$ the ideal of $X_n$ in $D_n$. Set $J^{<r>}_{D_n}:=\{a\in J_{D_{n+s}}^{[r]}|\phi(a)\in p^r\so_{D_{n+s}}\}/p^n$ for some $s\geq r$.
For $0\leq r\leq p-1$, $J_{D_n}^{<r>}=J_{D_n}^{[r]}$.
 This definition is independent of $s$.
Consider the following
complexes of sheaves on $X_{\eet}$. 
\begin{align}
 \label{lifted}
S_n(r)_{X,Z}: & =\Cone(J_{D_n}^{<r-{\scriptscriptstyle\bullet}>}\otimes
\Omega\kr_{Z_n}\stackrel{ 1-\phi_r}{\longrightarrow} \so_{D_n}\otimes
\Omega\kr_{Z_n})[-1],\\
S^{i}_n(r)_{X,Z}: & =\Cone(J_{D_n}^{[r-{\scriptscriptstyle\bullet}]}\otimes
\Omega\kr_{Z_n}\verylomapr{p^{r+i} -p^i\phi}\so_{D_n}\otimes
\Omega\kr_{Z_n})[-1],\notag
\end{align}
where $\Omega\kr_{Z_n}:=\Omega\kr_{Z_n/W_n(k)}$ and $\phi_r $ is "`$\phi/p^r$"' (see \cite[2.1]{Ts} for details). The complexes
$S_n(r)_{X,Z}$, $S^i_n(r)_{X,Z}$ are, up to canonical quasi-isomorphisms, independent of
the choice of $\iota$ and $\{F_n\}$ (and we will omit the subscript $Z$ from the notation). Again, the natural maps $\omega: S^i_n(r)_X \to S_n(r)_X$ and 
$\tau: S_n(r)_X \to S^i_n(r)_X$ have kernels and cokernels annihilated by $p^{r+i}$.

  In general, immersions as above exist \'etale locally, and we define
$S_n(r)_X\in
{\mathbf D}^+(X_{\eet},{\mathbf Z}/p^n)$ by gluing the local complexes, and  $S_n(r)_{X_{\overline{\so}_K}}
\in {\mathbf D}^+((X_{\overline{\so}_K})_{\eet},{\mathbf Z}/p^n)$ as the
inductive limit of $S_n(r)_{X_{{\so_K}'}}$, where ${\so_K}'$ varies over the
integral closures of $\so_K$ in all finite extensions of $K$ in $\ovk$. Similarly, we define
 $S^i_n(r)_X$ and $S^i_n(r)_{X_{\overline{\so}_K}}$.

 Let now $X$ be a log-syntomic scheme over $W(k)$.
Using log-crystalline cohomology,  the above construction of syntomic complexes goes through almost
verbatim (see \cite[2.1]{Ts} for details) to yield
the logarithmic analogs $S_n(r)$ and  $S^{i}_n(r)$ on $X_{\eet}$. 
There are natural maps
$$
\varepsilon:  H^i(
X_{\eet},S_n(r))\to  H^i(
X_{\eet}^{\times},S_n(r)),\quad \varepsilon:  H^i(
X_{\eet},S^{i}_n(r))\to H^i(
X_{\eet}^{\times},S^{i}_n(r)),
$$
where, for clarity,  we wrote $X^{\times}$ for the log-scheme $X$ with its full log-structure. In this paper we are often interested in log-schemes coming from a regular syntomic scheme $X$ over $W(k)$ and a relative simple (i.e., with no self-intersections) normal crossing divisor $D$  on $X$. In such cases we will write $S_n(r)_X( D)$ and $S^{i}_n(r)_X(D)$ for the  syntomic complexes and use the Nisnevich topology instead of the \'etale one. We will write $H^*(X,S_n(*)_X( D))$ and $H^*(X,S^{i}_n(*)_X(D))$ for the corresponding cohomology groups. We will employ the same convention while talking about log-\'etale cohomology $H^*(X_K^{\times}, {\mathbf Z}/p^n(*))$: 
we will write $H^*(X_K(D_K), {\mathbf Z}/p^n(*))$ instead.
\subsubsection{Products}
  We need to discuss products. Assume that we are in the lifted situation (\ref{lifted}). Then we have a product structure
$$\cup: S^{i}_n(r)_{X,Z}\otimes S^{j}_n(r^{\prime})_{X,Z}\to S^{i+j}_n(r+r^{\prime})_{X,Z}, \quad r,r^{\prime},i,j\geq 0,
$$
defined by the following formulas
\begin{align*}
(x,y)\otimes (x^{\prime},y^{\prime}) & \mapsto (xx^{\prime}, (-1)^ap^{r+i}xy^{\prime}+yp^j\phi(x^{\prime}))\\
(x,y)\in S^{i}_n(r)_{X,Z}^a & =(J_{D_n}^{[r-a]}\otimes
\Omega^a_{Z_n})\oplus(\so_{D_n}\otimes
\Omega^{a-1}_{Z_n}),\\
(x^{\prime},y^{\prime})\in S^{j}_n(r^{\prime})_{X,Z}^b & = (J_{D_n}^{[r^{\prime}-b]}\otimes
\Omega^b_{Z_n})\oplus(\so_{D_n}\otimes
\Omega^{b-1}_{Z_n}).
\end{align*}
Globalizing, we obtain the product structure
$$\cup: S^{i}_n(r)_{X}\otimes ^{{\mathbb L}}S^{j}_n(r^{\prime})_{X}\to S^{i+j}_n(r+r^{\prime})_{X}, \quad r,r^{\prime},i,j\geq 0.
$$
This product is clearly compatible with the crystalline product.

   Similarly, we have the product structures
$$\cup: S_n(r)_{X,Z}\otimes S_n(r^{\prime})_{X,Z}\to S_n(r+r^{\prime})_{X,Z}, \quad r,r^{\prime}\geq 0,
$$
defined by the formulas
\begin{align*}
(x,y)\otimes (x^{\prime},y^{\prime}) & \mapsto (xx^{\prime}, (-1)^axy^{\prime}+y\phi_{r^{\prime}}(x^{\prime}))\\
(x,y)\in S_n(r)_{X,Z}^a & =(J_{D_n}^{<r-a>}\otimes
\Omega^a_{Z_n})\oplus(\so_{D_n}\otimes
\Omega^{a-1}_{Z_n}),\\
(x^{\prime},y^{\prime})\in S_n(r^{\prime})_{X,Z}^b & = (J_{D_n}^{<r^{\prime}-b>}\otimes
\Omega^b_{Z_n})\oplus(\so_{D_n}\otimes
\Omega^{b-1}_{Z_n}).
\end{align*}
Globalizing, we obtain the product structure
$$\cup: S_n(r)_{X}\otimes ^{{\mathbb L}}S_n(r^{\prime})_{X}\to S_n(r+r^{\prime})_{X}, \quad r,r^{\prime}\geq 0.
$$
This product is also clearly compatible with the crystalline product.
  
 The above product structures are compatible with the maps $\omega$ and the maps $\omega_0$.
On the other hand the maps $\tau$ are, in general, not compatible with products.
\subsubsection{Syntomic symbol maps}
   Let $X$ be a regular syntomic scheme over $W(k)$ with a divisor $D$ with relative simple normal crossings.
Recall that there are first Chern class maps defined by Kato and Tsuji \cite[2.2]{Ts1}
\begin{align*}
{c}_1^{\synt}:\quad & j_*\so^*_{X\setminus D}[-1]\to i_*j_*\so^*_{{(X\setminus D)}_{n+1}}[-1]\to
S_n(1)_X(D),\\
{c}_1^{\synt}:\quad & j_*\so^*_{X\setminus D}[-1]\to i_*j_*\so^*_{{(X\setminus D)}_{n}}[-1]\to
S^{\prime}_n(1)_X(D),
\end{align*}
that are compatible, i.e., the following diagram commutes
$$
\xymatrix{
j_*\so^*_{X\setminus D}[-1] \ar[d]^{p{c}_1^{\synt}} \ar[r]^{{c}_1^{\synt}}  & S^{\prime}_n(1)_X(D)\ar[dl]^{\omega}\\
S_n(1)_X(D) &\\
}
$$
Here $j:X\setminus D\hookrightarrow X$ is the natural immersion. 
In the lifted situation these classes are defined in the following way. Let $C_n$ be the complex
$$
(1+J_{D_n}\to M^{\gp}_{D_n})\simeq j_*\so^*_{(X\setminus D)_n}[-1],
$$
where, for a log-scheme $X$, $M_X$ denotes its log-structure. 
 The Chern class maps
\begin{equation}
 \label{symbol}
{c}_1^{\synt}:j_*\so^*_{{(X\setminus D)}_{n}}[-1]\to
S^{\prime}_n(1)_X(D), \quad
{c}_1^{\synt}: j_*\so^*_{{(X\setminus D)}_{n+1}}[-1]\to
S_n(1)_X(D),
\end{equation}
 are defined by the morphisms of complexes $$C_n\to S^{\prime}_n(1)_{X,Z},\quad
C_{n+1}\to S_n(1)_{X,Z}$$ given by the formulas
\begin{align*}
1+J_{D_n}\to (S^{\prime}_n(1)_{X,Z})^0 & =J_{D_n};\quad a\mapsto  \log a;\\
1+J_{D_{n+1}}\to (S_n(1)_{X,Z})^0 & =J_{D_n};\quad a\mapsto  \log a \mod p^n;
\end{align*}
and
\begin{align*}
 M^{\gp}_{D_n}\to (S^{\prime}_n(1)_{X,Z})^1 & =(\so_{D_n}\otimes \Omega^1_{Z_n})\oplus \so_{D_n};
\quad b\mapsto (d \log b,  \log (b^p\phi_{D_n}(b)^{-1}));\\
M^{\gp}_{D_{n+1}}\to (S_n(1)_{X,Z})^1 & =(\so_{D_n}\otimes \Omega^1_{Z_n})\oplus \so_{D_n};
\quad b\mapsto (d\log  b\mod p^n, p^{-1} \log(b^p\phi_{D_{n+1}}(b)^{-1})).\\
\end{align*}
\subsubsection{Syntomic cohomology and $p$-adic nearby cycles}
For log-schemes over $\so_K^{\times}$, in a stable range, syntomic cohomology tends to  compute (via the period morphism) $p$-adic nearby cycles . We will recall the relevant theorems.

 Let $X$ be a log-syntomic scheme over $W(k)$. Let $i:X_{0,\eet}\to X_{\eet}$ and $j:X_{\tr,K,\eet}\to X_{\eet}$ be the natural maps. 
Here $X_{\tr}$ is the open set of $X$ where the log-structure is trivial. For $0\leq r\leq p-2$, there is a natural homomorphism on the \'etale site of $X_n$ \cite{FM} (the Fontaine-Messing period map)
$$
\alpha_{r}: \sss_n(r) \rightarrow i^*\R j_*{\mathbf Z}/p^n(r)
$$  
from syntomic complexes to $p$-adic nearby cycles.
It factors through $\tau_{\leq r} i^*\R j_*{\mathbf Z}/p^n(r).$
One checks that $\alpha_r$ is compatible with products.
Similarly, for any $r\geq 0$, we get a natural map \cite{FM}
$$
\tilde{\alpha}_{r}: \sss_n(r) \rightarrow i^*\R j_*{\mathbf Z}/p^n(r)'.
$$
 Composing with the map 
$\omega: \sss'_n(r)\to \sss_n(r)$ we get a natural, compatible with products,  
 morphism
$$
\alpha_{r}: \sss^{\prime}_n(r) \rightarrow i^*\R j_*{\mathbf Z}/p^n(r)'.
$$
\begin{theorem}(\cite[Theorem 5.1]{Ts1})
\label{input0}
 For $i\leq r\leq p-2$ and  for  a fine and saturated log-scheme $X$ log-smooth over  $\so_K^{\times}$  the period map
\begin{equation}
\label{period2}\alpha_{r}:\quad  \sss_n(r)_{X} \stackrel{\sim}{\rightarrow} \tau_{\leq r}i^*\R j_*{\mathbf Z}/p^n(r)_{X_{\tr}}.
\end{equation}
is an isomorphism. 
\end{theorem}
\begin{theorem}(\cite[Theorem 1.1]{CN})
\label{input1}
 For   $0\leq i\leq r$ and for a semistable scheme $X$ over $\so_K$,   consider the period map 
\begin{equation}
\label{maineq1}
\alpha_{r}:\quad  \sh^i(\sss^{\prime}_n(r)_{X}) \rightarrow i^*\R^ij_*{\mathbf Z}/p^n(r)'_{X_{\tr}}.
\end{equation}
If $K$ has enough roots of unity\footnote{See Section (2.1.1) of \cite{CN} for what it means for a field to contain enough roots of unity. The field $F$ contains enough roots of unity and for any $K$, the field $K(\zeta_{p^n})$, for $n\geq c(K)+3$, where $c(K)$ is  the conductor of $K$, contains enough roots of unity.}   then the kernel  and cokernel of this map are annihilated by $p^{Nr}$ for a universal constant $N$ depending only on $p$  (and $d$ if $p=2$) {\rm (not depending on  $X$, $K$, $n$ or $r$)}.
In general, the kernel  and cokernel of this map are annihilated by $p^{Nr}$ for an integer $N=N(p,e)$, which depends on $e$ and $p$ (also $d$ if $p=2$) but not on $X$ or $n$.
\end{theorem}
\subsection{Syntomic-\'etale cohomology}
We will now recall the definition and basic properties of
syntomic-\'etale cohomology. More details can be found in \cite{FM},\cite{EN}. Let $X$ be a log-scheme, log-syntomic  over $\Spec(W(k))$. Let  $j^{\prime}: X_{\tr,K}\to X_{K}$ denote the natural open immersion.
\subsubsection{Syntomic-\'etale cohomology}
Denote by $\se_n(r)$ and $\se^{\prime}_n(r)$  the syntomic-\'etale complexes on $X_{\eet}$ \cite[2.2.2]{EN} associated to $\sss_n(r)$ and $\sss^{\prime}_n(r)$, respectively.  They are 
obtained by gluing the complexes of sheaves $\sss_n(r) $ and $\sss'_n(r) $ and the complexes of sheaves $j^{\prime}_*G{\mathbf Z}/p^n(r)^{\prime}$,  where $G$ denotes the Godement resolution of a sheaf (or a complex of sheaves),   by the maps $\tilde{\alpha}_{r}$ and $\alpha_r$. We have the distinguished triangles
\begin{equation}
\label{Niis}
j_{\eet!} \R j_*^{\prime}{\mathbf Z}/p^n(r)^{\prime}\to \se_n(r)\to i_*\sss_n(r),\quad j_{\eet !}  \R j_*^{\prime}{\mathbf Z}/p^n(r)^{\prime}\to \se^{\prime}_n(r)\to i_*\sss^{\prime}_n(r),
\end{equation}
as well as the natural maps
$$\tilde{\alpha}_r: \se_n(r)\to \R j_*{\mathbf Z}/p^n(r)^{\prime},\quad \alpha_r: \se_n(r)^{\prime}\to \R j_*{\mathbf Z}/p^n(r)^{\prime}
$$
compatible with the maps $\tilde{\alpha}_r$ and $\alpha_r$ from syntomic complexes. For $a\geq 0$, we have the truncated version of the above - the distinguished triangles
\begin{equation}
\label{Niis1}
j_{\eet!} \tau_{\leq a}\R j_*^{\prime}{\mathbf Z}/p^n(r)^{\prime}\to \tau_{\leq a}\se_n(r)\to i_*\tau_{\leq a}\sss_n(r),\quad j_{\eet !} \tau_{\leq a} \R j_*^{\prime}{\mathbf Z}/p^n(r)^{\prime}\to \tau_{\leq a}\se^{\prime}_n(r)\to i_*\tau_{\leq a}\sss^{\prime}_n(r).
\end{equation}
\subsubsection{Syntomic-\'etale cohomology and \'etale cohomology of the generic fiber}
For a log-scheme over $\so_K^{\times}$, in a stable range, syntomic-\'etale cohomology tends to compute \'etale cohomology of the generic fiber.
\begin{theorem}(\cite[Theorem 2.5]{EN})
\label{keylemma11}
Let $X$ be a log-scheme log-smooth   over ${\so_K}^{\times}$. 
 Then
\begin{enumerate}
\item we have a natural quasi-isomorphism
$$
\tilde{\alpha}_r: \tau_{\leq r}\se_n(r) \simeq \tau_{\leq r}\R j_*{\mathbf Z}/p^n(r), \quad 0\leq r\leq p-2.
$$
\item if $X$ is semistable, there is a constant $N$ as in Theorem \ref{input1} and a natural morphism $$
{\alpha}_r: \se^{\prime}_n(r)\to  
\R j_*{\mathbf Z}/p^n(r)^{\prime},\quad r\geq 0,
$$
such that the induced map on cohomology sheaves in degrees $\leq r$ has kernel and cokernel annihilated by $p^{Nr}$. 
\end{enumerate}
\end{theorem}

  The above theorem implies that  the logarithmic syntomic-\'etale cohomology is close to  the logarithmic syntomic-\'etale cohomology of the complement of the divisor at infinity. 
  \begin{corollary}(\cite[Cor. 2.6]{EN})
  \label{reduction-kolo}
  Let $X$ be a semistable scheme  over ${\so_K}$ with a divisor at infinity $D_{\infty}$. We treat it is as a log-scheme over $\so_K^{\times}$. Let $Y:=X\setminus D_{\infty}$ and let $ j_1: Y\hookrightarrow X$. 
 \begin{enumerate}
\item we have a natural quasi-isomorphism
$$
\tilde{\alpha}_r: \tau_{\leq r}\se_n(r)_X \stackrel{\sim}{\to} \tau_{\leq r}\R j_{1*}\se_n(r)_Y , \quad 0\leq r\leq p-2.
$$
\item there is a constant $N$ as in Theorem \ref{input1} and a natural morphism $$
{\alpha}_r: \se^{\prime}_n(r)_X\to \R j_{1*}\se^{\prime}_n(r)_Y  ,\quad r\geq 0,
$$
such that the induced map on cohomology sheaves in degrees $\leq r$ has kernel and cokernel annihilated by $p^{Nr}$. 
\end{enumerate}
  \end{corollary}
\subsubsection{Nisnevich syntomic-\'etale cohomology}
    We will pass now to the Nisnevich topos of $X$. Denote by $\varepsilon: X_{\eet}\to X_{\Nis}$ the natural projection. For  $r\geq 0$, by applying $\R \varepsilon_*$ to the \'etale period map above and using that $\R \varepsilon_*i^*=i^*\R \varepsilon_*$\footnote{This equality fails for the projection to Zariski topology and is the reason we use Nisnevich topology instead of Zariski.} (cf. \cite[2.2.b]{GD}),   we obtain     a natural map
$$
\tilde{\alpha}_{r}: \R \varepsilon_*\sss_n(r) \rightarrow i^*\R j_*\R \varepsilon_*{\mathbf Z}/p^n(r)'.
$$
 Composing with the map 
$\omega: \R \varepsilon_*\sss'_n(r)\to \R \varepsilon_*\sss_n(r)$ we get a natural, compatible with products,  
 morphism
$$
\alpha_{r}: \R \varepsilon_*\sss^{\prime}_n(r) \rightarrow i^*\R j_*\R \varepsilon_*{\mathbf Z}/p^n(r)'.
$$
Write, for simplicity,  $\sss_n(r)$ and $\sss^{\prime}_n(r)$  for the derived pushforwards of $\sss_n(r)$ and $\sss^{\prime}_n(r)$ from $X_{\eet}$ to $X_{\Nis}$. Same for $\se_n(r)$ and $\se^{\prime}_n(r)$.  
Notice that the latter are quasi-isomorphic to 
the complexes obtained by gluing the complexes of sheaves $\sss_n(r) $ and $\sss'_n(r) $ on $X_{1,\Nis}$ and the complexes of sheaves $\varepsilon_*j^{\prime}_*G{\mathbf Z}/p^n(r)^{\prime}$  on $X_{K,\Nis}$ by the maps $\tilde{\alpha}_{r}$ and $\alpha_r$. Hence we have the distinguished triangles
\begin{equation}
\label{Niis2}
j_{\Nis!} \R j_*^{\prime}\R \varepsilon_*{\mathbf Z}/p^n(r)^{\prime}\to \se_n(r)\to i_*\sss_n(r),\quad j_{\Nis !}  \R j_*^{\prime}\R \varepsilon_*{\mathbf Z}/p^n(r)^{\prime}\to \se^{\prime}_n(r)\to i_*\sss^{\prime}_n(r),
\end{equation}
as well as the natural maps
$$\tilde{\alpha}_r: \se_n(r)\to \R j_*\R \varepsilon_*{\mathbf Z}/p^n(r)^{\prime},\quad \alpha_r: \se_n(r)^{\prime}\to \R j_*\R \varepsilon_*{\mathbf Z}/p^n(r)^{\prime}
$$
compatible with the maps $\tilde{\alpha}_r$ and $\alpha_r$ from syntomic complexes. For $a\geq 0$, we have the truncated version of the above - the distinguished triangles
\begin{equation}
\label{Niis3}
j_{\Nis !} \tau_{\leq a}\R j_*^{\prime}\R \varepsilon_*{\mathbf Z}/p^n(r)^{\prime}\to \tau_{\leq a}\se_n(r)\to i_*\tau_{\leq a}\sss_n(r),\quad j_{\eet !} \tau_{\leq a} \R j_*^{\prime}\R \varepsilon_*{\mathbf Z}/p^n(r)^{\prime}\to \tau_{\leq a}\se^{\prime}_n(r)\to i_*\tau_{\leq a}\sss^{\prime}_n(r).
\end{equation}

Define the following complexes of sheaves on $X_{\Nis}$
\begin{align*}
\sss_n(r)_{\Nis}:=\tau_{\leq r}\sss_n(r), & \quad  \sss^{\prime}_n(r)_{\Nis}:=\tau_{\leq r}\sss^{\prime}_n(r);\\
\se_n(r)_{\Nis}:=\tau_{\leq r}\se_n(r), & \quad  \se^{\prime}_n(r)_{\Nis}:=\tau_{\leq r}\se^{\prime}_n(r).
\end{align*}

  We will need the following twisted version of the complexes $\se^{\prime}_n(r)$ and $\se^{\prime}_n(r)_{\Nis}$. Set the gluing map (on $X_{\eet}$)
$$
\alpha_{r}: \sss^{1}_n(r) \rightarrow i^*\R j_*{\mathbf Z}/p^n(r+1)^{\prime}(-1)
$$
to be equal to the composition
$$\sss^{1}_n(r)\stackrel{\omega_0}{\to} \sss^{\prime}_n(r)\stackrel{\alpha_r}{\to} i^*\R j_*{\mathbf Z}/p^n(r)^{\prime}\rightarrow i^*\R j_*{\mathbf Z}/p^n(r+1)^{\prime}(-1).
$$
And define the complexes $\se^{1}_n(r)$ by gluing $\sss^{1}_n(r)$ and $j_{\eet!} j_*^{\prime}G{\mathbf Z}/p^n(r+1)^{\prime}(-1)$ via $\alpha_r$. We have the distinguished triangle
$$j_{\eet!} \R j_*^{\prime}{\mathbf Z}/p^n(r+1)^{\prime}(-1)\to \se^1_n(r)\to i_*\sss^1_n(r).
$$
Write $\se^{1}_n(r)$ for the derived pushforward of $\se^{1}_n(r)$ to $X_{\Nis}$ and set $\se^{1}_n(r)_{\Nis}:=\tau_{\leq r}\se^{1}_n(r)$. 

\subsubsection{Syntomic-\'etale cohomology and differential forms}
We will need a differential definition of the syntomic-\'etale complexes. Let $X$ be a log-syntomic, finite and saturated (fs for short)  scheme over $W(k)$. Assume first that $X$ is affine and we have an immersion 
 $X\hookrightarrow Z$ over $W(k)$ such that $Z$ is a log-smooth $W(k)$-scheme endowed with a compatible system of liftings of the Frobenius $\{F_n: Z_n\to Z_n\}$. 
Choose sufficiently large algebraically closed fields $\Omega$ and $\Omega^{\prime}$ of characteristic zero and $p$, respectively. Let $C$ be the set of all isomorphism classes of fs monoids
 $P$ such that $P^*=\{1\}$. For each isomorphism class $c\in C$ choose a representative $P_c$ of $c$ and define the log-geometric point $\Omega_c$ to be $\Spec(\Omega_c)$ with 
 $M_{\Omega_c}=\Omega\oplus_{n\in {\mathbb N}, n\neq 0}1/nP_c$ and $\Omega^{\prime}_c$ to be $\Spec(\Omega^{\prime})$ with $M_{\Omega_c^{\prime}}=\Omega\oplus_{n\in {\mathbb N}, p\nmid n}1/nP_c$.
  Let $G$ denote the Godement resolution with respect to all log-geometric points whose sources are $\Omega_c$ or $\Omega_c^{\prime}$ for some $c\in C$. The set of points we have chosen is actually larger than what would suffice here but it will turn out useful later.
 
    The map $$\alpha^{\prime}_r: \sss_n(r)\to i_*i^*\R j_*{\mathbf Z}/p^n(r)^{\prime}\stackrel{\delta}{\to}j_{\eet !}\R j^{\prime}_*{\mathbf Z}/p^n(r)^{\prime}[1],$$ 
 where the map $\delta$ is obtained from the distinguished triangle
$$j_{\eet !}\R j^{\prime}_*{\mathbf Z}/p^n(r)^{\prime}\stackrel{\eta}{\to} \R j_*{\mathbf Z}/p^n(r)^{\prime}\to i_*i^*\R j_*{\mathbf Z}/p^n(r)^{\prime},
$$  
 can be induced by a map
 $\alpha^{\prime}_r: S_n(r)\to j_{\eet !}j^{\prime}_*G{\mathbf Z}/p^n(r)^{\prime}[1] $ that, in turn, can be obtained 
by shififying the following morphisms of presheaves on $X_{\eet}$ \cite[3.1]{Ts} (we write $\theta$ for the operation $i_*i^*j_*G$ and assume that $U=\Spec(A)\to X$ is a strict \'etale map)
\begin{align*}
\alpha_{r,U}^{\prime}:\quad \Gamma(U,S_n(r)_{X,Z})\to \Gamma(U^h,\theta_U & \overline{S}_n(r)_{U,Z})\stackrel{\sim}{\leftarrow}\Gamma(U^h,\theta_U\Lambda_{U^h})\stackrel{\sim}{\leftarrow}\Gamma(U,\theta\Lambda_U)\\
 & \stackrel{\sim}{\leftarrow}\Gamma(U,\Cone(\eta(G\Lambda_U)))\to  \Gamma(U,j_{\eet !}j^{\prime}_*G\Lambda_U[1]),
\end{align*}
where we put $\Lambda_Y={\mathbf Z}/p^n(r)^{\prime}_{Y_{K,\tr}}$ and $\eta(G\Lambda_U):j_{\eet !}j^{\prime}_*G\Lambda_U\hookrightarrow j_{\eet *}j^{\prime}_*G\Lambda_U$ is the natural  injection. Here $U^h$ denotes the henselization of $U$ with respect to the ideal $pU$ and $\overline{S}_n(r)_{U,Z}$ is an analog of the syntomic complex $S_n(r)_{U,Z}$ \cite[3.1]{Ts}. 
This complex is defined using, instead of $A$, $\overline{A^h}$ - the integral closure of $A^h$ in an algebraic closure of the fraction field $\Frac(A^h)$ that is \'etale over $A^h_{\tr,K}$, and instead of the immersion $X\hookrightarrow Z$, the immersion 
$\Spec(\overline{A^h})\hookrightarrow \Spec(\A_{\crr}(\overline{A^h}))$ with the natural Frobenius liftings and log-structures. Unlike in \cite[3.1]{Ts} we allow here trivial log-structure on the special fiber.
The complex $\overline{S}_n(r)_{U,Z}$ is a complex of locally free sheaves on $(U^h_{\tr,K})_{\eet}$. 
 It is a resolution of $\Lambda_{U^h_{\tr,K}}$.
 
   In the case of a general $X$ we choose a strict \'etale affine covering $X^{\prime}\to X$ and immersion $X^{\prime}\hookrightarrow Z^{\prime}$ with Frobenius liftings $F_{Z^{\prime}_n}:Z^{\prime}_n\to Z^{\prime}_n$.
    From this we construct a (\v{C}ech) hypercovering $X^{{\scriptscriptstyle\bullet}}\to X$ with immersion $X^{{\scriptscriptstyle\bullet}}\hookrightarrow Z^{{\scriptscriptstyle\bullet}}$ and Frobenius liftings $F_{Z^{{\scriptscriptstyle\bullet}}_n}:Z^{{\scriptscriptstyle\bullet}}_n\to Z^{{\scriptscriptstyle\bullet}}_n$. By applying the above construction to each level we obtain a morphism  $\alpha_{r}^{\prime}: S_n(r)_{X^{{\scriptscriptstyle\bullet}},Z^{{\scriptscriptstyle\bullet}}}\to j_{\eet !}j^{\prime}_*G\Lambda_{X^{{\scriptscriptstyle\bullet}}}[1]$. Our map $\alpha_r^{\prime}$ is now defined as $\R \varepsilon_*\alpha_r^{\prime}$, where $\varepsilon: X^{{\scriptscriptstyle\bullet}}_{\eet}\to X_{\eet}$ is the change of topoi map. It can be represented by the composition
$$\alpha_r^{\prime}:\quad  S_n(r)_X\to \varepsilon_*GS_n(r)_{X^{{\scriptscriptstyle\bullet}},Z^{{\scriptscriptstyle\bullet}}}\stackrel{\alpha_{r}^{\prime}}{\to}\varepsilon_*Gj_{\eet !}j^{\prime}_*G\Lambda_{X^{{\scriptscriptstyle\bullet}}}[1]\stackrel{\sim}{\leftarrow}j_{\eet !}j^{\prime}_*G\Lambda_{X}[1].
$$
 
 Consider the induced map 
$$\alpha^{\prime}_r: \quad S^{\prime}_n(r)_X\stackrel{\omega}{\to} S_n(r)_X\stackrel{\alpha_r^{\prime}}{\to} j_{\eet!}j^{\prime}_*G\Lambda_X[1].
$$ 
This map, a priori a zigzag of maps of complexes,  can be "straighten up", that is, we can find a complex $\wt{S}^{\prime}_n(r)_X$ and genuine maps of complexes $f:\wt{S}^{\prime}_n(r)_X\to S^{\prime}_n(r)_X$ and $\tilde{\alpha}_r:\wt{S}^{\prime}_n(r)_X\to j_{\eet!}j^{\prime}_*G\Lambda_X[1]$ that fit
    into the following commutative diagram
$$
\xymatrix{
\wt{S}^{\prime}_n(r)_X\ar[d]^{f}_{\wr}\ar[rd]^{\tilde{\alpha}_r} \\ 
S^{\prime}_n(r)_X\ar[r]^-{\alpha^{\prime}_r} &
 j_{\eet!}j^{\prime}_*G\Lambda_X[1]
}
$$
This is done by replacing each diagram of maps $A\stackrel{a}{\to} B\stackrel{\sim}{\leftarrow} C$ by the map $a^{\prime}: D\to C$ from  the following commutative diagram of complexes
\begin{equation}
\label{zigzag}
\xymatrix{  D=\Cone(a-b)[-1]\ar[r]^-{a^{\prime}}\ar[d]^{b^{\prime}}_{\sim} & C\ar[d]^b_{\sim}\\
A\ar[r]^a 
 & B
}
\end{equation}

    Set 
  $$E^{\prime}_n(r)_X:=\Cone(\wt{S}^{\prime}_n(r)_X\stackrel{\tilde{\alpha}_r}{\to}j_{\eet !}j^{\prime}_*G\Lambda_X[1])[-1],\quad E^{\prime}_n(r)_{\Nis,X}:=\tau_{\leq r}\varepsilon^{\prime}_*GE^{\prime}_n(r)_X,
  $$
  where $\varepsilon^{\prime}:X_{\eet}\to X_{\Nis}$ is the change of topology map. 
  These complexes represent $\se^{\prime}_n(r)_X$ and $\se^{\prime}_n(r)_{\Nis,X}$, respectively. They are functorial in $X$. The complexes 
  $E^{\prime}_n(r)_X$ and $E^{\prime}_n(r)_{\Nis,X}$ inherit functorial product structure that is compatible with the maps
  $$E^{\prime}_n(r)_X\to \wt{S}^{\prime}_n(r)_X\to S^{\prime}_n(r)_X.
  $$
To construct it note that all the maps in $\alpha_{r,U}$ are compatible with products and so are  the maps $\omega$ and $\delta$. It suffices thus, in the diagram \ref{zigzag}, to equip
the cone $D$ with product structure that is compatible with the projections on $A$ and $C$. But this is standard  and can be done as in \cite[3.1]{N0}.

  The above syntomic-\'etale complexes have  twisted versions. Consider first
\begin{align*}
E^{1}_n(r)_X: =\Cone(\wt{S}^{1}_n(r)_X\stackrel{\tilde{\omega}_{0}}{\to}\wt{S}^{\prime}_n(r)_X\stackrel{\tilde{\alpha}^{\prime}_r}{\to}j_{\eet !}j^{\prime}_*G{\mathbf Z}/p^n(r+1)^{\prime}(-1)[1])[-1],\quad E^{1}_n(r)_{\Nis,X}: =\tau_{\leq r}\varepsilon^{\prime}_*GE^{1}_n(r)_X,
\end{align*}
where $$
\tilde{\alpha}^{\prime}_r:\wt{S}^{\prime}_n(r)_X\stackrel{\tilde{\alpha}_r}{\to}j_{\eet !}j^{\prime}_*G{\mathbf Z}/p^n(r)^{\prime}[1]{\to}j_{\eet !}j^{\prime}_*G{\mathbf Z}/p^n(r+1)^{\prime}(-1)[1]
$$ 
and the complex $\wt{S}^{1}_n(r)_X$ is defined via the following commutative diagram
$$
\xymatrix{ \wt{S}^{1}_n(r)_X=\Cone(f-\omega_{0})[-1]\ar[d]^{\tilde{\omega}_{0}}\ar[r]^-{f^1}_-{\sim} & S^{1}_n(r)_X\ar[d]^{\omega_{0}}\\
\wt{S}^{\prime}_n(r)_X\ar[r]^-{f}_{\sim} 
& S^{\prime}_n(r)_X 
}
$$
We also the following twisted version
\begin{align*}
E^{2}_n(r)_X: =\Cone(\wt{S}^{2}_n(r)_X\stackrel{\tilde{\omega}_{1,0}}{\to}\wt{S}^{\prime}_n(r)_X\stackrel{\tilde{\alpha}^{\prime}_r}{\to}j_{\eet !}j^{\prime}_*G{\mathbf Z}/p^n(r+2)^{\prime}(-2)[1])[-1],
\end{align*}
where $$
\tilde{\alpha}^{\prime}_r:\wt{S}^{\prime}_n(r)_X\stackrel{\tilde{\alpha}_r}{\to}j_{\eet !}j^{\prime}_*G{\mathbf Z}/p^n(r)^{\prime}[1]{\to}j_{\eet !}j^{\prime}_*G{\mathbf Z}/p^n(r+2)^{\prime}(-2)[1],
$$ 
$\omega_{1,0}=\omega_1\omega_0$,
and the complex $\wt{S}^{2}_n(r)_X$ is defined via the following commutative diagram
$$
\xymatrix{ \wt{S}^{2}_n(r)_X=\Cone(f-\omega_{1,0})[-1]\ar[d]^{\tilde{\omega}_{1,0}}\ar[r]^-{f^1}_-{\sim} & S^{2}_n(r)_X\ar[d]^{\omega_{1,0}}\\
\wt{S}^{\prime}_n(r)_X\ar[r]^-{f}_-{\sim} 
& S^{\prime}_n(r)_X 
}
$$
Using again the construction of products on cones from \cite[3.1]{N0} and the fact that the maps 
$$
\omega:S^{\prime}_n(r)_X\to S_n(r)_X,\quad 
\omega_0:S^{1}_n(r_1)_X\to S^{\prime}_n(r)_X
$$
as well as the map $\alpha_r$ commute with products, we construct products
$$\cup: E^{\prime}_n(r_1)_X\otimes ^{{\mathbb L}}E^{1}_n(r_2)_X\to E^{1}_n(r_1+r_2)_X;\quad \cup: E^{\prime}_n(r_1)_{\Nis,X}\otimes ^{{\mathbb L}}E^{1}_n(r_2)_{\Nis,X}\to E^{1}_n(r_1+r_2)_{\Nis,X}
$$
that are compatible with the product $\cup: S^{\prime}_n(r_1)_X\otimes ^{{\mathbb L}}S^{1}_n(r_2)_X\to S^{1}_n(r_1+r_2)_X$.

 In an analogous way we define the complexes $E_n(r)_X$ and $E_n(r)_{\Nis,X}$ and the corresponding products. 
\subsubsection{Syntomic-\'etale symbol maps}
  If $X$ is fs then we  have Chern class maps
  \begin{equation}
 \label{symbol1}
{c}_1^{\synt}:j_*\so^*_{X\setminus D}[-1]\to
E^{\prime}_n(1)_{\Nis,X}(D), \quad
{c}_1^{\synt}: j_*\so^*_{X\setminus D}[-1]\to
E_n(1)_{\Nis,X}(D)
\end{equation}
that are compatible with the syntomic Chern class maps. 
To define these maps it suffices, by degree reason, to define maps
\begin{equation}
{c}_1^{\synt}:j_*\so^*_{X\setminus D}[-1]\to
E^{\prime}_n(1)_X(D), \quad
{c}_1^{\synt}: j_*\so^*_{X\setminus D}[-1]\to
E_n(1)_X(D).
\end{equation}
For that recall that the syntomic Chern classes are compatible with the \'etale Chern classes \cite[Prop. 3.2.4]{Ts}, i.e., that the following diagram commutes
$$
\xymatrix{j_*\so^*_{X\setminus D}[-1]\ar[r]^{c_1^{\synt}}\ar[dr]^{pc_1}&
S^{\prime}_n(1)_X(D)\ar[d]^{\alpha_1}\\
& i_*i^*\R j_{\eet *}\R j^{\prime}_*{\mathbf Z}/p^n(1)^{\prime}
}
$$
where the map $c_1$ is induced from the Chern class map 
$$c_1^{\eet}: j_*\so^*_{X\setminus D}[-1]\to
\R j_{\eet *}\R j^{\prime}_*{\mathbf Z}/p^n(1).
$$
It follows that the composition
$$j_*\so^*_{X\setminus D}[-1]\stackrel{c_1^{\synt}}{\to}
S^{\prime}_n(1)_X(D)\stackrel{\alpha_1}{\to}
i_*i^*\R j_{\eet *}\R j^{\prime}_*{\mathbf Z}/p^n(1)^{\prime}\to
j_{\eet !}\R j^{\prime}_*{\mathbf Z}/p^n(1)^{\prime}[1]
$$ is trivial, hence there exists a unique (for degree reason) map
$$ c_1^{\synt}: j_*\so^*_{X\setminus D}[-1]\to
E^{\prime}_n(1)_X(D)
$$ compatible with the syntomic Chern class map. The construction for the complex $E_n(1)_X(D)$ is analogous.

\subsection{Syntomic cohomology and motivic cohomology}
 Let $X$ be a smooth scheme  over $\so_K$. Let  ${\mathbf Z}(r)_{\M}$ denote  the complex of motivic sheaves
${\mathbf Z}(r)_{\M}:=X\mapsto z^r(X,2r-*) $ in the Nisnevich topology of $X$, where the complex $z^r(X,*)$ is the Bloch's cycles complex \cite{Bl}. Let ${\mathbf Z}/p^n(r)_{\M}:={\mathbf Z}(r)_{\M}\otimes\Z/p^n$.
 We have that
$H^j(X_{\Nis},{\mathbf Z}/p^n(i)_{\M}) =H^j\Gamma(X,\Z/p^n(r)_{\M})$ is the 
Bloch higher Chow group \cite[Prop. 3.6]{GD}.  

   Recall that the main theorem of \cite{EN} shows that  syntomic-\'etale complexes on smooth schemes over $\so_K$ approximate motivic complexes. 
\begin{theorem}(\cite[Theorem 3.10]{EN})
\label{keylemma10}
Let $X$ be a semistable scheme  over $\so_K$ with a smooth special fiber.  
Let $j^{\prime}: X_{\tr}\hookrightarrow X$ be the natural open immersion. Then
\begin{enumerate}
\item there is a natural cycle class map
$$
\cl^{\synt}_r: \R j^{\prime}_*{\mathbf Z}/p^n(r)_{\M}\to \se_n(r)_{\Nis},\quad 0\leq r\leq p-2.
$$
It is a quasi-isomorphism.
\item there is a   natural cycle class map 
$$\cl^{\synt}_r:\R j^{\prime}_*{\mathbf Z}/p^n(r)_{\M}\to \se^{\prime}_n(r)_{\Nis},\quad  r\geq 0.
$$
It is a  $p^{Nr}$-quasi-isomorphism for  a constant $N$ as in Theorem \ref{input1}.
\end{enumerate}
We have analogous statements in the \'etale topology. These cycle class maps are compatible (via the localization map and the period map) with the \'etale cycle class maps.
\end{theorem}

\section{Cohomology of classifying spaces}
In this section we study cohomology of classifying spaces: we compute  their (Nisnevich) syntomic cohomology and we show that their Nisnevich syntomic-\'etale  and  syntomic cohomologies agree  in even degrees. 
\subsection{Classical computations}
\label{classical}
  We are interested in 
 cohomology of $BGL_a$ and the related classifying simplicial  schemes $B(GL_a\times GL_b)$ and $BGL(a,b)$. Recall that, roughly speaking, 
 if a ${\mathbf Z}\times{\mathbf Z}$-graded cohomology theory $X/S\mapsto H^*(X,*)$ satisfies homotopy property and projective space theorem (of bidegrees $(2i,i))$,  then we have the following isomorphisms
 \begin{align*}
 H^*(BGL_{a},*) & \simeq H^*(S,*) [x_1,\ldots,x_a],\\
 H^*(BGL(a,b),*) & \simeq H^*(BGL_{a}\times BGL_{b},*)
  \simeq  H^*(S,*)[x_1,\ldots,x_{a};y_1,\ldots,y_{b}], 
 \end{align*}
 where the classes
$x_i,y_i\in H^{2i}(BGL_{a},i)$ are the Chern classes of the universal locally free sheaf on $BGL_{a}
$ (defined via a projective space theorem).  This applies, for example,  to
\begin{enumerate}
\item 
 \'etale cohomology with $\ell$-adic coefficients;
 \item    (\'etale)  logarithmic de Rham-Witt cohomology $H^*(X,W_n\Omega^*_{X,\log}[-*])$, for a smooth scheme $X$ over $k$  \cite[III, Theorem 2.2.5]{Gr}:
    \begin{align*}
 H^*(BGL_{a},W_n\Omega^*_{BGL_{a}, \log}[-*]) & \simeq H^*(k,W_n\Omega^*_{k,\log}[-*]) [x_1,\ldots,x_a],\\
 H^*(BGL(a,b),W_n\Omega^*_{BGL(a,b), \log}[-*]) & \simeq H^*(BGL_{a}\times BGL_{b},W_n\Omega^*_{BGL_{a}\times BGL_{b}, \log}[-*])\\
  &  \simeq  H^*(k,W_n\Omega^*_{k,\log}[-*])[x_1,\ldots,x_{a};y_1,\ldots,y_{b}], 
 \end{align*}
where the classes
$x_i,y_i\in H^{2i}(BGL_{l},W_n\Omega^i_{BGL_{l}, \log}[-i])$ are the logarithmic de Rham-Witt Chern classes of the universal locally free sheaf on $BGL_{a}$.
We have $H^*(k,W_n\Omega^*_{k,\log}[-*])=H^*_{\eet}(k,{\mathbf Z}/p^n)$, which is trivial in degrees at least $2$. 
\item motivic cohomology ${\mathbf Z}/p(*)_{\M}$ over a field $F$ \cite[Lemma 7]{PO}:
  \begin{align*}
  H^*(BGL_{a},{\mathbf Z}/p(*)_{\M})& \simeq H^*(F,{\mathbf Z}/p(*)_{\M})[x_1,\ldots,x_a]\\
  H^*(BGL(a,b),{\mathbf Z}/p(*)_{\M}) & \simeq H^*(BGL_{a}\times BGL_{b},{\mathbf Z}/p(*)_{\M})\\
    & \simeq  H^*(F,{\mathbf Z}/p(*)_{\M})[x_1,\ldots,x_{a};y_1,\ldots,y_{b}], 
   \end{align*}
  where  $ x_i,y_i\in H^{2i}(BGL_{a},{\mathbf Z}/p(i)_{\M})$ are the motivic  Chern classes of the universal locally free sheaf on $BGL_{a}$.
\item   
   Nisnevich logarithmic de Rham-Witt cohomology $H^*_{\Nis}(X,W_n\Omega^*_{X,\log}[-*])$, for smooth schemes $X$ over $k$:
    \begin{align}
    \label{vanish0}
 H^*_{\Nis}(BGL_{a},W_n\Omega^*_{BGL_{a}, \log}[-*]) & \simeq H^*_{\Nis}(k,W_n\Omega^*_{k,\log}[-*]) [x_1,\ldots,x_a],\\
 H^*_{\Nis}(BGL(a,b),W_n\Omega^*_{BGL(a,b), \log}[-*]) & \simeq H^*_{\Nis}(BGL_{a}\times BGL_{b},W_n\Omega^*_{BGL_{a}\times BGL_{b}, \log}[-*])\notag\\
  &  \simeq  H^*_{\Nis}(k,W_n\Omega^*_{k,\log}[-*])[x_1,\ldots,x_{a};y_1,\ldots,y_{b}], \notag
 \end{align}
where 
$x_i,y_i\in H^{2i}_{\Nis}(BGL_{l},W_n\Omega^i_{BGL_{l}, \log}[-i])$ are the Nisnevich logarithmic de Rham-Witt Chern classes of the universal locally free sheaf on $BGL_{a}$. This follows from the quasi-isomorphism ${\mathbf Z}/p(i)_{\M}\simeq W_n\Omega^i_{X, \log}[-i]$. We have $H^*_{\Nis}(k,W_n\Omega^*_{k,\log}[-*])={\mathbf Z}/p^n$.
\end{enumerate}

  Alternatively, one can proceed as in \cite[II]{Gr}, and use projective space theorem plus weak purity to obtain the same formulas. 
  This method of computations applies to de Rham (hence crystalline) and to Hodge cohomologies.
We get  (we write $BGL_{a,n}$ for $BGL_l/W_n(k)$ if it is not confusing)  
$$H^*_{\crr}(BGL_{a,n})\simeq H^*_{\dr}(BGL_{a,n})\simeq W_n(k)[x_1,\ldots,x_a],$$
 where the classes
$x_i\in H^{2i}_{\dr}(BGL_{a,n})$ are the de Rham Chern classes of the universal locally free sheaf on $BGL_{a,n}
$ (defined via a projective space theorem). Similarly, for  Hodge cohomology we have 
$$
H^*_{Hdg}(BGL_{a,n})\simeq W_n(k)[x_1,\ldots,x_a].
$$
In particular
$$
H^{2i}_{\dr}(BGL_{a,n})=H^i(BGL_{a,n},\Omega^{
i}_{BGL_{a,n}})=\oplus_IW_n(k)x_I,\quad I\subset \{1,\ldots, a\}, |I|=i.
$$
Similarly, we get
\begin{align*}
H^*_{\crr}(BGL(a,b)_n) & \simeq H^*_{\dr}(BGL(a,b)_n)\simeq H^*_{\dr}((BGL_{a}\times BGL_{b})_n)\\
 & \simeq W_n(k)[x_1,\ldots,x_{a};y_1,\ldots,y_{b}].
\end{align*}
\subsection{Syntomic computations}
\subsubsection{Syntomic cohomology}
 \begin{lemma}
\label{syntcomp}Let $B$ be one of the classifying simplicial schemes $BGL_l$, $B(GL_a\times GL_b)$, or $BGL(a,b)$. Then
$$
H^{i}(B,S_n(j))=\begin{cases}
M & \text { for } i=2j, \\
H_{\dr}^{i-1}(B) & \text{ for } i=2j-2m-1,m\geq 0,\\
0 & \text{ for } i\geq 2j+2,
\end{cases}
$$
where $M$ is equal to $\oplus_I{\mathbf Z}/p^nx_I$ for $BGL_l$ and to  $\oplus_{I,J}{\mathbf Z}/p^nx_Iy_J$ in the other two cases. Here
$I\subset \{1,\ldots, l\},|I|=j$; $I\subset \{1,\ldots, a\},|I|=c, J\subset \{1,\ldots, b\},|J|=d, c+d=j$, respectively. 

 Moreover we have the following long exact sequence
$$
0\to   H^{2j}(B,S_n(j))\to H^{2j}(B_n,\Omega^{\geq
j}_{B_n})\stackrel{1-\phi_j}{\longrightarrow}H^{2j}_{\dr}(B_n)\\
 \to  H^{2j+1}(B,S_n(j))\to 0.
$$
\end{lemma}
\begin{proof}
First we will prove the lemma for $BGL_l$. We have 
$$S_n(j)=\Cone(\Omega^{\geq
j}_{BGL_{l,n}}\stackrel{1-\phi_j}{\longrightarrow}\Omega^{{\scriptscriptstyle\bullet}}_{BGL_{l,n}})[-1]
$$
This yields 
 the  long exact sequence
$$
\cdots \to  H^{i}(BGL_l,S_n(j))\to H^{i}(BGL_{l,n},\Omega^{\geq
j}_{BGL_{l,n}})\stackrel{1-\phi_j}{\longrightarrow}H^{i}_{\dr}(BGL_{l,n})\\
\to  H^{i+1}(BGL_l,S_n(j))\to \cdots.
$$
Using the vanishing of the de Rham cohomology in odd degrees we get the following long exact sequence ($i\geq 0$)
$$
0\to   H^{2i}(B,S_n(j))\to H^{2i}(B_n,\Omega^{\geq
j}_{B_n})\stackrel{1-\phi_j}{\longrightarrow}H^{2i}_{\dr}(B_n)\\
 \to  H^{2i+1}(B,S_n(j))\to 0.
$$
Hence the long exact sequence in the lemma. 

By considering the maps
$$
H^{2j}(B,S_n(j))\to H^{2j}(B_n,\Omega^{\geq
j}_{B_n})^{\phi_j=1},\quad H^{i-1}_{\dr}(B_n)
 \to  H^{i}(B,S_n(j))
$$
and by devissage we can reduce the  computations to $n=1$. The long exact sequence in our lemma  implies that
$$H^{2j}(BGL_l,S_1(j))\simeq 
H^{2j}(BGL_{l,1},\Omega^{\geq
j}_{BGL_{l,1}})^{\phi_j=1}=(\oplus_I W_1(k)x_I)^{\phi_j=1},
$$
where $I\subset \{1,\ldots, l\},|I|=j$. 
But $\phi(x_i)=p^ix_i$; hence $H^{2j}(BGL_l,S_1(j))\simeq \oplus_I{\mathbf Z}/px_I$, as wanted.

  For $2i\geq 2j+2$ we claim that the map $$1-\phi_j: H^{2i}(BGL_{l,1},\Omega^{\geq
j}_{BGL_{l,1}})\to H^{2i}_{\dr}(BGL_{l,1})$$ is an isomorphism. 
By the computations of de Rham and Hodge cohomologies above it suffices to show that the map
$$1-\phi_j: H^{i}(BGL_{l,1},\Omega^{
i}_{BGL_{l,1}})\to H^{i}(BGL_{l,1},\Omega^{
i}_{BGL_{l,1}})$$ is an isomorphism. 
 But $\phi_j=p^{i-j}\phi_i$ and $1-\phi_j=1$, as wanted. It follows that
$ H^{i}(BGL_l,S_1(j))=0$ for $i\geq 2j+2$.

  For $i\leq 2j-2$ and $i$  even we have $H^{i}(BGL_l,S_1(j))\hookrightarrow  H^{i}(BGL_{l,1},\Omega^{\geq
j}_{BGL_{l,1}})=0$.  This implies that, for $i\leq 2j$ and $i$ odd, $H^{i}(BGL_l,S_1(j))=H^{i-1}_{\dr}(BGL_{l,1})$, as wanted.

 In view of the above computations of crystalline cohomology, the argument for $B(GL_a\times GL_b)$ and  $BGL(a,b)$ is practically identical .
\end{proof}

  Similarly, we have the following lemma.
\begin{lemma}
\label{syntcomp1}Let $B$ be as above. Then
$$
H^{i}(B,S^{\prime}_n(j))=\begin{cases}
M^{\prime} & \text { for } i=2j,  \\
H_{\dr}^{i-1}(B_n) & \text{ for } i=2j-2m-1,m\geq 0,\\
H_{\dr}^{i-1}(B_j) & \text{ for } i=2j+2m+1,m\geq 1,
\\H_{\dr}^{i}(B_{n-j}) & \text{ for } i=2j+2m,m\geq 1,\\
0 & \text{ otherwise,}
\end{cases}
$$
where $M^{\prime}$ is equal to $(\oplus_IW_n(k)x_I)^{\phi=p^j}$ for $BGL_l$ and to $(\oplus_{I,J}W_n(k)x_Iy_J)^{\phi=p^j}$ in the other two cases (with the index sets $I,J$ as in the previous lemma).

 Moreover we have  the following long exact sequence
$$
0\to   H^{2j}(B,S^{\prime}_n(j))\to H^{2j}(B_n,\Omega^{\geq
j}_{B_n})\stackrel{p^j-\phi}{\longrightarrow}H^{2j}_{\dr}(B_n)\\
 \to  H^{2j+1}(B,S^{\prime}_n(j))\to 0.
$$
\end{lemma}
\begin{proof}We present the proof for $BGL_l$. The proofs for the other schemes are basically the same. We
have   the  long exact sequence
$$
\cdots \to  H^{i}(BGL_l,S^{\prime}_n(j))\to H^{i}(BGL_{l,n},\Omega^{\geq
j}_{BGL_{l,n}})\stackrel{p^j-\phi}{\longrightarrow}H^{i}_{\dr}(BGL_{l,n})\\
\to  H^{i+1}(BGL_l,S^{\prime}_n(j))\to \cdots.
$$
Using the vanishing of the de Rham cohomology in odd degrees we get the following  long exact sequence ($i\geq 0$)
$$
0\to   H^{2i}(B,S^{\prime}_n(j))\to H^{2i}(B_n,\Omega^{\geq
j}_{B_n})\stackrel{p^j-\phi}{\longrightarrow}H^{2i}_{\dr}(B_n)\\
 \to  H^{2i+1}(B,S^{\prime}_n(j))\to 0.
$$
Hence the long exact sequence in the lemma. 
It implies that
$$H^{2j}(BGL_l,S^{\prime}_n(j))\simeq H^{2j}(BGL_{l,n},\Omega^{\geq
j}_{BGL_{l,n}})^{\phi=p^j}=(\oplus_I W_n(k)x_I)^{\phi=p^j},
$$
where $I\subset \{1,\ldots, l\},|I|=j$.

  For $2i\geq 2j+2$ we claim that the following sequence is exact
\begin{align*}
0\to H^{2i}(BGL_{l,n-j}, & \Omega^{\geq
j}_{BGL_{l,n-j}}) \to H^{2i}(BGL_{l,n},\Omega^{\geq
j}_{BGL_{l,n}})\stackrel{p^j-\phi}{\longrightarrow} H^{2i}_{\dr}(BGL_{l,n})\\
& \to H^{2i}_{\dr}(BGL_{l,j})\to 0
\end{align*} 
Indeed, by the computations of de Rham and Hodge cohomologies above it suffices to show that the following sequence is exact
\begin{align*}
0\to  H^{i}(BGL_{l,n-j}, & \Omega^{
i}_{BGL_{l,n-j}})\to H^{i}(BGL_{l,n},\Omega^{
i}_{BGL_{l,n}})\stackrel{p^j-\phi}{\longrightarrow} H^{i}(BGL_{l,n},\Omega^{
i}_{BGL_{l,n}})\\
& \to H^{i}(BGL_{l,j},\Omega^{
i}_{BGL_{l,j}})\to 0
\end{align*} But $p^j-\phi=p^j(1-\phi_j)$ and we have shown in the  previous proof that $1-\phi_j$ is an isomorphism.

  For $i\leq 2j-2$ and $i$  even we have $H^{i}(BGL_l,S^{\prime}_n(j))\hookrightarrow  H^{i}(BGL_{l,n},\Omega^{\geq
j}_{BGL_{l,n}})=0$.  This implies that, for $i\leq 2j$ and $i$ odd, $H^{i}(BGL_l,S^{\prime}_n(j))=H^{i-1}_{\dr}(BGL_{l,n})$, as wanted.
\end{proof} 
\begin{remark}
\label{product11}
Note that the map
$$\tau: H^{2j}(BGL_l,S_n(j))\hookrightarrow H^{2j}(BGL_l,S^{\prime}_n(j)):\quad \oplus_I{\mathbf Z}/p^nx_I\hookrightarrow  (\oplus_I W_n(k)x_I)^{\phi=p^j},
$$
is the canonical injection and that it is compatible with products (since both products are compatible with the de Rham product). An analogous statement is true for the maps
\begin{align*}
\tau: H^{2j}(BGL_a\times BGL_b,S_n(j)) & \hookrightarrow H^{2j}(BGL_a\times BGL_b,S^{\prime}_n(j)),\\
\tau: H^{2j}(BGL(a,b),S_n(j))& \hookrightarrow H^{2j}(BGL(a,b),S^{\prime}_n(j))
\end{align*}
that are both equal to the canonical injection
$$\oplus_{I,J}{\mathbf Z}/p^nx_Iy_J\hookrightarrow (\oplus_{I,J} W_n(k)x_Iy_J)^{\phi=p^j}.
$$
\end{remark}
\subsubsection{Truncated syntomic cohomology}
We  will now compute the truncated syntomic cohomology of $BGL_l$
and the related simplicial schemes $B(GL_a\times GL_b)$ and $BGL(a,b)$.
\begin{proposition}
\label{truncated}Let $B$ denote the simplicial scheme $BGL_l/W(k)$, $B(GL_a\times GL_b)/W(k)$, or $BGL(a,b)/W(k)$. Then 
\begin{enumerate}
 \item the natural morphism $S_n(j)_{\Nis}\to S_n(j)$ induces an isomorphism
$$H^{2j}(B,S_n(j)_{\Nis})\stackrel{\sim}{\to}H^{2j}(B,S_n(j)),\quad j\geq 0.
$$
\item 
the morphism $$\tau: H^{2j}(B,S_n(j)_{\Nis})\to H^{2j}(B,S^{\prime}_n(j)_{\Nis}),\quad j\geq 0.
$$
is compatible with products.
\end{enumerate}
\end{proposition}
\begin{proof}
\begin{lemma}
\label{devissageNis}
For a smooth scheme $X$ over $W(k)$ and $n\geq m\geq 0$, the following sequence is exact
$$0\to S_m(j)_{\Nis}\to  S_n(j)_{\Nis}\to  S_{n-m}(j)_{\Nis}\to 0
$$
\end{lemma}
\begin{proof} We can argue locally so 
 assume  that there is a Frobenius lift on $X_n$, $n\geq 0$. Assume first that $j\geq 1$. Recall \cite[3.5]{K} that then the syntomic complex $S_n(j)$ can be represented by the complex
\begin{align}
\label{form}
 0\to \so_{X_n}\stackrel{-d}{\to}  \Omega^1_{X_n}\stackrel{-d}{\to}\ldots \stackrel{-d}{\to}\Omega^{j-2}_{X_n}\lomapr{(0,-d)}
 \Omega^{j}_{X_n}\oplus \Omega^{j-1}_{X_n}\verylomapr{(1-\phi_j,-d)}\Omega^{j}_{X_n},
\end{align}
where $\so_{X_n}$ is in degree $1$.
On the \'etale site the last map is surjective  and $\tau_{\leq j}S_n(j)_{\eet}\simeq S_n(j)_{\eet}$ \cite[3.6]{K}. However this is not true on the Nisnevich site. The truncated complex $S_n(j)_{\Nis}$ can be represented by the complex 
\begin{align*}
 0\to \so_{X_n}\stackrel{-d}{\to}  \Omega^1_{X_n}\stackrel{-d}{\to}\ldots\stackrel{-d}{\to} \Omega^{j-2}_{X_n}\lomapr{-d}N_{X_n}(j)_{\Nis},
\end{align*}
where $N_{X_n}(j)_{*}$ is  the kernel of the map $(1-\phi_j,-d):\Omega^{j}_{X_n}\oplus \Omega^{j-1}_{X_n}\to\Omega^{j}_{X_n}$ taken in topology $*$ and $\so_{X_n}$ is in degree $1$.

    Let 
$$\nu^j_{X,n}=\ker(1-\phi_j:\Omega^{j}_{X_n}\to \Omega^{j}_{X_n}/d\Omega^{j-1}_{X_n}), \quad j\geq 0.$$ We easily find that we have the short exact sequence
 \begin{align}
 \label{above}
 0\to Z\Omega^{j-1}_{X_n}\to N_{X_n}(j)_{\Nis}\to \nu^j_{X,n}\to 0,
 \end{align}
  where $Z\Omega^{j-1}_{X_n}=\ker(d: \Omega^{j-1}_{X_n}\to \Omega^{j}_{X_n})$. We note that $Z\Omega^{j-1}_{X_n}\simeq Z\Omega^{j-1}_{X}\otimes{\mathbf Z}/p^n$.  It follows that we have the short exact sequence
  \begin{equation}
  \label{eye}
  0\to \widetilde{S}_n(j)_{\Nis}\to S_n(j)_{\Nis}\to \nu^j_{X,n}\to 0,
  \end{equation}
  where
  $
  \widetilde{S}_n(j)_{\Nis}$ is the following complex
  $$
   0\to \so_{X_n}\stackrel{-d}{\to}  \Omega^1_{X_n}\stackrel{-d}{\to}\ldots\stackrel{-d}{\to} \Omega^{j-2}_{X_n}\lomapr{-d}Z\Omega^{j-1}_{X_n},
    $$
    We have the short exact sequence
    $$
    0\to \wt{S}_m(j)_{\Nis}\to  \wt{S}_n(j)_{\Nis}\to  \wt{S}_{n-m}(j)_{\Nis}\to 0
    $$
    
   It follows that it suffices to show that the natural sequence
   $$
   0\to \nu^j_{X,m}\to \nu^j_{X,n}\to\nu^j_{X,n-m}\to 0
        $$
   is exact. We will show below that we have a natural isomorphism $ \nu^j_{X,n,\eet}\stackrel{\sim}{\to} W_n\Omega^j_{X_1,\log}$ 
   hence a natural isomorphism $\nu^j_{X,n}\stackrel{\sim}{\to} W_n\Omega^j_{X_1,\log}$.  Since $W_n\Omega^j_{X_1,\log}$ is   locally in the Nisnevich topology generated by symbols $\dlog\{[a_1],\ldots,[a_j]\}$, for local sections  $a_1,\ldots, a_j\in \so^{\times}_{X_1}$ \cite[p.505]{Il},  this will give us the exactness we want. 
   
     To construct the map $f: \nu^j_{X,n,\eet}{\to} W_n\Omega^j_{X_1,\log}$, consider the following commutative diagram
     $$\xymatrix{
     0\ar[r] & \nu^j_{X,n,\eet}\ar[r] & \Omega^j_{X_n}\ar[d] \ar[r]^{1-\phi_j} & \Omega^j_{X_n}/d\Omega^{j-1}_{X_n}\ar[d] \ar[r] & 0\\
     0\ar[r] &  W_n\Omega^j_{X_1,\log}\ar[r] & W_n\Omega^j_{X_1}/dVW_{n-1}\Omega^{j-1}_{X_1}\ar[r]^{1-F_j} & W_n\Omega^j_{X_1}/dW_n\Omega^{j-1}\ar[r] & 0
          }
     $$
     The vertical maps are induced by the global Frobenius lift. The top sequence is exact by definition, the bottom one by Lemma 4.3 from \cite{BEK}. We obtained the map $f$. To show that it is an isomorphism we can reduce by devissage to $n=1$. But then our claim is clear. 
     
      For $j=0$, we we have 
      $$
      S_n(0)_{\eet}: \so_{X_n}\lomapr{1-\phi}\so_{X_n}.
      $$
      Hence $\tau_{\leq 0}S_n(0)_{\eet}\simeq S_n(0)_{\eet}$ and $$
      S_n(0)_{\Nis}=\ker(1-\phi:\so_{X_n}\to \so_{X_n})=\nu^0_{X,n},      $$
      which we have shown satisfies devissage. 
\end{proof}  

 We claim now that to show that the map $H^{2j}(B_n, S_n(j)_{\Nis}) \to H^{2j}(B_n, S_n(j)) $, $j\geq 0$, is an isomorphism we may assume $n=1$. Indeed, by Lemma \ref{devissageNis}, it suffices to show that
  $H^{2j+1}(B_n,S_n(j)_{\Nis})=0$ and that the map $H^{2j-1}(B_n,S_n(j))\to H^{2j-1}(B_m,S_m(j))$, $n\geq m$,  is surjective. For the first, by devissage, reduce to $n=1$. Then we have a global Frobenius. Since $H^a(B_n,\Omega^b_{B_n})$, $a\neq b$, we reduce to showing that $H^{j+1}(B_n,N_{B_n}(j)_{\Nis})=0$ or, by the exact sequence (\ref{above}), to $H^{j+1}(B_{n,\Nis},\nu^j_{B,n})=0$ and $H^{j+1}(B_{n},Z\Omega^{j-1}_{B_n})=0$. But this follows from (\ref{vanish0}) and from (\ref{vanish1}) below.  For the second claim, by Lemma \ref{syntcomp}, we need to show that the map $H^{2j-2}_{\dr}(B_n)\to H^{2j-2}_{\dr}(B_m)$, $n\geq m$, is surjective. But this is clear.

    Take then $n=1$.  We want to show that the map $H^{2j}(B_1, S_1(j)_{\Nis}) \to H^{2j}(B_1, S_1(j)) $, $j\geq 0$, is an isomorphism. Consider   the natural map $ S_n(j)_{\Nis}\to  W_n\Omega^j_{B_1,\log }[-j]$ and its \'etale analog.  Since the map $H^j(B_{1,\Nis},\Omega^j_{B_1,\log})\to H^j(B_{1,\eet},\Omega^j_{B_1,\log})$ is an isomorphism (see Section \ref{classical}),
 it suffices to show that so are the maps
 \begin{equation}
 \label{isom1}
 H^{2j}(B_1, S_n(j)_{\Nis}) \to H^j(B_1,\Omega^j_{B_1,\log} ),\quad H^{2j}(B_1, S_n(j)_{\Nis}) \to H^j(B_1,\Omega^j_{B_1,\log} ).
    \end{equation}
 We will argue in the Nisnevich case, the \'etale case being analogous. 
  We have a global Frobenius. Hence,  by 
   the short exact sequence (\ref{eye}), 
  it suffices to show that  $H^{2j}(B_n,\wt{S}_n(j)_{\Nis})=H^{2j+1}(B_n,\wt{S}_n(j)_{\Nis})=0$.

  Since $H^i(B_n,\Omega^k_{B_n})=0$ unless $i=k$, we have  $H^{2j+1}(B_n,\wt{S}_n(j)_{\Nis})=0$.  The vanishing of $H^{2j}(B_n,\wt{S}_n(j)_{\Nis})$ will follow when we show that $H^j(B_n,Z\Omega^{j-1}_{B_n})=0$. We will, in fact, show more. Namely that
  \begin{equation}
  \label{vanish1}
  H^k(B_n,Z\Omega^{j}_{B_n})=H^k(B_n,B\Omega^{j}_{B_n})=0,\quad k > j, 
  \end{equation}
  where $B\Omega^{j}_{B_n}=\im(d: \Omega^{j-1}_{B_n}\to \Omega^{j}_{B_n})$. 
  We will argue by induction on $j$ starting with $j=-1$, which is the trivial case. Assume that the statement is true for $j$. To prove it for $j+1$ recall that we have the  isomorphism
  $$C^{-1}: W_n\Omega^*_{B_1}\stackrel{\sim}{\to} \sh^*(W_n\Omega^{{\scriptscriptstyle\bullet}}_{B_1}).
  $$
  Here  $C^{-1}$ is the inverse Cartier isomorphism. 
  It yields the following short exact sequence of abelian sheaves
  $$0\to B\Omega^k_{B_n}\to Z\Omega^k_{B_n}\to \Omega^{k}_{B_1}\to 0
  $$
  Applying cohomology we get the exact sequence
  $$H^k(B_n,B\Omega^{j+1}_{B_n})\to H^k(B_n,Z\Omega^{j+1}_{B_n})\to H^k(B_n,\Omega^{j+1}_{B_1}),\quad k>j+1.
  $$
  The last group is trivial, so it remains to show that so is $H^k(B_n,B\Omega^{j+1}_{B_n})$. For that we use the short exact sequence
  $$0\to Z\Omega^{j}_{B_n}\to \Omega^{j}_{B_n}\stackrel{d}{\to}B\Omega^{j+1}_{B_n}\to 0
  $$
  and the induced long exact sequence
  $$H^k(B_n,\Omega^{j}_{B_n})\stackrel{d}{\to} H^k(B_n,B\Omega^{j+1}_{B_n})\to H^{k+1}(B_n,Z\Omega^{j}_{B_n})\to H^{k+1}(B_n,\Omega^{j}_{B_n}).
  $$
  Since the first and the last group are trivial we get that $H^k(B_n,B\Omega^{j+1}_{B_n})\stackrel{\sim}{\to} H^{k+1}(B_n,Z\Omega^{j}_{B_n})$. But, by induction, the last group is trivial and we are done.

      For the second statement of the proposition, consider the following commutative diagram
$$
\xymatrix{
      H^{2j}(B,S_n(j)_{\Nis})\ar[d]^{\wr}\ar[r]^{\tau} &  H^{2j}(B,S^{\prime}_n(j)_{\Nis})\ar[d]\ar@{^{(}->}[r] & H^{2j}_{\dr}(B_n)\\
         H^{2j}(B,S_n(j))\ar[r]^{\tau} & H^{2j}(B,S^{\prime}_n(j)) \ar@{^{(}->}[ru]  
}
$$      
The top and the bottom injections follow from the fact that $H^{2j-1}_{\dr}(B_n)=0$; they commute with products. It suffices now to evoke Remark \ref{product11}.
\end{proof}

    Let $B$ denote the simplicial scheme $BGL_l/W(k).$ Let $x_i\in H^{2i}(B,S_n(j)_{\Nis})$, $1\leq i\leq l$,  be the Chern classes of the universal locally free sheaf on $B_n$ obtained from the de Rham classes via the isomorphism in Proposition \ref{truncated}. 
    \begin{corollary}
\label{truncatedcor} The natural morphism
$$
H^{*}(W(k),S_n(*)_{\Nis})[x_1,\ldots,x_l]\stackrel{\sim}{\to}
H^{*}(B,S_n(*)_{\Nis})
$$
is an isomorphism. 
\end{corollary}
\begin{proof}Since we have shown in Proposition \ref{truncated} that the map $H^{2j}(B,S_n(j)_{\Nis})\to H^{2j}(B_n,\Omega_{B_n}^{\geq j})$ is an isomorphism, it suffices, by \cite[Example 2.7]{EN},  to show that, for $b\geq 1$,  so is the map 
$H^{2j}(B_n,\Omega_{B_n}\kr)\to H^{2j+1}(B,S_n(j+b)_{\Nis})$. 
 But this follows from the fact that $H^{2j}(B_n,\Omega^{\geq b+j}_{B_n})=H^{2j+1}(B_n,\Omega^{\geq b+j}_{B_n})=0$.
 \end{proof}
Analogous statements are true for $B(GL_a\times GL_b)/W(k)$ and $BGL(a,b)/W(k)$.
\subsubsection{Truncated syntomic-\'etale cohomology}

We will briefly recall a  computation of nearby cycles due to Bloch-Kato \cite{BK}. 
We work in Nisnevich topology. For a smooth scheme $X$ over $\so_K$, $j\geq 0$, set $L^j=i^*\R ^jj_*{\mathbf Z}/p(j)_{\M}$. For $m\geq 1$, let $U^mL^j$ be the subsheaf of $L^j$ generated locally be local sections of the form $\{x_1,\ldots,x_j\}$ such that $x_1-1\in\pi^mi^*\so_X$, for a uniformizer $\pi$ of $K$, and $x_i\in i^*j_*\so^*_{X_K}$, $2\leq i\leq j$. We set $U^0L^j:=L^j$.  
\begin{theorem}(Bloch-Kato \cite{BK})
\label{BK-Nis}
 \begin{enumerate}
 \item For a smooth scheme $X$ over $W(k)$, we have 
 $$
 \gr^m(L^j)\simeq \begin{cases}
 \Omega^j_{X_k,\log}\oplus\Omega^{j-1}_{X_k,\log} & \mbox{if } m=0,\\
  \Omega^{j-1}_{X_{k}} & \mbox{if } m=1.
  \end{cases}
  $$ Moreover, $
U^mL^j=0$, for $ m>1$.
\item  For a smooth scheme $X$ over ${\so_K}=W(k)(\zeta_p)$, we have
 $$
 \gr^m(L^j)\simeq 
 \begin{cases}
 \Omega^j_{X_{k^{\prime}},\log}\oplus\Omega^{j-1}_{X_{k^{\prime}},\log} & \mbox{if } m=0,\\
 \Omega^{j-1}_{X_{k^{\prime}}} & \mbox{if } 1\leq m<p,\\
 \Omega^{j-1}_{X_{k^{\prime}}}/(1-C)Z\Omega^{j-1}_{X_{k^{\prime}}}\oplus \Omega^{j-2}_{X_{k^{\prime}}}/(1-C)Z\Omega^{j-2}_{X_{k^{\prime}}}& \mbox{if }m=p.
  \end{cases}
 $$
 Moreover, 
for $m> p$, $U^mL^j=0$.
\end{enumerate}
\end{theorem}
\begin{proof}
Outside of vanishing of $U^mL^j$ both claims are proved in \cite[1.4.1,6.7,p.135]{BK}.  For the vanishing of $U^mL^j$ we can argue locally and then, by the Gersten conjecture for \'etale cohomology, we can pass to the fraction field $\widetilde{K}$ of the generic point of the special fiber. It suffices now to show that $U^mH^j_{\eet}(\widetilde{K},{\mathbf Z}/p(j))=0$ for $m>e^{\prime}$, $e^{\prime}=ep/(p-1)$. But cohomological symbol gives an isomorphism $K^M_j(\widetilde{K})/p\stackrel{\sim}{\to} H^j_{\eet}(\widetilde{K},{\mathbf Z}/p(j))$. Hence it suffices to evoke Lemma 5.1 from \cite{BK}.
\end{proof}

\begin{proposition}
\label{truncatedetale}
For $B$ being $BGL_l/W(k)$, 
$B(GL_a\times GL_b)/W(k)$, or $BGL(a,b)/W(k)$, we have the natural isomorphisms
$$H^{2j}(B,E_n(j)_{\Nis})\stackrel{\sim}{\to}H^{2j}(B,S_n(j)_{\Nis}),\quad H^{2j}(B,E^{\prime}_n(j)_{\Nis})\stackrel{\sim}{\to}H^{2j}(B,S^{\prime}_n(j)_{\Nis}).
$$
Moreover, the morphism
$$\tau: H^{2i}(BGL_l,E_n(i)_{\Nis})\to H^{2i}(BGL_l,E^{\prime}_n(i)_{\Nis})
$$
is compatible with products. 
\end{proposition}
\begin{proof}We will start with the first claimed isomorphism. 
Recall that we have  the distinguished triangle in Nisnevich topology (\ref{Niis3})
\begin{equation*}
j_!\tau_{\leq j}\R \varepsilon_*{\mathbf Z}/p^n(j)^{\prime}  \to E_n(j)_{\Nis}\to S_n(j)_{\Nis},
\end{equation*}
where $\varepsilon$ is the projection from the \'etale  to the Nisnevich site. Set $T_n(j):= \tau_{\leq j}\R \varepsilon_*{\mathbf Z}/p^n(j)$, $j\geq 0$. 
It suffices to show that $$
H^{2j}(B,j_!T_n(j))=H^{2j+1}(B,j_!T_n(j))=0,\quad  n\geq 1.
$$ 
We will use for that the distinguished triangle 
\begin{equation}
\label{Nisn}
j_!T_n(j)  \to j_*T_n(j) \to i_*i^*j_*T_n(j).
\end{equation}
It suffices to show that
\begin{align}
\label{added}
H^{2j-1}(B,j_*T_n(j)) &\twoheadrightarrow  H^{2j-1}( B,i_*i^*j_*T_n(j)),\quad 
H^{2j}(B,j_*T_n(j)) \stackrel{\sim}{ \to} H^{2j}( B,i_*i^*j_*T_n(j)),\\
H^{2j+1}(B,j_*T_n(j)) & =0.\notag
\end{align}

   We will  argue in the case of $B=BGL_l/W(k)$ -- the other cases being analogous. 
We will start with a computation of  the cohomology of the complexes $T_n(j)$; this will prove the third equality in (\ref{added}). In the category of smooth schemes over $K$, 
 by the Beilinson-Lichtenbaum Conjecture, $T_n(j)\simeq {\mathbf Z}/p^n(j)_{\M}$. In particular, 
the complexes $T_n(j)$ satisfy projective space theorem \cite[4.1]{Kah} and hence we can define the universal classes
$$x_j\in H^{2j}(BGL_l/K,T_n(j))
$$
as the Chern classes of the universal vector bundle over $BGL_l/K$. We start here with the first Chern class induced by the map 
${\mathbb G}_m[-1]\to \tau_{\leq 1}(\R \varepsilon_*{\mathbb G}_m[-1])\to \tau_{\leq 1}\R \varepsilon_*{\mathbf Z}/p^n(1)$ that arises from the \'etale map
 ${\mathbb G}_m[-1]\to {\mathbf Z}/p^n(r)$ of Kummer theory.

  These classes  are clearly finer than  the \'etale universal classes. In fact we have the following computation.
\begin{lemma}
\label{tretale}
We have the isomorphisms 
\begin{align*}
H^{*}_{\eet}(BGL_l/K,{\mathbf Z}/p^n(*))  & \simeq H^{*}_{\eet}(K,{\mathbf Z}/p^n(*)) [x_1,\ldots, x_l],\\
H^{*}(BGL_l/K,T_n(*)) & \simeq H^{*}(K,T_n(*)) [x_1,\ldots, x_l].
\end{align*}
They imply the following isomorphisms
\begin{align*}H^{2i+k}(BGL_l/K,T_n(i))&  =0;\quad k\geq 1;\\
H^{2i}(BGL_l/K,T_n(i)) & \simeq {\mathbf Z}/p^n(i)[x_I],\quad |I|=i, i\leq l;\\
H^{2i}_{\eet}(BGL_l/K,{\mathbf Z}/p^n(i)) & \simeq
 {\mathbf Z}/p^n(i)[x_I]\oplus {\mathbf Z}/p^n(i)[x_J],\quad |I|=i,  |J|=i-1, i\leq l,\\
 H^{2i-1}(BGL_l/K,T_n(i))  & \stackrel{\sim}{\to} H^{2i-1}_{\eet}(BGL_l/K,{\mathbf Z}/p^n(i))\simeq H^{1}_{\eet}(K,{\mathbf Z}/p^n(1)) [x_1,\ldots, x_l].
\end{align*}
Moreover the natural morphism  
$$
H^{2i}(BGL_l/K,T_n(i))\to H^{2i}_{\eet}(BGL_l/K,{\mathbf Z}/p^n(i)), \quad i\leq l,
$$
is  the obvious injection.
\end{lemma}
\begin{proof} The standard computation yields the isomorphism
$$H^{*}_{\eet}(BGL_l/K,{\mathbf Z}/p^n(*)) \simeq H^{*}_{\eet}(K,{\mathbf Z}/p^n(*)) [x_1,\ldots, x_l].
$$
For degree $(2i,i)$ only the groups $H^{0}_{\eet}(K,{\mathbf Z}/p^n) \simeq {\mathbf Z}/p^n$ and 
$
H^{2}_{\eet}(K,{\mathbf Z}/p^n(1))\simeq  {\mathbf Z}/p^n
$ matter. For degree $(2i-1,i)$ only the group $H^{1}_{\eet}(K,{\mathbf Z}/p^n(1))$ matters. On the other hand, by \cite[Lemma 7]{PO}, we  have
$$H^{*}(BGL_l/K,T_n(*))\simeq H^{*}(K,T_n(*)) [x_1,\ldots, x_l].
$$
Since $H^{a}(K,T_n(b))=0,$ $a >b,$ the groups $H^{2i+k}(BGL_l/K,T_n(i))$, $k\geq 1$, are clearly trivial. Also, for degree $(2i,i)$ (reps. $(2i-1,i)$ ) only the group $H^{0}(K,{\mathbf Z}/p^n(0)_{\M})\simeq H^{0}_{\eet}(K,{\mathbf Z}/p^n) \simeq {\mathbf Z}/p^n$ 
($H^{1}(K,{\mathbf Z}/p^n(1)_{\M})\simeq H^{1}_{\eet}(K,{\mathbf Z}/p^n(1))$) matters. Our lemma follows.
\end{proof}

 To prove the remaining claims in  (\ref{added}), we note that, since both $j_*T_n(j)$ and $i_*i^*j_*T_n(j)$ satisfy devissage, we can assume that  $n=1$.  For the second map, consider the following commutative diagram
$$
\begin{CD}H^{2j}(B_{{\so_K}^{\prime}},j_*{\mathbf Z}/p^n(j)_{\M}) @>>> H^{2j}(B_{{\so_K}^{\prime}}, i_*i^*j_*{\mathbf Z}/p^n(j)_{\M})\\
@VV N_1 V @VV N_2 V\\
H^{2j}(B,j_*{\mathbf Z}/p^n(j)_{\M}) @>>> H^{2j}(B, i_*i^*j_*{\mathbf Z}/p^n(j)_{\M})
\end{CD}
$$
Here ${\so_K}^{\prime}$ is the ring of integers in $K^{\prime}=K_0(\xi_p)$ and $N_1,N_2$ denote the  maps induced by the norm map (of complexes of \'etale sheaves)
$$\pi_*{\mathbf Z}/p^n(j)_{\M}\to {\mathbf Z}/p^n(j)_{\M},
$$
where $\pi: B_{K^{\prime}}\to B_{F}$ is the natural projection. 
In the case of the map $N_2$ this uses the isomorphisms
$$
\pi_*i_*i^*j_*T_n(j) \simeq \tau_{\leq j}i_*\pi_*i^*\R \varepsilon_*\R j_*T_n(j) \simeq \tau_{\leq j}i_*\R \varepsilon_*\pi_*i^*\R j_*T_n(j) 
  \simeq \tau_{\leq j}i_*i^*\R \varepsilon_*\R j_*\pi_*T_n(j). 
$$
The second isomorphism follows from the isomorphism $\R \varepsilon_*i^*=i^*\R \varepsilon_*$ \cite[2.2.b]{GD} and the third one from the same isomorphism and the proper base change theorem in \'etale cohomology.

 Since $N_*\pi^*=(p-1)$, to show that the second map in \ref{added}  is an isomorphism it suffices to show that so is the map
$$H^{2j}(B_{{\so_K}^{\prime}},j_*{\mathbf Z}/p(j)_{\M}) \to  H^{2j}(B_{{\so_K}^{\prime}}, i_*i^*j_*{\mathbf Z}/p(j)_{\M})
$$
Consider the following composition of maps
$$
\tau^{\prime}: H^{2j}(B_{{\so_K}^{\prime}},j_*{\mathbf Z}/p(j)_{\M}) \to  H^{2j}(B_{{\so_K}^{\prime}}, i_*i^*j_*{\mathbf Z}/p(j)_{\M})
\stackrel{\sigma_j}{\to}
H^{j}(B_{k^{\prime}},\Omega^j_{B_{k^{\prime}},\log}).
$$
The map $\sigma_j$ is induced by the composition 
$$\sigma_j:i^*j_*{\mathbf Z}/p(j)_{\M}\to i^*\sh^jj_*{\mathbf Z}/p(j)_{\M}[-j]\simeq i^*j_*\sk_j^M/p[-j]\to \Omega^j_{B_{k^{\prime}},\log}[-j]\oplus \Omega^{j-1}_{B_{k^{\prime}},\log}[-j].
$$
Here $\sk_j^M/p$ is the sheaf of Milnor K-theory groups mod-$p$ and the isomorphism follows from \cite[Theorem 7.6.]{KZ}. The last morphism  \cite[6.6]{BK} maps symbols
$
\{\tilde{x}_1,\ldots,\tilde{x}_j\}$,
for local sections $x_i$ of $\so^*_{B/k^{\prime}}$ and $\tilde{x}_i$ lifting of $x_i$ to $i^*\so_{B_{{\so_K}^{\prime}}}^*$, to $(\dlog(x_1)\wedge \cdots\wedge \dlog(x_j),0)$ and maps symbols 
$\{\tilde{x}_1,\ldots,\tilde{x}_{j-1},\pi_{K^{\prime}}\}$ to $(0,\dlog(x_1)\wedge \cdots\wedge \dlog(x_{j-1}))$.
First, we claim that  the map
$$\sigma_j:H^{2j}(B_{{\so_K}^{\prime}}, i_*i^*j_*{\mathbf Z}/p(j)_{\M})\to H^{j}(B_{k^{\prime}},\Omega^j_{B_{k^{\prime}},\log})
$$ is an isomorphism. This will follow from the spectral sequence
$$H^a(B_{{\so_K}^{\prime}},i_*i^*\sh^bj_*{\mathbf Z}/p(j)_{\M})\Rightarrow H^{a+b}(B_{{\so_K}^{\prime}},i_*i^*j_*{\mathbf Z}/p(j)_{\M})
$$
if we show that 
$$
H^a(B_{{\so_K}^{\prime}},i_*i^*\sh^bj_*{\mathbf Z}/p(j)_{\M})=
\begin{cases}
0 & \mbox{for } a+b=2j-1,2j, b< j,\\
 H^{j}(B_{k^{\prime}},\Omega^j_{B_{k^{\prime}},\log})& \mbox{for }a=b=j
\end{cases}
$$

 Let us start with  $b=j$. By a result of Bloch-Kato \cite{BK}, the nearby 
cycles $i_*i^*\sh^jj_*{\mathbf Z}/p(j)_{\M}$ have a descending filtration whose graded pieces are listed in Theorem \ref{BK-Nis}.  
It suffices thus to show that (we set
$B^{\prime}=B_{k^{\prime}}$ and $\Omega_{B^{\prime}}=\Omega_{B^{\prime}/k^{\prime}}$)
$$
 H^{i}(B^{\prime},\Omega^{j-1}_{B^{\prime}}/(1-C)Z\Omega^{j-1}_{B^{\prime}})= H^i(\Omega^{j-2}_{B^{\prime}}/(1-C)Z\Omega^{j-2}_{B^{\prime}})=0,\quad i=j,j+1,
$$
(note that $H^{j}(B^{\prime},\Omega^{j-1}_{B^{\prime}})=0$ and $H^{s}(B^{\prime},\Omega^{t}_{B^{\prime},\log})=0 $, for $s\neq t$). But this follows from the exact sequence
\begin{equation}
\label{cartier1}
0\to \Omega^k_{B^{\prime},\log}\to Z\Omega^k_{B^{\prime}}\lomapr{1-C}\Omega^k_{B^{\prime}}
\end{equation}
and the computations from Proposition \ref{truncated} that showed that $ H^{s}(B^{\prime},Z\Omega^{t}_{B^{\prime}})=0$, for $s>t$.

  For $b< j$, $a+b=2j$, use the isomorphism $\xi^{j-b}_p: i_*i^*\sh^bj_*{\mathbf Z}/p(b)\stackrel{\sim}{\longrightarrow}i_*i^*\sh^bj_*{\mathbf Z}/p(j)$ and 
 once more Theorem \ref{BK-Nis}. It suffices thus to show that the groups 
$$
H^{2j-b}(B^{\prime},\Omega^{b-1}_{B^{\prime},\log}\oplus \Omega^{b}_{B^{\prime},\log}), \quad 
H^{2j-b}(B^{\prime},\Omega^{b-1}_{B^{\prime}}),\quad H^{2j-b}(B^{\prime},\Omega^{b-1}_{B^{\prime}}/(1-C)Z\Omega^{b-1}_{B^{\prime}}),\quad H^{2j-b}(\Omega^{b-2}_{B^{\prime}}/(1-C)Z\Omega^{b-2}_{B^{\prime}})
$$
are trivial. But this follows as above from the exact sequence (\ref{cartier1})
and the computations in Proposition \ref{truncated}.

 The argument for $b< j$, $a+b=2j-1$, is analogous.
 
  It suffices now to show that the composition $\tau^{\prime}: H^{2j}(B_{{\so_K}^{\prime}},j_*{\mathbf Z}/p(j)_{\M}) \to H^{j}(B^{\prime},\Omega^{j}_{B^{\prime},\log})$ is an isomorphism as well. But both groups are isomorphic to 
  ${\mathbf Z}/p[x_I], |I|=j$. Hence it suffices to show that the map $\tau^{\prime}$ is compatible with Chern classes of vector bundles. Since both 
  cohomologies have projective space theorem, the map $\tau^{\prime}$ is functorial and compatible with products,  it suffices to show compatibility with the first Chern class maps and this is clear from the definition of the map $\sigma_j$.
  
  For the first map  of \ref{added} we note that injectivity follows from injectivity of the map
  $$H^{i}_{\eet}(B,\R  j_*{\mathbf Z}/p(j))\to H^{i}_{\eet}(B,i_*i^*\R j_*{\mathbf Z}/p(j)), \quad i\geq 0. 
  $$
  This is because we have Lemma \ref{tretale}. But, in fact, the above map is an isomorphism: since the sheaves $i_*i^*\R j_*{\mathbf Z}/p(j)$ satisfy projective space theorem and weak purity this follows from the method of Illusie \cite{Gr} of computing cohomology of classifying spaces (note that $H^{*}_{\eet}(W(k),\R j_*{\mathbf Z}/p(j))\simeq H^{*}_{\eet}(W(k),i_*i^*\R j_*{\mathbf Z}/p(j)) $).

  It remains to prove that the first map in (\ref{added}) is surjective. For that consider  the localization sequence in motivic cohomology
$$i_*{\mathbf Z}/p(j-1)_{\M}[-2]\to {\mathbf Z}/p(j)_{\M}\to j_*{\mathbf Z}/p(j)_{\M}
$$
Applying to it the exact functor $i_*i^*$ we obtain the following commutative diagram with short exact sequences as columns
$$
\xymatrix{
H^1(W(k),{\mathbf Z}/p(1)_{\M})[x_I] \ar[d]^{\beta_1} \ar[r]^{\gamma} &  \ker(\kappa_2) \ar[d]^{\beta_2}\ar[r]^{\sim} & H^{j-1}(B_k,\Omega^{j-1}_{B_k})\ar[d]\\
H^{2j-1}(B,j_*{\mathbf Z}/p(j)_{\M}) \ar[r] \ar[d]^{\kappa_1}& H^{2j-1}(B,i_*i^*j_*{\mathbf Z}/p(j)_{\M})\ar[r]^{\sim} \ar[d]^{\kappa_2} & H^{j-1}(B,i_*i^*\sh^jj_*{\mathbf Z}/p(j)_{\M})\ar[dd]^{\sigma_j}\\
H^{2j-2}(B_k,{\mathbf Z}/p(j-1)_{\M}) \ar[d]^{\wr}\ar@{=}[r] & H^{2j-2}(B_k,{\mathbf Z}/p(j-1)_{\M})\ar[d]^{\wr}\\
H^{j-1}(B_k,\Omega^{j-1}_{B_k,\log}) \ar@{=}[r] &   H^{j-1}(B_k,\Omega^{j-1}_{B_k,\log})\ar[r]^{\sim} & H^{j-1}(\Omega^{j-1}_{B_k,\log}\oplus\Omega^j_{B_k,\log}).
}
$$
The maps $\kappa_1,\kappa_2$ are the boundary maps induced by the above short exact sequence. In the top left term the index $I$ satisfies  $|I|=j-1$.
 For the rightmost column, we note that we have the short exact sequence
 $$0\to \Omega^{j-1}_{B_k}\to i^*\sh^jj_*{\mathbf Z}/p(j)_{\M}\to  \Omega^{j-1}_{B_k,\log}\oplus\Omega^{j}_{B_k,\log} \to 0
 $$
 of  Bloch-Kato (quoted in  Theorem \ref{BK-Nis}). Here the map $\Omega^{j-1}_{B_k}\to i^*\sh^jj_*{\mathbf Z}/p(j)_{\M}\simeq i^*j_*\sk^M_j/p$ is defined by sending 
 $x\dlog(y_1)\wedge\cdots\wedge\dlog(y_{j-1})$ to the symbol $\{1+\tilde{x}p,\tilde{y}_1,\cdots, \tilde{y}_{j-1}\}$. 
 This yields 
the  short exact sequences
 $$
 0\to  H^{j-1}(B_k,\Omega^{j-1}_{B_k})\to  H^{j-1}(B,i_*i^*\sh^jj_*{\mathbf Z}/p(j)_{\M}) \to H^{j-1}(B_k,\Omega^{j-1}_{B_k,\log}\oplus\Omega^{j}_{B_k,\log})\to  0
 $$
The fact that the maps $\kappa_1 $  and $\kappa_2$ are surjective,  $H^1(W(k),{\mathbf Z}/p(1)_{\M})[x_I] \simeq \ker(\kappa_1)$, 
   and the obvious definitions of the maps $\beta_1$ and $\gamma$ follow from the fact that 
\begin{align*}
H^{*}(B,j_*{\mathbf Z}/p(j)_{\M}) & \simeq  H^{*}(F,{\mathbf Z}/p(*)_{\M}) [x_1,\ldots, x_l];\\
H^{j-1}(B_k,\Omega^{j-1}_{B_k,\log}) & \simeq   {\mathbf Z}/p[x_I],\quad |I|=j-1,
\end{align*}
and $\kappa_1$ is the natural morphism compatible with the short exact sequence
\begin{align}
\label{loc1}
0 \to H^1(W(k),{\mathbf Z}/p(1)_{\M})\to H^1(F,{\mathbf Z}/p(1)_{\M})\stackrel{\kappa_1}{\to} H^0(k,{\mathbf Z}/p(0)_{\M})\to 0.
\end{align}
Note that $H^0(k,{\mathbf Z}/p(0)_{\M})\simeq {\mathbf Z}/p $.

 Consider the natural  map 
  \begin{align*}
  H^{2j-1}(B,i_*i^*j_*{\mathbf Z}/p(j)_{\M}) {\to} H^{j-1}(B,i_*i^*\sh^jj_*{\mathbf Z}/p(j)_{\M}).
  \end{align*}
  We claim that it is an isomorphism.  This will follow from the fact that
 $H^{2j-1}(B,\Lambda(j))=H^{2j}(B,\Lambda(j))=0$, for $\Lambda(j):=\tau_{\leq {j-1}}i_*i^*j_*{\mathbf Z}/p(j) _{\M}$.  To see that this is true consider 
  the spectral sequence
$$H^a(B,\sh^b\Lambda(j))\Rightarrow H^{a+b}(B,\Lambda(j)).
$$
It suffices to prove that 
$$
H^a(B,\sh^b\Lambda(j))=
0,\quad  \mbox{for } a+b=2j-1, a+b=2j, b< j.
$$
We pass to ${\so_K}^{\prime}$. The norm map $N: H^a(B_{{\so_K}^{\prime}},\sh^b\Lambda(j))\to H^a(B,\sh^b\Lambda(j))$ 
  yields that the pullback map $\pi^*: H^a(B,\sh^b\Lambda(j))\to H^a(B_{{\so_K}^{\prime}},\sh^b\Lambda(j))$ is injective. 
  Hence it suffices to show that $H^a(B_{{\so_K}^{\prime}},\sh^b\Lambda(j))=0$ for $ b< j$ and $a+b=2j-1, 2j$. But this we have done above. 
  It follows that the map $\ker(\kappa_2)\to H^{j-1}(B_k.\Omega^{j-1}_{B_k})$ is an isomorphism as well. 

  It suffices now to show that the map
 $$
 \gamma: \quad H^1(W(k),{\mathbf Z}/p(1)_{\M})[x_I] \to \ker(\kappa_2)
 $$
 is an isomorphism, or, because it is injective 
 that $H^1(W(k),{\mathbf Z}/p(1)_{\M})[x_I] \simeq H^{j-1}(B_k,\Omega^{j-1}_{B_k})$, $|I|=j-1$, at least abstractly.  But
  $$H^{j-1}(B_k,\Omega^{j-1}_{B_k})\simeq k[x_I],\quad |I|=j-1; \quad H^1(W(k),{\mathbf Z}/p(1)_{\M})\simeq k.
  $$
The second isomorphism follows from the isomorphisms
$$H^1(W(k),{\mathbf Z}/p^n(1)_{\M})\simeq K^M_1(W(k))/p^n
\simeq W(k)^*/p^n\simeq W_n(k).
$$  
We have proved the first claim of the proposition.

 To show that we have the isomorphism $H^{2j}(B,E^{\prime}_n(j)_{\Nis})\stackrel{\sim}{\to}H^{2j}(B,S^{\prime}_n(j)_{\Nis})$ we use
 the distinguished triangle in Nisnevich topology (Lemma \ref{Niis3})
\begin{equation*}
j_!\tau_{\leq j}\R \varepsilon_*{\mathbf Z}/p^n(j)^{\prime}  \to E^{\prime}_n(j)_{\Nis}\to S^{\prime}_n(j)_{\Nis}.
\end{equation*}
It suffices to show that $$
H^{2j}(B,j_!T_n(j))=H^{2j+1}(B,j_!T_n(j))=0,\quad  n\geq 1.
$$ 
what we have just done. 

  The statement about products follows easily from the first part of the proposition and Proposition \ref{truncated}.
  \end{proof}
  \begin{remark}
  It is likely that, in fact, we have a natural isomorphism 
  $$H^{*}(B,E_n(*)_{\Nis})\stackrel{\sim}{\to}H^{*}(B,S_n(*)_{\Nis}).  $$
  \end{remark}
  \begin{corollary}
  \label{syntet}
  The period map
  $$
  \alpha_j: H^{2j}(BGL_l,E^{\prime}_n(j)_{\Nis})\to  H^{2j}_{\eet}(BGL_{l,K},\Z/p^n(j)^{\prime})
  $$
  maps  the universal class $x^{\synt}_j$ to $p^jx^{\eet}_j$. For $j\leq p-2$ and $E$-cohomology, we have an analogous statement with no twist necessary.
  \end{corollary}
  \begin{proof}
  We have the following commutative diagram
  $$
  \xymatrix{
  H^{2j}(BGL_l,E^{\prime}_n(j)_{\Nis})\ar[d]_{\wr}\ar[r]^-{\alpha_j} &  H^{2j}_{\eet}(BGL_{l,F},\Z/p^n(j)^{\prime})\ar[d]^{i^*}_{\wr}\\
H^{2j}(BGL_l,S^{\prime}_n(j))\ar[r]^-{\alpha_j}\ar[d] &  H^{2j}_{\eet}(BGL_{l},i_*i^*\Z/p^n(j)^{\prime})\\
 H^{2j}(BGL_l,S_n(j))\ar[ru]_-{p^j\tilde{\alpha}_j}
  }
  $$
  The top left arrow is an isomorphism by Propositions \ref{truncated} and \ref{truncatedetale}, the right arrow by the proof of Lemma \ref{tretale}. It suffices now to show that the bottom period map $\alpha_j$ is compatible with the universal classes as stated. But this was shown in the proof of Theorem 4.10 in \cite{N10}, i.e., we have
$$
\alpha_j(x^{\synt}_j)=p^j\tilde{\alpha}_j(x^{\synt}_j)=p^jx^{\eet}_j,
$$
as wanted.

  \end{proof}
  \section{K-theory}In this section we review or prove basic facts concerning $K$-theory of simplicial schemes and log-$K$-theory.
\subsection{K-theory of simplicial schemes}We start with $K$-theory of simplicial schemes. 
\subsubsection{Definition of $K$-theory}
For a scheme $X$, 
let $K(X)$ be the Thomason-Throbaugh spectrum of nonconnective K-theory as defined in Thomason-Throbaugh \cite[6.4]{T2}. 
For a natural number $n\geq 2$, let $K/n(X)$ denote the corresponding spectrum mod $n$ \cite[9.3]{T2}. 
After the usual rigidification \cite[5.1.2]{GS1}, which we assume from now on, these spectra are strictly contravariant in $X$ and covariant on the categories of noetherian schemes and proper maps of finite Tor-dimension and of quasi-compact schemes and perfect projective morphisms \cite[3.16.5, 3.16.5]{T2}.

 Following Thomason \cite[5.6]{T}, we define  functors from 
simplicial schemes to the category of spectra
$$
X\mapsto K(X):=\holim_iK(X_{i}),\quad
X\mapsto K/{n}(X):=\holim_iK/{n}(X_{i}), 
$$
and set $$
K_i(X):=\pi_i(K(X)),\quad
K_i(X,{\mathbf Z}/n):=\pi_i(K/{n}(X)).
$$  
We have 
$$K_j(X)=\pi_j(\holim_i K(X_i)_0), \quad j\geq 0; \quad 
K_j(X,{\mathbf Z}/n)=\pi_j(\holim_i K(X_i)_0,{\mathbf Z}/n), \quad
j\geq 2.
$$
Here, for any  spectrum $F$, $F_0$ denotes its 
0'th space. 
  
  We will also need $K$-theory with compact support. Let $X$ be a scheme and $i:D\hookrightarrow X$ a simple normal crossing divisor. Let $D_1, D_2, \ldots, D_n$ be the irreducible components of $D$. Let $\wt{D}_{\kr}=\cosk^D_0(D^{(1)})$ be the coskeleton of $D^{(1)}=\coprod_{i=1}^{i=n}D_i$ over $D$, i.e., the \v{C}ech-nerve of the map $D^{(1)}\to D$. We have
$$
\cosk^D_0(D^{(1)})_n=D^{(1)}\times_D\ldots\times_D D^{(1)},\quad\quad (n+1)-\text{times},
$$
with the natural boundary and degeneracy maps. We have the following $K$-theory with compact support spectra
\begin{align*}
K^c(X,D) :=\fiber (K(X)\stackrel{i^*}{\to}K(\wt{D}_{\kr})), \quad 
K^c/n(X,D):=\fiber (K/n(X)\stackrel{i^*}{\to}K/n(\wt{D}_{\kr})),
\end{align*}
and the following  $K$-theory groups with compact support
\begin{align*}
K_i^c(X,D):=\pi_i(K^c(X,D)),\quad
K_i^c(X,D,{\mathbf Z}/n):=\pi_i(K^c/{n}(X,D)).
\end{align*}
Clearly
\begin{align*}
K^c_j(X,D) & =\pi_j(\fiber( K(X)_0\stackrel{i^*}{\to}K(\wt{D}_{\kr})_0)), \quad j\geq 0; \quad\\ 
K^c_j(X,D,{\mathbf Z}/n) & =\pi_j(\fiber( K(X)_0\stackrel{i^*}{\to}K(\wt{D}_{\kr})_0),{\mathbf Z}/n), \quad
j\geq 2.
\end{align*}

     For a scheme $Y$ with an ample family of line bundles \cite[2.1.1]{T2} (for example, for $Y$ regular or quasi-projective), there exists  a 
strictly natural homotopy
equivalence $K(Y)\simeq K_Q(Y)$, where $K_Q(Y)$ is the Quillen K-theory
spectrum of $Y$ suitably rigidified \cite[5.1.2]{GS1}, \cite[3.10]{T2}. Hence, strictly
 naturally,
$$K(Y)_0\simeq K_Q(Y)_0\simeq \Omega BQP(Y),
$$
where $\Omega BQP(Y)$ is the loop space of the classifying space of Quillen Q-construction applied to the category of finitely generated locally free sheaves on $Y$. 
Recall that for $Y=\Spec(A)$, the strictly natural map $$K_0(A)\times BGL(A)^+\stackrel{\sim}{\to}\Omega BQP(Y)
$$ is a homotopy equivalence.
\subsubsection{Operations on $K$-theory}
 Let now $X$ be a 
simplicial scheme that is degenerate above certain  degree. Using the above
 homotopy equivalences of spaces, the classical definition of 
$\lambda$-operations on K-theory can be extended to  $K_j(X)$, $j\geq 0$, 
and to $K_j(X,{\mathbf Z}/n)$, $j\geq 2$, 
\cite[3.2]{GS}, \cite[B.2.6,B.2.8]{HW}. Let us recall how this is done. 

 We work with the site $C$ of noetherian $B$-schemes, for $B$ a field or a local ring, equipped with the Zariski topology. We equip the category of (globally pointed) presheaves of simplicial sets on $C$ 
 with Jardine's model structure \cite{J1}.
    Recall that a map of presheaves of simplicial sets $E\to F$ is called a weak equivalence if it induces an isomorphism
     $\tilde{\pi}_*(E)\to\tilde{\pi}_*(E) $ on the 
     sheaves of homotopy groups. For a presheaf of simplicial sets $F$ we will denote by $F^{f}$ a fibrant replacement of $F$. That is,
we have a map $F\to F^f$ to a fibrant presheaf of simplicial sets $F^f$ that is a weak equivalence.
 Note that $F\to F^f$  can be chosen to be functorial.
    
    For a presheaf of simplicial sets $X$, we define the cohomology of $X$ with values in $F$ 
$$H^{-m}(X,F):=[S^mX,F]=\pi(S^mX,F^f), \quad m\geq 0.
$$
Here $S^mX$ denotes the m'th suspension of the  presheaf $X$. The bracket $[,]$ denotes maps in the homotopy category of presheaves of simplicial sets
and $\pi(,)$ stands for the set of pointed homotopy classes of maps. Similarly, we define  the mod $p^n$ cohomology of $X$ with values in $F$ 
$$H^{-m}(X,F; {\mathbf Z}/p^n):=[P^mX,F]=\pi(P^mX,F^f), \quad m\geq 2.
$$
Here $P^mX=P^m\wedge X$, where $P^m$ denotes the constant presheaf of $m$-dimensional mod $p^n$ Moore spaces.

 We have several important spectral sequences.  By filtering $X$ by its skeletons one constructs  spectral sequences of Bousfield-Kan type \cite[B.1.7]{HW}.
\begin{align}
\label{spectr}
E^{st}_1 & =H^{-t}(X_s,F) \Rightarrow
H^{s-t}(X,F), \quad t-s\geq 1,\\
E^{st}_1 & =H^{-t}(X_s,F;{\mathbf Z}/m) \Rightarrow
H^{s-t}(X,F;{\mathbf Z}/m), \quad t-s\geq 3.\notag
\end{align}
By the same method we get the other hypercohomology spectral sequences, namely, the weight spectral sequences \cite[5.13,5.48]{T}.
\begin{align}
\label{spectr1}
E^{st}_2 & =H^s(n\mapsto \pi_t(F(X_n))) \Rightarrow
H^{s-t}(X,F), \quad t-s\geq 1,\\
E^{st}_2 & =H^s(n\mapsto \pi_t(F(X_n),{\mathbf Z}/m)) \Rightarrow
H^{s-t}(X,F;{\mathbf Z}/m),  \quad t-s\geq 3.\notag
\end{align}
Finally we have the Brown spectral sequences induced by taking the Postnikov tower of $F$ \cite[Prop. 2]{GS}
\begin{align}
\label{spectr2}
E^{st}_2 & =H^{s}(X,\tilde{\pi}_t(F)) \Rightarrow
H^{s-t}(X,F), \quad t-s\geq 1,\\
E^{st}_2 & =H^{s}(X,\tilde{\pi}_t(F,{\mathbf Z}/n)) \Rightarrow
H^{s-t}(X,F;{\mathbf Z}/n), \quad t-s\geq 3.\notag
\end{align}
In these spectral sequences the r'th differential is $d_r: E^{s,t}_r\to E_r^{s+r,t+r-1}$. 
The spectral sequences converge strongly for $X$ degenerate above certain simplicial degree. 

   Denote by $K$ the presheaf ${\mathbf Z}\times {\mathbf Z}_{\infty}BGL$, where $BGL(U)=\injlim_nBGL_n(U)$. Then, for a simplicial scheme $X$,  we have \cite[3.2.3]{GS}
\begin{align*}
K_m(X)=H^{-m}(X,K),\quad m\geq 0;\qquad K_m(X,{\mathbf Z}/p^n)=H^{-m}(X,K;{\mathbf Z}/p^n), \quad m\geq 2,
\end{align*}
where we wrote $X$ for the scheme $X$ as well as for the presheaf of simplicial sets on $C$ represented by $X$. 
This presentation of K-theory groups as generalized cohomology of presheaves of the K-theory spaces allowed Gillet-Soul\'e \cite[4.2]{GS} to mimic the classical definition and to define maps $$
\lambda^k:\quad K_m(X)\to K_m(X), \quad k\geq 1,m\geq 0,$$
that turn $K_m(X)$ into an $H^0(S^0,{\mathbf Z})$-$\lambda$-algebra for any presheaf $X$ 
that is $K$-coherent \cite[3.1]{GS}, i.e., for which
$$\injlim_NH^{-m}(X,K^N)\stackrel{\sim}{\to}H^{-m}(X,K),
$$
where $K^N={\mathbf Z}\times {\mathbf Z}_{\infty}BGL_N$. Note that $H^0(S^0,{\mathbf Z})$ is a $\lambda$-ring. 
Similarly, for any $K$-coherent presheaf $X$, we can define  maps $$\lambda^k:\quad K_m(X,{\mathbf Z}/p^n)\to K_m(X,{\mathbf Z}/p^n), \quad k\geq 1,m\geq 2,
$$
that make $K_m(X,{\mathbf Z}/p^n)$, $m\geq 2$, into an $H^0(S^0,{\mathbf Z})$-$\lambda$-algebra.
\subsubsection{Examples}
Presheaves whose components are representable by  schemes  of the site $C$ (except for one copy of $*$ in each degree) we will call constructed from schemes. If they are degenerate above a finite simplicial degree and are built from regular schemes they are $K$-coherent \cite[Lemma 2.1]{Jeu}. Loosely, we will call such presheaves
(pointed regular)  finite simplicial schemes. 
\begin{example}
A simplicial scheme gives rise to a presheaf built from schemes. An $m$-truncated simplicial scheme $X$ gives rise to a finite simplicial scheme $\skl_mX$ that is degenerate above degree $m$. If $X$ is a regular schemes and $i:D\hookrightarrow X$ a divisor  that has normal crossings then the simplicial scheme $\wt{D}_{{\scriptscriptstyle\bullet}}=\cosk_0^D(D^{(1)})$ is not degenerate above any simplicial degree if $D$ has more than one irreducible component. To remedy this note that
$$\wt{D}_{n}=\coprod_{I\in\{1,\ldots,m\}^n}\bigcap_{i\in I} D_i,\quad D=\cup_{i=1}^{i=m}D_i.
$$
And  consider the simplicial scheme $$
\wt{D}^{\prime}_{{\scriptscriptstyle\bullet}}:\quad n \mapsto \coprod_{I\in \Delta(m)_n}\bigcap_{i\in I}D_i
$$
with the natural face and degeneracy maps. Here $\Delta(m)$ is the simplicial set
$$\Delta(m)_k=\{(i_0,\ldots,i_k)|1\leq i_0\leq {\scriptscriptstyle\bullet}\leq i_k\leq m\}
$$
The simplicial scheme $\wt{D}^{\prime}_{{\scriptscriptstyle\bullet}}$ is degenerate above degree $m-1$ and the natural inclusion of simplicial schemes $$
\wt{D}^{\prime}_{{\scriptscriptstyle\bullet}}\to \wt{D}_{{\scriptscriptstyle\bullet}}$$
 is a weak equivalence \cite[Lemma 7.1]{HW}.

     In particular, this can be applied to $K$-theory with compact support. Set $$C(X,D):=\cofiber(\wt{D}_{{\scriptscriptstyle\bullet}}\stackrel{i_*}{\to}X).$$
Define
$$K^c_m(X,D)=H^{-m}(C(X,D),K),\quad m\geq 0;\qquad K^c_m(X,D,{\mathbf Z}/p^n)=H^{-m}(C(X,D),K;{\mathbf Z}/p^n), \quad m\geq 2.
$$
Replacing $C(X,D)$ with the weakly equivalent $C^{\prime}(X,D)=
\cofiber(\wt{D}^{\prime}_{{\scriptscriptstyle\bullet}}\stackrel{i_*}{\to}X)$
that is degenerate above degree $m$ we get $\lambda$-operations on $K$-theory in the case all the irreducible components  of $D$ are regular. 
\end{example}
\begin{example}
\label{relative}
For any map $f: Y\to X$ of noetherian finite dimensional schemes the cone $C(X,Y)$ is degenerate above finite simplicial degree. 
The homotopy cofibre sequence 
$$Y\stackrel{f}{\to}X\to C(X,Y)
$$
yields the long exact sequence
$$\to H^{-m}(C(X,Y),K;{\mathbf Z}/p^n)\stackrel{}{\to}H^{-m}(X,K;{\mathbf Z}/p^n)\stackrel{f^*}{\to}H^{-m}(Y,K;{\mathbf Z}/p^n)\to H^{-m+1}(C(X,Y),K;{\mathbf Z}/p^n)\to
$$
Here $m\geq 3$. 
In the case $X$ and $Y$ are regular, this sequence is compatible with $\lambda$-operations. 
\end{example}
\subsubsection{$\gamma$-filtrations}
  The construction of the Loday product $$
{\mathbf Z}_{\infty}BGL_N(U)\wedge {\mathbf Z}_{\infty}BGL_N(U) \to {\mathbf Z}_{\infty}BGL_N(U) 
  $$
  is functorial in $U$. It induces products 
$$
  [S^mX,K]\times [S^nX,K]\to [S^{n+m}X,K],\quad 
  [P^mX,K]\times [P^nX,K]\to [P^{n+m}X,K]
$$
  for all $K$-coherent presheaves $X$. 

 For any $K$-coherent presheaf  $X$, we have the following $\gamma$-filtrations compatible with products:
\begin{align*}
F^k_{\gamma}K_0(X) & =
\begin{cases}
K_0(X)  & \text{if $k\leq 0$},\\
\langle 
\gamma_{i_1}(x_1) \cdots\gamma_{i_n}(x_n)|\ve(x_1)=
\ldots =\ve(x_n)=0,i_1+\cdots +i_n\geq k \rangle & \text{if $k> 0$},
\end{cases}\\
F^k_{\gamma}K_q(X,  {\mathbf Z}/p^n)  &  =\langle 
\gamma_{i_1}(x_1)\cup \cdots\cup\gamma_{i_n}(x_n)|
 x_i\in K_{q_i}(X,{\mathbf Z}/p^n), q_i\geq 2, \\
& \qquad \qquad i_1+\cdots +i_n\geq k\rangle,\\
\sff^k_{\gamma}K_q(X,  {\mathbf Z}/p^n)   & =\langle 
 a \gamma_{i_1}(x_1)\cup \cdots\cup\gamma_{i_n}(x_n)|
a\in F^{i_0}_{\gamma}K_0(X) ,x_i\in K_{q_i}(X,{\mathbf Z}/p^n), q_i\geq 2, \\
& \qquad \qquad i_0+ i_1+\cdots +i_n\geq k\rangle,
\end{align*}
 where $p^n>2$ and $\ve$ is the augmentation  
$\ve: K_0(X)\to H^0(X,{\mathbf Z})$ obtained by projecting ${\mathbf Z}\times {\mathbf Z}_{\infty}BGL$ to ${\mathbf Z}$. 
\begin{lemma}
\label{length}
If $X$ is a regular finite simplicial scheme over a field  such that $X\simeq \skl_mX$ then $$
F^{j+d+m+1}_{\gamma}K_j(X,{\mathbf Z}/p^n)=0, \quad j\geq 2,
$$
where $d$ is the maximum of dimensions of the schemes appearing in $X$.
\end{lemma}
\begin{proof} 
First, we claim that  $F^{q}_{\gamma}K_j(X,{\mathbf Z}/p^n)\subset F^{q-j}K_j(X,{\mathbf Z}/p^n),$
where the last filtration is obtained from the Brown spectral sequence (\ref{spectr2}). Since the Brown spectral sequence is multiplicative,
it suffices to show that for any $N\geq 1$ we have 
$$\gamma^{q}(H^{-j}(X,K^N;{\mathbf Z}/p^n))\subset F^{q-j}H^{-j}(X,K;{\mathbf Z}/p^n),\quad j\geq 2.
$$
Take $x\in H^{-j}(X,K^N;{\mathbf Z}/p^n)$ and consider the following diagram
$$
\xymatrix{
P^jX\ar[d]\ar[r]^{x} & K^N\ar[r]^{\gamma^q_N}& K\ar[r]\ar[d] & K<q-1>\ar[d]\\
S^jX\ar[rr]^{\gamma^q(x)}& & K/p^n\ar[r] & K/p^n<q-1>,
}
$$
where $K/p^n$ is the $0$'th  space of the spectrum $K/p^n$ and, for a presheaf of simplicial sets $F$, we wrote
$$\{F<n>\}:= \ldots\to  F<m+1> \to F<m>\to  F<m-1> \to \ldots
$$
for  a Postnikov tower of $F$ that we chose to be functorial. The above diagram commutes (in the homotopy category): the right square commutes by functoriality; the left 
square commutes by S-duality \cite[A.2]{Kah}
$$[P^jX,K]\simeq [S^jX,K/p^n].
$$

  Now recall that we have
  \begin{align*}
  F^{q-j}H^{-j}(X,K;{\mathbf Z}/p^n)=\ker(H^{-j}(X,K/p^n)\to H^{-j}(X,K/p^n<q-1>)).
  \end{align*}
Hence it suffices to show that the bottom map in the above diagram is nullhomotopic. Since we have an injection 
$$[S^jX,K/p^n]\hookrightarrow [P^jX,K/p^n]
$$
it suffices to prove that the composition 
$${\mathbf Z}_{\infty}BGL_N \stackrel{\gamma^q_N}{\longrightarrow} {\mathbf Z}\times {\mathbf Z}_{\infty}BGL_N
\to  K \to K<q-1>
$$
is nullhomotopic. But this was done in \cite[5.1]{GS}.

  It remains now to show that 
\begin{equation*}
F^{b}K_j(X,{\mathbf Z}/p^n)=0,\quad b > d+m.
\end{equation*} 
First we claim that, for any $b>\cd X$, $$
F^{b}K_j(X,{\mathbf Z}/p^n)=0.
$$ 
Indeed, we have
$$F^{b}H^{-j}(X,K;{\mathbf Z}/p^n)=\ker(H^{-j}(X,K/p^n)\to H^{-j}(X,K/p^n<b+j-1>)).
$$
and the following map of Brown spectral sequences
$$
\xymatrix{H^{-j}(X,K/p^n)\ar[r] & H^{-j}(X,K/p^n<b+j-1>)\\
H^s(X,\tilde{\pi}_t(K/p^n))\ar@{=>}[u]\ar[r] & H^s(X,\tilde{\pi}_t(K/p^n<b+j-1>))\ar@{=>}[u]
}
$$
It follows that 
 $F^{b}K_j(X,{\mathbf Z}/p^n)=0,\quad b >\cd X,$ where $\cd X$ is the cohomological dimension of $X$.
Next we  show 
 that $\cd X\leq d+m$. For that, for any sheaf $\sff$  of abelian groups over $X$, it suffices to look at   the  spectral sequence (\ref{spectr})
 $$E_1^{st}=H^{-t}(X_s,\sff) \Rightarrow H^{s-t}(X,F)
 $$
 and to remember that $X\simeq \skl_mX$ and $\cd X_s\leq d$. 
\end{proof}
We will also consider another  $\gamma$-filtration: 
$\widetilde{F}^i_{\gamma}=\langle\gamma^k(x)|k\geq i\rangle$, where $\langle\ldots\rangle$ denotes the subgroup generated by 
the given elements. These filtrations are related. For a scheme $X$,  by \cite[3.4]{So} and Lemma \ref{length},
we have 
\begin{align}
\label{gmma}
  M(d,i,2j)F^i_{\gamma} K_j(X,{\mathbf Z}/p^n)
\subset\widetilde{F}^i_{\gamma} K_j(X,{\mathbf Z}/p^n)
\subset F^i_{\gamma} K_j(X,{\mathbf Z}/p^n),\quad j\geq 2,
\end{align}
where $d$ is the dimension of $X$ and the integers
$ M(k,m,n)$ are defined by the following procedure \cite[3.4]{So}.
Let $l$ be a positive integer, and let $w_l$ be the greatest common divisor of
 the set of integers $k^{N}(k^l-1)$, as $k$ runs over the positive integers
 and $N$ is large enough with respect to $l$. Let
$M(k)$ be the product of the $w_l$'s for $2l<k$.
 Set $M(k,m,n)=\prod_{2m\leq 2l\leq n+2k+1}M(2l)$. An odd prime $p$ divides $M(d,i,j)$ if and only if 
$p<(j+2d+3)/2$, and divides $M(l)$ if and only if 
$p<(l/2)+1$.

  More generally we have the following lemma.
\begin{lemma}
\label{length1}
If $X$ is a regular finite simplicial scheme over a field  such that $X\simeq \skl_mX$ then 
$$ M(d+m,i,2j)F^i_{\gamma} K_j(X,{\mathbf Z}/p^n)
\subset\widetilde{F}^i_{\gamma} K_j(X,{\mathbf Z}/p^n)
\subset F^i_{\gamma} K_j(X,{\mathbf Z}/p^n),\quad j\geq 2,$$
where $d$ is the maximum of dimensions of the components of $X$.
\end{lemma}
\begin{proof}
Since, having Lemma \ref{length}, the argument in \cite[3.4]{So} goes through almost verbatim (with $d+m$ in place of the cohomological  dimension), we refer
 the interested reader for details to \cite{So}.
\end{proof}

  For a bounded below complex of presheaves of abelian groups $\sff$ on the site $C$ equipped with Zariski toplogy and a simplicial presheaf $X$ on $C$, we set
$$
H^n(X,\sff):=H^{-n}(X,\sk(\sff)),\quad n\geq 0,
$$
where $\sk(\sff):=\sk(\tau_{\leq 0}\sff)$ is the Dold-Puppe functor. If $\sff$ is built from injectif sheaves then $\sk(\sff)$ is a fibrant presheaf of simplicial sets \cite[1.2.2]{GS} and, for  $X$ constructed from schemes, we have  $H^n(X,\sff)=H^{-n}\sk(\sff)(X)$. In particular, 
if $X$ is an object of the site $C$ and $\sff$ is a bounded below complex of sheaves of abelian groups on $C$, then
$$H^n(X,\sff)=H^n(X_{\Zar},\sff).
$$
\subsection{Log-K-theory}
In this section we collect basic facts about log-$K$-theory. 
\subsubsection{Definition of log-$K$-theory}
For a scheme $X$, 
let $G(X)$ be the Thomason-Throbaugh spectrum of  $G$-theory as defined in Thomason-Throbaugh \cite[3.3]{T2}. 
For a natural number $n\geq 2$, let $G/n(X)$ denote the corresponding spectrum mod $n$ \cite[9.3]{T2}. After the usual rigidification \cite[5.1.3]{GS1}, which we assume from now on, these spectra are strictly covariant for noetherian schemes and proper maps as well as for quasi-compact schemes and pseudo-coherent projective morphisms \cite[3.16.1, 3.16.3]{T2}.

   Following Thomason \cite[5.15]{T}, we define  functors from 
finite proper simplicial schemes to the category of prespectra
$$
X\mapsto {\mathbb G}(X):=\hocolim_iG(X_{i}),\quad
X\mapsto {\mathbb G}/{n}(X):=\hocolim_iG/{n}(X_{i}), 
$$
and set $$
G_i(X):=\pi_i({\mathbb G}(X)),\quad
G_i(X,{\mathbf Z}/n):=\pi_i({\mathbb G}/{n}(X)).
$$  
Similarly, we have $K$-theory prespectra and groups.

  Let $X$ be a scheme and $i:D\hookrightarrow X$ a simple normal crossing divisor. Define the following prespectra
\begin{align*}
G(X( D)) & :=\cofiber ({\mathbb G}(\wt{D}_{\kr})\stackrel{i_*}{\to }G(X)), \quad 
G/n(X( D)):=\cofiber ({\mathbb G}/n(\wt{D}_{\kr})\stackrel{i_*}{\to} G/n(X));\\
K(X( D)) & :=\cofiber ({\mathbb K}(\wt{D}_{\kr})\stackrel{i_*}{\to} K(X)), \quad 
K/n(X( D)):=\cofiber ({\mathbb K}/n(\wt{D}_{\kr})\stackrel{i_*}{\to} K/n(X))
\end{align*}
Taking homotopy groups we get the  log-$G$-theory and log-$K$-theory groups
\begin{align*}
G_i(X( D))& :=\pi_i(G(X( D))),\quad
G_i(X( D),{\mathbf Z}/n):=\pi_i(G/{n}(X( D)));\\
K_i(X( D))& :=\pi_i(K(X( D))),\quad
K_i(X( D),{\mathbf Z}/n):=\pi_i(K/{n}(X( D))).
\end{align*}
We have 
\begin{align*}
G_j(X(D)) & =\pi_j(\cofiber(\hocolim_i G(\wt{D}_i)_0\stackrel{i_*}{\to}
G(X)_0)), \quad j\geq 0; \\ 
G_j(X(D),{\mathbf Z}/n) & =\pi_j(\cofiber(\hocolim_i G(\wt{D}_i)_0\stackrel{i_*}{\to}
G(X)_0),{\mathbf Z}/n), \quad
j\geq 2.
\end{align*}
 Recall that, for a noetherian scheme $Y$, there exists  a 
homotopy
equivalence $G(Y)\simeq G_Q(Y)$, where $G_Q(Y)$ is the Quillen $G$-theory
spectrum of $Y$ (see \cite[3.13]{T2}). Hence 
$$G(Y)_0\simeq G_Q(Y)_0\simeq \Omega BQM(Y),
$$
where $\Omega BQM(Y)$ is the loop space of the classifying space of Quillen Q-construction applied to the category of coherent sheaves on $Y$. 
\subsubsection{Localization sequences}
   We have the following localization sequence.
\begin{lemma}
Let $Y$ be a closed subscheme of $X$ such that the scheme
$D_Y=Y\cap D$ is a simple normal crossing divisor on $Y$. Then there is a homotopy cofibre sequence
$$G(Y( D_Y))\to G(X( D))\to G(U( D_U)),
$$
where $U=X\setminus Y$.
In particular, we have the long exact sequence of homotopy groups
$$\to G_j(Y(D_Y))\to G_j(X( D))\to G_j(U( D_U)) \to G_{j-1}(Y(D_Y))\to
$$
Similarly for $G$-theory with mod-$n$ coefficients.
\end{lemma}
\begin{proof}
For any $r\geq 0$, $\wt{D}_{r}$ is the disjoint union of the schemes $D_{\sigma}$, where $\sigma$ runs over all maps $\sigma:\{1,\ldots, r-1\}\to
\{1,\ldots, n\}$ and $D_{\sigma}=D_{Im(\sigma)}$. For any $\sigma$ as above we have a 
homotopy cofibre sequence
$$
G(Y\cap D_{\sigma})\to G(D_{\sigma})\to G(U\cap D_{\sigma})
$$
Since homotopy colimits preserve homotopy cofibre sequences we get the following commutative diagram of homotopy cofibre sequences
$$
\begin{CD}
\hocolim_r G(Y\cap \wt{D}_{r})@>>> \hocolim_r G(\wt{D}_{r})@>>> \hocolim_r  G(U\cap \wt{D}_{r})\\
@VVV @VVV @VVV\\
G(Y) @>>> G(X) @>>> G(U)
\end{CD}
$$
Our lemma follows since taking homotopy cofibre preserves homotopy cofibre sequences. 
\end{proof}
The following lemma shows that log-$G$-theory equals $G$-theory of the open complement.
\begin{lemma}
\begin{enumerate}
\item The natural map $s: \wt{D}_{\kr}\to D$ induces  homotopy equivalences
$$
s_*:\quad {\mathbb G}(\wt{D}_{\kr})\stackrel{\sim}{\to}G(D),\quad s_*:\quad {\mathbb G}/n(\wt{D}_{\kr})\stackrel{\sim}{\to}G/n(D).
$$ 
 \item 
The natural maps
\begin{equation}
G(X( D))\stackrel{\sim}{\to}G(X\setminus D),\quad G/n(X( D))\stackrel{\sim}{\to}G/n(X\setminus D)
\end{equation}
are homotopy equivalences.
\item 
If all the irreducible components of $D$ are regular then the natural maps
$$
K(X( D))\stackrel{\sim}{\to}G(X( D)),\quad
K/n(X( D))\stackrel{\sim}{\to}G/n(X( D))
$$
are homotopy equivalences
\end{enumerate}
\end{lemma}
\begin{proof}
For the first property we will argue by induction on the number $m$ of irreducible components of $D$; the case of $m=1$ being clear. Write $D=\cup_{i=1}^{i=m}D_i, Y=D_1,  U=X\setminus Y$, where each $D_i$ is an irreducible component of $D$. We have  the following diagram of maps of homotopy cofibre sequences
$$
\begin{CD}
\hocolim_r G(Y\cap \wt{D}_{r})@>>> \hocolim_r G(\wt{D}_{r})@>>> \hocolim_r  G(U\cap \wt{D}_{r})\\
@AA\wr A @AA\wr A @AA\wr A\\
\hocolim_r G(Y) @>>> {\mathbb G}(\wt{D}_{{\scriptscriptstyle\bullet}}) @>>> {\mathbb G}(\wt{D}_{U,{\scriptscriptstyle\bullet}})\\
@V s_*V \wr V @Vs_* VV @Vs_* V\wr V\\
G(Y)@>>> G(D)@>>> G(D_U)
\end{CD}
$$
The top sequence is obtained as in the proof of the previous lemma. 
The right lower map is a homotopy equivalence by the inductive assumption. The left lower map is induced by an augmentation of a constant simplicial scheme, hence a homotopy equivalence. 

  For the second property consider the following commutative diagram of prespectra
$$
\xymatrix{
{\mathbb G}(\wt{D}_{\kr})\ar[r]^{i_*}\ar[d]^{s_*}_{\wr} & G(X)\ar[r]\ar@{=}[d] & G(X(D))\ar@{->}^{\wr}[d]\\
G(D)\ar[r]^{i_*} & G(X)\ar[r]& G(X\setminus D),
}
$$
where the bottom row is a homotopy cofiber sequence. Our property follows immediately. The third property follows
from the fact that for a regular schemes $Y$ the natural morphism $K(Y)\to G(Y)$ is a homotopy equivalence.
\end{proof}
We get the following localization sequence for log-$G$-theory.
\begin{corollary}
Let $Y$ be one of the irreducible components of the divisor $D$. We have the following homotopy cofiber sequence
$$
G(Y(D_Y))\stackrel{i_*}{\to}G(X(D^{\prime}))\to G(X(D)),
$$
where $D^{\prime}$ is the divisor $D$ minus the component $Y$ and $D_Y=D^{\prime}\cap Y$.
\end{corollary}
\begin{proof}
This follows immediately from the above lemma since we have the following map of homotopy cofiber sequences
$$
\xymatrix{G(Y(D_Y))\ar[d]^{\wr}\ar[r]^{i_*} & G(X(D^{\prime}))\ar[r]\ar[d]^{\wr} & G(X(D))\ar[d]^{\wr}\\
G(Y\setminus D_Y)\ar[r] & G(X\setminus D^{\prime})\ar[r] & G(X\setminus D).
}
$$
\end{proof}
\subsubsection{Pairings, projection formula, base change}
Recall \cite[5.1.3]{GS1} that the classical pairing of spectra \cite[3.15.3]{T2} 
$$
G(X)\wedge K(X)\to G(X)
$$
after the rigidification of the  $G$-theory  and $K$-theory  spectra that we have imposed makes the projection formula strict. That is,  for any closed immersion $f: Y\to X$
the pushforward map 
$f_*: G(Y)\to G(X)$ is a map of module spectra over the ring spectrum 
$K(X)$ in the strict sense, i.e., 
the following projection formula diagram of spectra is strictly commutative (not just commutative up to a homotopy)
$$
\xymatrix{
&   K(Y)\wedge G(Y)
\ar[r]^-{\wedge} & G(Y)\ar[dd]^{f_*}\\
K(X)\wedge G(Y)\ar[ur]^{f^*\wedge 1 }\ar[dr]^{1\wedge f_*}  & &\\
& K(X)\wedge G(X)\ar[r]^-{\wedge} & G(X)
}
$$
This allows us to state the following lemma.
\begin{lemma}
We have pairings of prespectra
$$
G(X( D)) \wedge K(X)\to G(X( D)),\quad 
G(X( D))\wedge K(X( D))\to G(X( D))
$$
that are compatible with the pairings
$$G(X\setminus D)\wedge K(X)\to G(X\setminus D),\quad G(X\setminus D)\wedge K(X\setminus D)\to G(X\setminus D).
$$
\end{lemma}
\begin{proof}
First, using the above strict projection formula, define the pairing ${\mathbb G}(\wt{D}_{\kr})\wedge K(X)\to {\mathbb G}(\wt{D}_{\kr})$ 
by the formula 
\begin{align*}
{\mathbb G}(\wt{D}_{\kr})\wedge K(X) & =
(\hocolim_nG(\wt{D}_n))\wedge K(X)\to
\hocolim_n(G(\wt{D}_n)\wedge K(X))\\
& \stackrel{\wedge}{\to}
\hocolim_nG(\wt{D}_n)={\mathbb G}(\wt{D}_{\kr}),
\end{align*}
where the first map is induced by the natural projection $\wt{D}_n\to X$.

 Similarly, set the pairing $G(X( D)) \wedge K(X)\to G(X( D))$ equal to 
the map
\begin{align*}
G(X( D)) \wedge K(X) & =\cofiber({\mathbb G}(\wt{D}_{\kr})\stackrel{i_*}{\to} K(X))\wedge K(X)
\to \cofiber({\mathbb G}(\wt{D}_{\kr})\wedge K(X)
\stackrel{i_*}{\to} K(X)\wedge K(X))\\ &
\stackrel{i_*\wedge}{\to}\cofiber({\mathbb G}(\wt{D}_{\kr})\stackrel{i_*}{\to} K(X))=
G(X( D)).
\end{align*}
The compatibility with the pairing $G(X\setminus D)\wedge K(X)\to G(X\setminus D)$ follows easily from the projection formula.
\end{proof}
Let $f: Y\hookrightarrow X$ be a closed immersion of schemes with normal simple crossing divisors
$D_X$, $D_Y$, such that $D_Y=f^{-1}D_X$. Assume that all the irreducible components of $D_X$, $D_Y$ are regular. Then we have the following generalization of the projection formula.
\begin{lemma}
\label{Kproj}
The following diagram commutes up to canonically chosen homotopy.
$$
\xymatrix{
&   G(Y( D_Y))\wedge K(Y)
\ar[r]^-{\wedge} & G(Y( D_Y))\ar[dd]^{f_*}\\
G(X( D_X))\wedge K(Y)\ar[ur]^{f^*\wedge 1 }\ar[dr]^{1\wedge f_*}  & &\\
& G(X( D_X))\wedge K(X)\ar[r]^-{\wedge} & G(X( D_X))
}
$$
\end{lemma}
\begin{proof}
 The easiest way to see this is to use the natural maps $G(Y( D_Y))\stackrel{\sim}{\to} K(Y\setminus D_Y)$ and $G(Y( D_Y))\stackrel{\sim}{\to }K(X\setminus D_X)$ to pass to the classical diagram that we know commutes up to canonically chosen homotopy \cite[3.17]{T2}.
$$
\xymatrix{
&   K(Y\setminus D_Y)\wedge K(Y)
\ar[r]^-{\wedge} & K(Y\setminus D_Y)\ar[dd]^{f_*}\\
K(X\setminus D_X)\wedge K(Y)\ar[ur]^{f^*\wedge 1 }\ar[dr]^{1\wedge f_*}  & &\\
& K(X\setminus D_X)\wedge K(X)\ar[r]^-{\wedge} & K(X\setminus D_X)
}
$$
\end{proof}
\begin{lemma}
\label{torindep}
Consider a pullback diagram of regular log-schemes
$$
\begin{CD}
Y@>i >> X\\
@VV j V @VV j^{\prime} V\\
Y^{\prime} @>i^{\prime} >> X^{\prime}
 \end{CD}
$$
where all the maps are regular closed immersions of codimension one and the 
 log-structures are induced by simple normal crossing divisors $D_{X^{\prime}}$, $D_{Y^{\prime}}$, $D_{X}$, $D_{Y}$ such that 
$$D_{Y^{\prime}}=(i^{\prime})^{-1}(D_{X^{\prime}}),\quad
 D_{Y}=j^{-1}D_{Y^{\prime}},\quad D_{X}=(j^{\prime})^{-1}D_{X^{\prime}}.
$$
Then there is a canonical homotopy between
$$
j^{{\prime}*}i^{\prime}_*\simeq i_*j^*:\quad K(Y^{\prime}( D_{Y^{\prime}}))\to K(X( D_Y)).
$$
\end{lemma}
\begin{proof}
Again, pass to the maps
$$j^{{\prime}*}i^{\prime}_*\simeq i_*j^*:\quad K(Y^{\prime}\setminus D_{Y^{\prime}})\to K(X\setminus D_Y).
$$
That there exists a canonical homotopy between these maps is a classical result
\cite[3.18]{T2} that follows from Tor-independence of $i^{\prime}$ 
and $j^{\prime}$.
\end{proof}
\subsection{Operations on log-$K$-theory}
In this section we will study the behaviour of certain localization sequences in $K$-theory   under $\lambda$-operations. 
\subsubsection{Adams-Riemann-Roch without denominators}We start with the proof of the Adams-Riemann-Roch without denominators.
Let $\tau$ be a natural transformation of $\lambda$-rings such that $\tau(0)=0$.
Let $X$ be a regular  scheme  and $i: Y\hookrightarrow X$ a regular divisor. Assume that both schemes are defined over $\so_K$. Set $U=X\setminus Y$, $j:U\hookrightarrow X$, and 
$$K_m^Y(X,{\mathbf Z}/p^n):=H^{-m}(C(X,U),K;{\mathbf Z}/p^n),\quad m\geq 2.
$$
For $m\geq 3$, by Example \ref{relative}, the long exact sequence 
$$\to K_{m}^Y(X,{\mathbf Z}/p^n)\stackrel{}{\to}K_m(X,{\mathbf Z}/p^n)\stackrel{}{\to}K_m(U,{\mathbf Z}/p^n)\stackrel{}{\to}K_{m-1}^Y(X,{\mathbf Z}/p^n)\to 
$$
is compatible with the action of $\tau$.

  We have a natural isomorphism  $i_*: K_m(Y,{\mathbf Z}/p^n)\stackrel{\sim }{\to }K_m^Y(X,{\mathbf Z}/p^n)$, $m\geq 2$. We need to understand how it behaves with respect to operations.  Recall that the group $K_m(Y,{\mathbf Z}/p^n)$ is a 
$K_0(Y)$-$\lambda$-algebra, i.e., the sum $K_0(Y)\oplus K_m(Y,{\mathbf Z}/p^n)$ is a $\lambda$-ring. As in any $\lambda$-ring, for every element $x\in K_m(Y,{\mathbf Z}/p^n)$, there exists an element 
$\tau(Y/X,x)\in K_m(Y,{\mathbf Z}/p^n)$ (cf., \cite[1.1.1]{Jou})
which is a universal polynomial with integral coefficients in $\lambda(\sn_{Y/X}^{\vee})$ and $\lambda (x)$, where $\sn_{Y/X}^{\vee}$ is the conormal sheaf of $Y$ in $X$, 
such that 
\begin{equation}
 \label{oper2}
\tau_Y(\lambda_{-1}([\sn_{Y/X}^{\vee}])x)=\lambda_{-1}([\sn_{Y/X}^{\vee}])\tau(Y/X,x).
\end{equation}
\begin{lemma}({\bf Adams-Riemann-Roch without denominators})
\label{RRK}
The diagram
$$
\xymatrix{
K_m(Y,{\mathbf Z}/p^n)\ar[r]^{\tau(Y/X,{\scriptscriptstyle\bullet})}\ar[d]^{i_*} &  K_m(Y,{\mathbf Z}/p^n)\ar[d]^{i_*}\\
K_m^Y(X,{\mathbf Z}/p^n)\ar[r]^{\tau_X} & K_m^Y(X,{\mathbf Z}/p^n)
}
$$
commutes, i.e., 
$$\tau_X(i_*(y))=i_*(\tau(Y/X,y)),\quad y\in K_m(Y,{\mathbf Z}/p^n)$$
\end{lemma}
\begin{proof}
Like in the proof of Lemma \ref{deformation} we argue by deformation to the normal cone. Since many computations here are very similar to the ones done in the proof of that lemma we will just describe those that are substantially different. 

 So we start with a special case of a closed immersion $i:Y\hookrightarrow X$, which is the zero section of a projective bundle $X={\mathbb P}(\sn\oplus \so_Y)$, where $\sn$ is an invertible sheaf on $Y$. The sheaf $\sn$ is the conormal sheaf of the immersion. Let $\pi: X\to Y$ be the projection.  Notice that the natural map $K_m^Y(X,{\mathbf Z}/p^n)\to K_m(X,{\mathbf Z}/p^n)$ is an injection. Hence it suffices to prove the commutativity of the following diagram
$$
\xymatrix{
K_m(Y,{\mathbf Z}/p^n)\ar[r]^{\tau(Y/X,{\scriptscriptstyle\bullet})}\ar[d]^{i_*} &  K_m(Y,{\mathbf Z}/p^n)\ar[d]^{i_*}\\
K_m(X,{\mathbf Z}/p^n)\ar[r]^{\tau_X} & K_m(X,{\mathbf Z}/p^n)
}
$$

  Take $y\in K_m(Y,{\mathbf Z}/p^n)$. As in the proof of Lemma \ref{deformation} we compute that there is an equality
$i_*(y)=\pi^*(\lambda_{-1}([\sn])y)$.  From the functoriality of operations
it follows that
$$
\tau_Xi_*(y)=\pi^*\tau_Y([\lambda_{-1}(\sn)]y).
$$
The equality (\ref{oper2}), the equality $\lambda_{-1}(\pi^*[\sn])=[i_*\so_Y]$ in $K_0(X)$ that follows from the exact sequence (\ref{Koszul}), and the projection formula in $K$-theory imply now that
\begin{align*}
 \tau_Xi_*(y)& = \pi^*(\tau(Y/X,y)\lambda_{-1}([\sn]))= \pi^*(\tau(Y/X,y))\lambda_{-1}(\pi^*[\sn])\\
& = \pi^*(\tau(Y/X,y))[i_*\so_Y]= i_*i^*\pi^*(\tau(Y/X,y))= i_*\tau(Y/X,y),
\end{align*}
as wanted.

  In general, we use deformation to the normal cone. We will use freely the notation from the proof of Lemma \ref{deformation}. As before, by functoriality of $K$-theory and \cite[3.18]{T2},
 we find that
$$j_{0*}i_*\tau(Y_0/W_0,y)= j_{0*}j_0^*\overline{i}_*\rho^*\tau(Y_0/W_0,y),
$$
where 
$$
\overline{i}_*:  K_m({\mathbb P}^1_Y,{\mathbf Z}/p^n)\to K_m^{{\mathbb P}^1_Y}(W,{\mathbf Z}/p^n),\quad 
j_{0}^*: K_m^{{\mathbb P}^1_Y}(W,{\mathbf Z}/p^n)\to K_m^Y(X,{\mathbf Z}/p^n).
$$
The projection formula in $K$-theory from \cite[3.17]{T2} and functoriality of $K$-theory and operations
imply now that
$$
j_{0*}i_*\tau(Y_0/W_0,y)= \overline{i}_*\tau({\mathbb P}^1_Y/W,y)\wedge ([t_*\so_X]+[l_*\so_{{\mathbb P}(\sn^{\vee}_{Y/X}\oplus \so_Y)}]).
$$
Functoriality,  the projection formula in $K$-theory, the fact that $l^*\overline{i}_*=0$,
 and \cite[3.18]{T2} imply that
$$
j_{0*}i_*\tau(Y_0/W_0,y)=t_*i_{\infty*}\tau(Y/{\mathbb P}(\sn^{\vee}_{Y/X}\oplus \so_Y),y).
$$
Here 
$$
i_{\infty*}: K_m(Y,{\mathbf Z}/p^n)\to K_m^Y({\mathbb P}(\sn_{Y/X}^{\vee}\oplus \so_Y),{\mathbf Z}/p^n),\quad 
t_*:  K_m^Y({\mathbb P}(\sn_{Y/X}^{\vee}\oplus \so_Y),{\mathbf Z}/p^n)\to K_m^{{\mathbb P}(\sn_{Y/X}^{\vee}\oplus \so_Y)}(W,{\mathbf Z}/p^n).
$$
Now we apply the computation we did in the special case to the embedding $i_{\infty}: Y\to {\mathbb P}(\sn^{\vee}_{Y/X}\oplus \so_Y)$ to conclude that
$$j_{0*}i_*\tau(Y_0/W_0,y)= t_*\tau_{{\mathbb P}(\sn^{\vee}_{Y/X}\oplus \so_Y)}i_{\infty*}(y).
$$

  Similarly, in Lemma \ref{deformation} we computed that $i_*(y)= j^*_0\overline{i}_*\rho^*(y)$.
This, functoriality of $K$-theory and operations, the projection formula in $K$-theory, 
the fact that $l^*\overline{i}_*=0$,
 and \cite[3.18]{T2} imply that
$$
j_{0*}\tau_Xi_*(y)=t_*\tau_{{\mathbb P}(\sn^{\vee}_{Y/X}\oplus \so_Y)}i_{\infty*}(y).
$$
Hence we have that $j_{0*}i_*\tau(Y_0/W_0,y)=j_{0*}\tau_Xi_*(y)$.

  It suffices now to show that the map
$$
j_{0*}: K_m^Y( X,{\mathbf Z}/p^n)\to K_m^{{\mathbb P}_Y}( W,{\mathbf Z}/p^n)
$$
is injective. But this map is isomorphic to the map $j_{0*}: K_m( Y,{\mathbf Z}/p^n)\to K_m( {{\mathbb P}_Y},{\mathbf Z}/p^n)$ that is injective by the projective spaces theorem in $K$-theory. 
\end{proof}
\subsubsection{$\gamma$-filtration and localization sequences}
 We need to understand the behaviour of the integral $\gamma$-filtrations in localization sequences. 
 We will need the following lemma. 
\begin{lemma}
\label{length2}
Let $X$ be a regular scheme over $\so_K$ of dimension $d$. Then
\begin{enumerate}
\item $F^{d+j+1}_{\gamma}K_j(X,{\mathbf Z}/p^n)=0$;
\item $M(d,i,2j)F^{i}_{\gamma}K_j(X,{\mathbf Z}/p^n)\subset \wt{F}^{i}_{\gamma}K_j(X,{\mathbf Z}/p^n)\subset F^{i}_{\gamma}K_j(X,{\mathbf Z}/p^n).$
\end{enumerate}
\end{lemma}
\begin{proof}
The property (1) is proved in \cite[14.5]{Ls}. It is stated there just for $K$-theory wth no coefficients but it hold as well for $K$-theory modulo $p^n$. 
All one needs to do is to paraphrize Corollary 13.11 to the case with coefficients which is immediate. The property (2) follows now as in \cite[3.4]{So}.
\end{proof}
Consider a triple of regular schemes 
$$U\stackrel{j}{\hookrightarrow}X\stackrel{i}{\hookleftarrow} Z
$$
with $Z$ a closed subscheme of $X$ of codimension one and $U$ its complement.
For $j\geq 3$, we have the localization sequence in $K$-theory
\begin{equation}
\label{locseq}
\to K_j(Z,{\mathbf Z}/p^n)\stackrel{i_*}{\to}K_j(X,{\mathbf Z}/p^n)\stackrel{j^*}{\to}K_j(U,{\mathbf Z}/p^n)\stackrel{\partial}{\to}K_{j-1}(Z,{\mathbf Z}/p^n)\to
\end{equation}
obtained from the long exact sequence
$$
\to K_j^Z(X,{\mathbf Z}/p^n)\to K_j(X,{\mathbf Z}/p^n)\stackrel{j^*}{\to}K_j(U,{\mathbf Z}/p^n)\to K_{j-1}^Z(X,{\mathbf Z}/p^n)\to
$$
and the isomorphism  $i_*: K_j(Z,{\mathbf Z}/p^n)\stackrel{\sim}{\to}K_j^Z(X,{\mathbf Z}/p^n)$. 

 We claim that, modulo certain universal constants, this sequence of maps behaves well with respect to the $\gamma$-filtration.
\begin{lemma}
\label{exactseq1}
Let $j\geq 3$. There exists a natural number $N=N(d,i,j)$ dependent only on $d,i,j$, where $d$ is the dimension of $X$, such that  we have a long sequence of maps
$$\to F^{i-1}_{\gamma}K_j(Z,{\mathbf Z}/p^n)\stackrel{Ni_*}{\to}F^i_{\gamma}K_j(X,{\mathbf Z}/p^n)\stackrel{Nj^*}{\to}F^i_{\gamma}K_j(U,{\mathbf Z}/p^n)\stackrel{N\partial}{\to}F^{i-1}_{\gamma}K_{j-1}(Z,{\mathbf Z}/p^n)\to.
$$
If a prime number  $p>j+d+1$ then $p$ does not divide $N(d,i,j)$.
\end{lemma}
\begin{proof}
By functoriality, the restriction map $j^*$ is compatible with $\gamma$-filtrations. So is the boundary map $\partial$ with $\wt{F}^{i}_{\gamma}$-filtration, i.e., $\partial: \wt{F}^i_{\gamma}K_j(U,{\mathbf Z}/p^n){\to}F^{i-1}_{\gamma}K_{j-1}(Z,{\mathbf Z}/p^n)$.
 To see that take an element $x=\gamma^k(y)\in \wt{F}^i_{\gamma}K_j(U,{\mathbf Z}/p^n)$ for $y\in K_j(U,{\mathbf Z}/p^n)$, $k\geq i$. Then, by the Adams-Riemann-Roch without denominators (Proposition \ref{RRK}) we have that $\partial x=\gamma^k(\sn,x)(\partial y)$, for the normal bundle $\sn=\sn_{Z/X}$ of $Z$ in $X$. But, for every $z\in K_{j-1}(Z,{\mathbf Z}/p^n)$, we have $\gamma^k(x)(\sn,z)\in F^{k-1}_{\gamma}K_{j-1}(Z,{\mathbf Z}/p^n)$ 
 \cite[0, Appendice, Prop. 1.5]{S6}. Since by Lemma \ref{length2}
$M(d,i,2j)F^i_{\gamma}K_j(U,{\mathbf Z}/p^n)\subset \wt{F}^i_{\gamma}K_j(U,{\mathbf Z}/p^n)$ 
this implies that
$$M(d,i,2j)\partial: \quad F^i_{\gamma}K_j(U,{\mathbf Z}/p^n){\to}F^{i-1}_{\gamma}K_{j-1}(Z,{\mathbf Z}/p^n).
$$
Consider now $x\in F^{i-1}_{\gamma}K_j(Z,{\mathbf Z}/p^n)$. We know that 
$\gamma^i(\sn,x)=(-1)^{i-1}(i-1)!x \mod F^{i}_{\gamma}K_j(Z,{\mathbf Z}/p^n)$. By Adama-Riemann-Roch without denominators from Proposition \ref{RRK} we have 
$$i_*((-1)^{i-1}(i-1)!x)=\gamma^{i}(i_*x)\quad \mod i_*(F^{i}_{\gamma}K_j(Z,{\mathbf Z}/p^n)).
$$
By induction, since $F^{j+d}_{\gamma}K_j(Z,{\mathbf Z}/p^n)=0$ , we get that 
$$C(d,i,j)i_*:\quad F^{i-1}_{\gamma}K_j(Z,{\mathbf Z}/p^n){\to}F^i_{\gamma}K_j(X,{\mathbf Z}/p^n),\quad C(d,i,j)=(i-1)!i!\cdots (j+d-1)!
$$
Set $N(d,i,j)=M(d,i,2j)C(d,i,j)$. Since an odd prime number $p$ divides $M(d,i,j)$ if and only if 
$p<(j+2d+3)/2$ we get the last statement of the lemma. 
\end{proof}
\begin{lemma}
\label{homolog1}
let $j\geq 3$, $N=N(d,i-1,j)N(d,i,j)$ for the constants $N(d,i-1,j)$ and $N(d,i,j)$ from Lemma \ref{exactseq1}. Then the following long sequence is exact up to certain universal constants
$$\to F^{i-1}_{\gamma}/F^{i}_{\gamma}K_j(Z,{\mathbf Z}/p^n)\stackrel{Ni_*}{\to}F^{i}_{\gamma}/F^{i+1}_{\gamma}K_j(X,{\mathbf Z}/p^n)\stackrel{Nj^*}{\to}F^{i}_{\gamma}/F^{i+1}_{\gamma}K_j(U,{\mathbf Z}/p^n)\stackrel{N\partial}{\to}F^{i-1}_{\gamma}/F^{i}_{\gamma}K_{j-1}(Z,{\mathbf Z}/p^n)\to,
$$
where $N=N(d,i-1,j)N(d,i,j)$ for the constants from Lemma \ref{exactseq1}.
More precisely, 
if the element $[x]$ at any level of the above long sequence is a cocycle then  $C[x]$ is a coboundary for the following  constant $C$ 
\begin{enumerate}
\item $Ni_*([x])=0$ then $C=(i-1)!i!\cdots (j+d-1)!N(d,i,j)^2N(d,i,j+1)^2$;
\item $N\partial ([x])=0$ then $C=(i-1)!(i-1)!i!\cdots (j+d-3)!N^2$;
\item $Nj^*([x])=0$ then $C=(i-1)!i!\cdots (j+d-1)!N^2$.
\end{enumerate}
\end{lemma}
\begin{proof}
For the first case, consider
an element $[x]\in F^{i-1}_{\gamma}/F^{i}_{\gamma}K_j(Z,{\mathbf Z}/p^n)$ for
$x\in F^{i-1}_{\gamma}K_j(Z,{\mathbf Z}/p^n)$  such that $Ni_*[x]=0$, i.e., $Ni_*(x)\in F^{i+1}_{\gamma}K_j(X,{\mathbf Z}/p^n)$. By Adams-Riemann-Roch without denominators
from Lemma \ref{RRK} we have
$$i_*\gamma^{i+1}(\sn,Nx)=\gamma^{i+1}(i_*(Nx))=(-1)^ii!i_*(Nx) 
\quad \mod F^{i+2}_{\gamma}K_j(X,{\mathbf Z}/p^n).
$$
Since $\gamma^{i+1}(\sn,Nx)\in F^{i}_{\gamma}K_j(Z,{\mathbf Z}/p^n)$, we get 
$$[(-1)^ii!Nx]=[(-1)^ii!Nx-\gamma^{i+1}(\sn,Nx)]\quad \text{and}\quad i_*((-1)^ii!Nx-\gamma^{i+1}(\sn,Nx))\in
F^{i+2}_{\gamma}K_j(X,{\mathbf Z}/p^n)$$
Since $F^{j+d+1}_{\gamma}K_j(X,{\mathbf Z}/p^n)=0$, by induction, we get that $[i!(i+1)!\cdots (j+d-1)!Nx]=[y]$ for $y\in F^{i-1}_{\gamma}K_j(Z,{\mathbf Z}/p^n)$ such that $i_*(y)=0$. By the localization sequence (\ref{locseq})
$y=\partial (w)$ for $w\in K_{j+1}(U,{\mathbf Z}/p^n)$. We also have $\gamma^i(\sn,y)=(-1)^{i-1}(i-1)!y$
modulo $F^{i}_{\gamma}K_j(Z,{\mathbf Z}/p^n)$ \cite[0, Appendice, Prop. 1.5]{S6}. From Adams-Riemann-Roch without denominators it now follows that
$$[\gamma^i(\sn,y)]=[\gamma^i(\sn,\partial (w))]=[\partial (\gamma^i(w))].
$$
Since clearly $\gamma^i(w)\in F^i_{\gamma}K_{j+1}(U,{\mathbf Z}/p^n)$, we have that 
$[(i-1)!y]$ is a coboundary for $\partial$. Summing up the above, $[(i-1)!i!\cdots (j+d-1)!N^2(d,i,j)N^2(d,i,j+1)]$ is a coboundary, as wanted.

  For the second case, consider
an element $[x]\in F^{i}_{\gamma}/F^{i+1}_{\gamma}K_j(U,{\mathbf Z}/p^n)$ for
$x\in F^{i}_{\gamma}K_j(U,{\mathbf Z}/p^n)$  such that $N\partial [x]=0$, i.e., $N\partial (x)\in F^{i}_{\gamma}K_{j-1}(Z,{\mathbf Z}/p^n)$. As above, Adama-Riemann-Roch without denominators implies that there exists an element
$w\in F^{i+1}_{\gamma}K_j(U,{\mathbf Z}/p^n)$ such that $i!\partial (Nx)=\partial (w)$ modulo $F^{i}_{\gamma}K_j(Z,{\mathbf Z}/p^n)$. Since $F^{j+d-1}_{\gamma}K_{j-1}(Z,{\mathbf Z}/p^n)$, this gives by induction that 
$$i!\cdots (j+d-3)!\partial (Nx)=\partial (w),\quad\text{for }w\in F^{i+1}_{\gamma}K_j(U,{\mathbf Z}/p^n).
$$
Set $z=i!\cdots (j+d-3)!\partial (Nx)-w$. We have $[i!\cdots (j+d-3)!\partial (Nx)]=[z]$ 
and $\partial (z)=0$. By the localization sequence (\ref{locseq})
$z=j^*(y)$ for $y\in K_{j}(X,{\mathbf Z}/p^n)$. We have $[\gamma^i(z)]=[(-1)^{i-1}(i-1)!z]$. Hence
$[(-1)^{i-1}(i-1)!z]=[j^*(\gamma^i(y))]$, i.e., 
$[(-1)^{i-1}(i-1)!z]$ is in the image of $F^i_{\gamma}K_{j}(X,{\mathbf Z}/p^n)$ by $j^*$. Summing it all up, we get that $[(i-1)!i!\cdots (j+d-3)!N^2x]$ is a coboundary, as wanted. 

   Because the restriction $j^*$ is compatible with $\gamma$-operations, the third case is proved in a similar manner. Consider
an element $[x]\in F^{i}_{\gamma}/F^{i+1}_{\gamma}K_j(X,{\mathbf Z}/p^n)$ for
$x\in F^{i}_{\gamma}K_j(X,{\mathbf Z}/p^n)$  such that $Nj^* [x]=0$, i.e., $Nj^* (x)\in F^{i+1}_{\gamma}K_{j}(U,{\mathbf Z}/p^n)$. We have
$$
(-1)^ii!Nj^*(x)=\gamma^{i+1}(j^*(Nx))=j^*(\gamma^{i+1}(Nx)) \quad\mod F^{i+2}_{\gamma}K_{j}(U,{\mathbf Z}/p^n).
$$
Hence $[Nx]=[Nx-\gamma^{i+1}(Nx)]$ and $j^*(Nx-\gamma^{i+1}(Nx))\in F^{i+2}_{\gamma}K_{j}(U,{\mathbf Z}/p^n)$. 
Since $F^{j+d+1}_{\gamma}K_{j}(U,{\mathbf Z}/p^n)$, inductively this implies that $
[i!(i+1)!\cdots (j+d-1)!Nx]=[z]$. and $j^*(z)=0$. Now, by the localization sequence (\ref{locseq}),
$z=i_*(y)$ for $y\in K_j(Z,{\mathbf Z}/p^n)$. By Adama-Riemann-Roch without denominators we have
$$
[i_*\gamma^{i}(\sn,y)]=[\gamma^{i}(i_*y)]=[\gamma^{i}(z)]=[(-1)^{i}(i-1)!z].
$$
Hence $[(-1)^{i}(i-1)!z]$ is a coboundary for $i_*$. Summing up, we get that
$[(i-1)!i!\cdots (j+d-1)!N^2x]$ is a coboundary, as wanted.
\end{proof}
\section{Syntomic Chern classes}
\subsection{Classical syntomic Chern classes}
   We will briefly review the construction and basic properties of classical syntomic Chern classes. 
 For details the reader can consult \cite[2.3]{N4}, \cite[2.2]{N2}, and \cite[2.3]{N10}.

  For $i\geq 0$ and a scheme $X$ flat over $W(k)$, there are
 functorial
and compatible families of syntomic Chern classes
\begin{align*}
c_{i,j}^{\synt}:\quad  & K_j(X)  \rightarrow H^{2i-j}(X,S_n(i))
\quad\text{for }j\geq 0,\\
\overline{c}_{i,j}^{\synt}:\quad  & K_j(X,{\mathbf Z}/p^n)
  \rightarrow H^{2i-j}(X,S_n(i))\quad\text{for }j\geq 2,
\end{align*}
that are also compatible with the crystalline Chern classes in
$H^{2i-j}_{\crr}(X_n,\so_{X_n}) $ via the
canonical map $H^{2i-j}\!(X_n,S_n\!(i))\!\to\!
H^{2i-j}_{\crr}(\!X_n,J^{<i>}_{X_n})$. Similarly, we
have syntomic Chern classes in $H^{2i-j}(X,S^{\prime}_n(i))$. 

   Recall the
construction  of the  classes $c_{i,j}^{\synt}$ and $\overline{c}_{i,j}^{\synt}$.
First, one constructs universal classes
$C^{\synt}_{i,l}\in H^{2i}(BGL_l,S_n(i))$, $C^{\synt}_{i,l}\in H^{2i}(BGL_l,S^{\prime}_n(i))$.
For $l\geq i$, $i\geq 0$, one defines $$C^{\synt}_{i,l}=x_i\in
H^{2i}(BGL_l,S_n(i)), \quad C^{\synt}_{i,l}=x_i\in
H^{2i}(BGL_l,S^{\prime}_n(i)).$$ By construction these classes
are compatible with the crystalline classes. Recall (see formulas (\ref{symbol})) that
classes $x_1\in
H^{2}(BGL_l,S_n(1))\hookrightarrow H^{2}(BGL_l,S^{\prime}_n(1))$
have a direct definition via symbol maps.

   The classes $C^{\synt}_{i,l}\in H^{2i}(BGL_l,S_n(i))$, $i\geq 0$,  yield compatible universal classes (see \cite[p. 221]{Gi})
$C^{\synt}_{i,l}\in H^{2i}(X,GL_l(\so_X),S_n(i))$, where the last group is the $GL_l(\so_X)$-cohomology of $X$ with values in $S_n(i)$ (equipped with the trivial action of $GL_l(\so_X)$). 
 Hence
 a natural map of pointed simplicial presheaves on $X$,
$$C^{\synt}_i:BGL(\so_X)\to \sk(2i,\tilde{S}_n(i){}_{X}),$$
where $\sk $ is the Dold--Puppe functor of
$\tau_{\leq 0}\tilde{S}_n(i){}_{X}[2i]$
and $\tilde{S}_n(i){}_{X}$ is an
injective resolution of $S_n(i){}_{X}$.
These classes induce Chern class maps $$C^{\synt}_i: K_X\to \sk(2i,\tilde{S}_n(i){}_{X}),\quad i\geq 0.$$
 Here we wrote $K_X$ for the presheaf of simplicial sets $U\mapsto K(U)_0$. 

   We can now define the  total Chern class map
$$
C^{\synt}_{X}:\quad K_X\to \prod_{i\geq 0}\sk(2i,\tilde{S}_n(i){}_{X})
$$ 
by putting the map $C^{\synt}_i$ in degree $i$. Similarly, we get total Chern class maps 
\begin{equation}
 \label{trivial}
C^{\synt}_{X}:\quad K_X\to \prod_{i\geq 0}\sk(2i,\tilde{S}^{\prime}_n(i){}_{X}).
\end{equation}
Notice that the Chern class map $C^{\synt}_0$ is a composition of the rank map $rk: K_X\to {\mathbf Z}$  with the natural map 
${\mathbf Z}\to S_n(0)=S^{\prime}_n(0)=(\so^{\crr}_n)^{\phi=1}$. The related augmented total Chern class maps $\wt{C}^{\synt}_X$ are defined by replacing $C^{\synt}_0$ with the
 rank map $rk: K_X\to {\mathbf Z}$. 
 
   We list the following properties of the  total Chern class maps. They are proved using the embedding $H^{2i}(BGL_l,S_n(i))\hookrightarrow H^{2i}_{\dr}(BGL_{l,n})$ 
   and the properties of de Rham classes.
\begin{lemma}(\cite[Lemma 2.1]{N4})
\label{universal}
\begin{enumerate}
\item The syntomic total Chern class is functorial, i.e., for a map $f: Y\to X$ of flat schemes over $W(k)$
the following diagram of presheaves of simplicial sets on $X$ commutes in the homotopy category.
$$
\xymatrix{
K_X\ar[d]^{f^*}\ar[r]^-{C^{\synt}_{X}} & \prod_{i\geq 0}\sk(2i,\tilde{S}_n(i){}_{X})\ar[d]^{f^*}\\
\R f_*K_Y\ar[r]^-{C^{\synt}_{Y}} & \R f_*\prod_{i\geq 0}\sk(2i,\tilde{S}_n(i){}_{Y})
}
$$
\item The syntomic total Chern class is compatible with addition, i.e., the following diagram is
 commutes in the homotopy category.
$$
 \xymatrix{K_X\times K_X\ar[r]^-{+}\ar[d]^{C_{X}^{\synt}\times
 C_{X}^{\synt}}  & K_X\ar[d]^{C_{X}^{\synt}} \\
 \prod_{i\geq 0}\sk(2i,\wt{S}_n(i)_X)\times\prod_{i\geq 0}\sk(2i,\wt{S}_n(i)_X)\ar[r]^-{\ast}&
\prod_{i\geq 0}\sk(2i,\wt{S}_n(i)_X).
 }
$$
\item The syntomic augmented total Chern class is compatible with products, i.e., the following diagram 
 commutes in the homotopy category.
$$
 \xymatrix{K_X\wedge K_X\ar[r]^-{\wedge}\ar[d]^{\wt{C}_{X}^{\synt}\wedge
 \wt{C}_{X}^{\synt}}  & K_X\ar[d]^{\wt{C}_{X}^{\synt}} \\
 {\mathbf Z}\times \prod_{i\geq 1}\sk(2i,\wt{S}_n(i)_X)\wedge {\mathbf Z}\times \prod_{i\geq 1}\sk(2i,\wt{S}_n(i)_X)\ar[r]^-{\star}&
{\mathbf Z}\times \prod_{i\geq 1}\sk(2i,\wt{S}_n(i)_X).
 }
$$
Here $\star$ is the Grothendieck product \cite[2.27]{Gi} defined via certain universal polynomials with integral coefficients. 

  Similarly for total Chern classes with values in $S^{\prime}_n(i)$-cohomology.
\end{enumerate}
\end{lemma}
\begin{lemma}
\label{lambda-structure} 
The augmented total Chern class 
$$
\wt{C}_{X}^{\synt} :\quad  K_0(X)\to {\mathbf Z}\times \{1\}\times \prod_{i\geq 1}H^{2i}(X,{S}_n(i))
$$
is a morphism of $\lambda$-rings.
\end{lemma}
\begin{proof}
This was shown in the proof of Lemma 2.1 in \cite{N4}.
\end{proof}

   The characteristic classes 
\begin{align*}
\overline{c}^{\synt}_{i,j}:\quad &  K_j(X,{\mathbf Z}/p^n)
  \rightarrow H^{2i-j}(X,S_n(i)),\quad j\geq 2,\\
  c^{\synt}_{i,j}:\quad &  K_j(X)
  \rightarrow H^{2i-j}(X,S_n(i)),\quad j\geq 0,
\end{align*}
 are 
defined \cite[2.22]{Gi}
as the composition
\begin{align*}
K_j(X,{\mathbf Z}/p^n) & \to
H^{-j}(X,K_X;{\mathbf Z}/p^n)
\to H^{-j}(X,BGL(\so_X)^+;{\mathbf Z}/p^n)\\
&  \stackrel{C^{\synt}_i}{\longrightarrow}  H^{-j}(X,
\sk(2i,\tilde{S}_n(i){}_{X});{\mathbf Z}/p^n)
 \stackrel{f}{\rightarrow}  H^{2i-j}(X,S_n(i)),
\end{align*}
where  $BGL(\so_X)^+$ is the (pointed) simplicial presheaf on
$X$ associated to the $+\,$- construction.  The map $f$ is defined as the composition
{\small\begin{align*} H^{-j}(X, \sk(2i,\tilde{S}_n(i){}_{X}
);{\mathbf Z}/p^n)  &  = \pi_j(\sk(2i,\tilde{S}_n(i)(X)),{\mathbf
Z}/p^n)) \!\stackrel{h_j}{\rightarrow}\!
H_j(\sk(2i,\tilde{S}_n(i)(X)),{\mathbf Z}/p^n))\\
 &  \to H_j(\tilde{S}_n(i)(X)[2i])= H^{2i-j}(X,S_n(i)),
\end{align*}}
where $h_j$ is the Hurewicz morphism.

This gives mod $p^n$ Chern classes with values in $H^{*}(X,S_n(*))$. Those 
with values in $H^{*}(X,S^{\prime}_n(*))$ and the  ones 
from integral $K$-theory are defined in an analogous way. All of the above also works (with basically identical proofs) if we replace $X$ with a finite simplicial scheme $X$, flat over $W(k)$.

 As we have shown in \cite[Lemma 2.1]{N2} the above Chern classes for schemes have the properties listed below. Again they 
hold as well  for finite simplicial schemes $X$, flat over $W(k)$.
\begin{lemma}
\label{prop}Let $X$ be a finite simplicial scheme that is flat over $W(k)$.
The syntomic Chern classes are functorial in $X$ and have the following properties.
\begin{enumerate}
\item $c_{ij}^{\,\synt}$, for $j>0$, is a group homomorphism.
\item $c_{{\scriptscriptstyle\bullet},0}^{\synt}(a+b)=c_{{\scriptscriptstyle\bullet},0}^{\synt}(a)c_{{\scriptscriptstyle\bullet},0}^{\synt}(b)$, for $a,b\in K_0(X)$.
\item $\overline{c}_{ij}^{\,\synt}$, for $j\geq 2$  
is a group homomorphism unless $j=2$ and $p=2$.
\item $\overline{c}_{ij}^{\,\synt}$ are compatible with the reduction 
maps $S_n(i)\to S_m(i)$, $n\geq m$.\\
Moreover, if $X$ is regular,  then
\item Let $p$ be odd or $p=2$, $n\geq 2$ and $l,q\neq 2$. 
If $\alpha\in K_l(X,{\mathbf Z}/p^n)$ and 
$\alpha'\in K_q(X,{\mathbf Z}/p^n)$, then
$$
\overline{c}_{ij}^{\,\synt}(\alpha\alpha')=-\sum_{r+s=i}
\frac{(i-1)!}{(r-1)!(s-1)!}
\overline{c}_{rl}^{\,\synt}(\alpha)\overline{c}_{sq}^{\,\synt}(\alpha'),
$$
assuming that $l,q\geq 2$, $l+q=j$, $2i\geq j$, $i\geq 0$.
\item If $\alpha\in F^j_{\gamma}K_0(X)$, $j\neq 0$,
 and 
$\alpha'\in F^k_{\gamma}K_q(X,{\mathbf Z}/p^n)$, $q\geq 2$, is such that
 $\overline{c}^{\,\synt}_{lq}(\alpha')=0$ for $l\neq k$,
then
\begin{equation*}\overline{c}_{j+k,q}^{\,\synt}(\alpha\alpha')=
-\frac{(j+k-1)!}{(j-1)!(k-1)!}c_{j0}^{\,\synt}(\alpha)
\overline{c}_{kq}^{\,\synt}(\alpha'),
\end{equation*}
assuming that $p\neq 2$ or $q>2$.
\item If $X$ is a scheme, the above multiplication formulas  hold also for  $p=2$, $n\geq 4$,
$q=2$ and $\alpha'$ such that $\partial \alpha' \in K_1(X)$ belongs to ${\so_K}^*$. 
\item The integral Chern class maps $c_{i0}^{\synt}$ restrict to zero on 
$F^{i+1}_{\gamma}K_0(X)$.
\item The Chern class maps $\overline{c}_{ij}^{\synt}$ restrict to zero on 
$F^{i+1}_{\gamma}K_j(X,{\mathbf Z}/p^n)$, $j\geq 2$, unless $j=2$, $p=2$.
\end{enumerate}
\end{lemma}
\begin{proof}
The proof is basically a translation of the proof of Lemma 2.1 in \cite{N2}
into the language of spaces used in \cite{GS}.
\end{proof}
\begin{remark}
 Let $X$ be a scheme over $K$ or $\ovk$. In an analogous way one constructs the \'etale total Chern classes
$$C_{X}^{\eet} :\quad  K_X\to \prod_{i\geq 0}\sk(2i,\R \ve_*{\mathbf Z}/p^n(i)_X),
$$
where $\ve: X_{\eet}\to X_{\Zar}$ is the natural projection, 
that have all the properties listed in Lemma \ref{universal}. They induce the characteristic classes
\begin{align*}
\overline{c}^{\eet}_{i,j}:\quad &  K_j(X,{\mathbf Z}/p^n)
  \rightarrow H^{2i-j}(X,{\mathbf Z}/p^n(i)),\quad j\geq 2,\\
  c^{\eet}_{i,j}:\quad &  K_j(X)
  \rightarrow H^{2i-j}(X,{\mathbf Z}/p^n(i)),\quad j\geq 0,
\end{align*}
that have all the properties listed in Lemma \ref{prop}.
\end{remark}
\subsection{Truncated syntomic Chern classes}
\label{truncated-kolo}

We will show now that there exists Chern classes into truncated syntomic as well as syntomic-\'etale cohomology that are compatible with the classical syntomic Chern classes. 
We define the universal classes $C^{\synt}_{i,a}\in H^{2i}(BGL_a/W(k),S_n(i)_{\Nis})$ as the unique classes mapping to the classical  syntomic universal classes. This can be done by Proposition \ref{truncated}. They induce Chern class  and  total Chern class maps
$$C^{\synt}_i:\quad K_X\to \sk(2i,\tilde{S}_n(i)_{\Nis,X}),\quad i\geq 0; \quad C^{\synt}_X: K_X\to\prod_{i\geq 0}\sk(2i,\tilde{S}_n(i)_{\Nis,X})
$$
As before, we get the characteristic classes
\begin{align*}
\overline{c}^{\synt}_{i,j}:\quad &  K_j(X,{\mathbf Z}/p^n)
  \rightarrow H^{2i-j}(X,S_n(i)_{\Nis}),\quad j\geq 2,\\
  c^{\synt}_{i,j}:\quad &  K_j(X)
  \rightarrow H^{2i-j}(X,S_n(i)_{\Nis}),\quad j\geq 0,
\end{align*}
 Via the map $\tau: H^{2i}(BGL_a,S_n(i)_{\Nis})\to H^{2i}(BGL_a,S^{\prime}_n(i)_{\Nis})$ we get induced
universal classes, total Chern maps and characteristic classes with values in $S^{\prime}_n(i)_{\Nis}$.
\begin{lemma}
\label{Zariskiprop}
The above truncated syntomic total Chern class maps and the induced characteristic classes satisfy analogs of Lemma \ref{universal}, Lemma \ref{lambda-structure},   and Lemma \ref{prop}.
\end{lemma}
\begin{proof}We start with Lemma \ref{universal} describing properties of the universal total Chern class maps. Property (1) follows from functoriality of truncated syntomic universal classes.
Property (2) follows from a universal Whitney sum formula. More specifically, consider the universal short exact sequence 
$$0\to E^a_{{\scriptscriptstyle\bullet}}\to E^{a+b}_{{\scriptscriptstyle\bullet}}\to E^b_{{\scriptscriptstyle\bullet}}\to 0
$$
of vector bundles over $BGL(a,b)$. Here the middle term is the universal vector bundle and the first and last terms  are classified by the natural maps to  $BGL_a$ and $BGL_b$, respectively.
A classical argument using only the projective space theorem gives the Whitney sum formula in de Rham cohomology
$$C_{{\scriptscriptstyle\bullet}}(E^{a+b}_{{\scriptscriptstyle\bullet}})=C_{{\scriptscriptstyle\bullet}}(E^a_{{\scriptscriptstyle\bullet}})C_{{\scriptscriptstyle\bullet}}(E^b_{{\scriptscriptstyle\bullet}}).
$$
By Proposition \ref{truncated} this induces the Whitney sum formula in the truncated syntomic cohomology.
 
 For property (3), again  one needs to consider only   the universal situation, i.e., the universal tensor product bundle $E^{a}_{{\scriptscriptstyle\bullet}}\otimes E^{b}_{{\scriptscriptstyle\bullet}}$ over $B(GL_a\times GL_b)$. Splitting principle gives the formula
 $$\wt{C}_{{\scriptscriptstyle\bullet}}(E^{a}_{{\scriptscriptstyle\bullet}}\otimes E^{b}_{{\scriptscriptstyle\bullet}})=\wt{C}_{{\scriptscriptstyle\bullet}}(E^{a}_{{\scriptscriptstyle\bullet}})\star\wt{C}_{{\scriptscriptstyle\bullet}}(E^{b}_{{\scriptscriptstyle\bullet}}).
 $$
 in de Rham cohomology. By Proposition \ref{truncated} it carries over to truncated syntomic cohomology.
 
 The analog of Lemma \ref{lambda-structure} is proved in the same way as the original lemma using Proposition \ref{truncated}. Having an analog of Lemma \ref{universal}, the proof of Lemma \ref{prop} for truncated syntomic cohomology carries verbatim from the untruncated case.
 \end{proof}
 
 Similarly, we construct (using Proposition \ref{truncatedetale}) universal Chern classes, total Chern maps, and characteristic classes with values in (truncated) syntomic-\'etale cohomologies $E_n(i)_{\Nis}$, $E^{\prime}_n(i)_{\Nis}$ and  $E_n(i)$, $E^{\prime}_n(i)$. Clearly, so defined  (truncated) syntomic-\'etale total Chern class maps and the induced characteristic classes satisfy analogs of Lemma \ref{universal},  Lemma \ref{lambda-structure},  and Lemma \ref{prop}.
\subsection{Logarithmic syntomic Chern classes}
\label{6.1}
Let $X$ be a semistable scheme over $\so_K^{\times}$ or a semistable scheme over $\so_K$ with a smooth special fiber. Let $D$ be  the horizontal divisor and set $j: X\setminus D\hookrightarrow X$. 
 In this section, we will construct universal 
Chern class maps (in the homotopy category of pointed presheaves of simplicial sets on $X_{\Nis}$)
$$
C^{\synt}_{X\setminus D,i}: j_*K_{X\setminus D}\to \sk(2i, \tilde{E}^{\prime}_n(i)_X( D)),\quad i\geq 0,
$$
where $\wt{E}_n^{\prime}(i)_X( D)_{\Nis}$ is an injective resolution of $E_n^{\prime}(i)_X(D)_{\Nis}$. In what follows we will skip the subscript ${\Nis}$ if confusion does not arise. Up to some universal constants these maps will have all the expected properties.

   Consider the map
$C^{\synt}_{X\setminus D}: j_*K_{X\setminus D} \to
\R j_*\prod_{i\geq 0}\sk(2i,\wt{E}_n^{\prime}(i)_{X\setminus D})$ induced by the total syntomic-\'etale Chern class maps defined in Section \ref{truncated-kolo}.
Lemma \ref{kolo1} below shows that a multiple $[p^{mi}]C^{\synt}_{X\setminus D}$, for a universal constant $m$,  of this map lifts, via the 
 map 
$$\sk(2i,\wt{E}_n^{\prime}(i)_{X}(D))\to  \sk(2i,\R j_*\wt{E}_n^{\prime}(i)_{X\setminus D})\to \R j_*\sk(2i,\wt{E}_n^{\prime}(i)_{X\setminus D}),
$$
to a map 
$$
C^{\synt}_{X\setminus D}(m):\quad j_*K_{X\setminus D}
\to
\prod_{i\geq 0}\sk(2i,\wt{E}_n^{\prime}(i)_{X}(  D)).
$$

\begin{lemma}
\label{kolo1}Let $\sll$ be a complex of pointed presheaves of simplicial sets on $X$ with homotopy presheaves concentrated in nonnegative degrees. 
The kernel and cokernel of the map
$$
\Hom_{\spp}(\sll,\sk(2i,\wt{E}_n^{\prime}(i)_{X}(D)))\to  \Hom_{\spp}(\sll, \R j_*\sk(2i,\wt{E}_n^{\prime}(i)_{X\setminus D}))
$$ 
is annihilated by $p^{Ni}$ for a constant $N=c_1c_2$, where $c_1$ is the constant\footnote{We take $c_1=N$ in the notation of Theorem \ref{input1}.} from Theorem \ref{input1} and $c_2$ is a constant depending only on the dimension $d$ of $X$.  Here the homomorphisms are taken in the homotopy category $\spp$ of  pointed presheaves of simplicial sets on $X_{\Nis}$.
\end{lemma}
\begin{proof}For $i=0$, we have an isomorphism by Corollary \ref{reduction-kolo}. Assume $i\geq 1$. Since the homotopy presheaves of $\sll$ are trivial in negative degrees, we have that $$
  \Hom_{\spp}(\sll, \sk(2i,\R j_*\wt{E}_n^{\prime}(i)_{X\setminus D}))
\stackrel{\sim}{\to } \Hom_{\spp}(\sll, \R j_*\sk(2i,\wt{E}_n^{\prime}(i)_{X\setminus D})).
$$
Hence it suffices to show that the above lemma holds for the map
$$
\Hom_{\spp}(\sll,\sk(2i,\wt{E}_n^{\prime}(i)_{X}(D)))\to  \Hom_{\spp}(\sll, \sk(2i,\R j_*\wt{E}_n^{\prime}(i)_{X\setminus D})).
$$ 

  We have, 
\begin{align*}
\Hom_{\spp}(\sll, & \sk(2i,\wt{E}_n(i)_X(D)))  \simeq \Hom_{\sss}(\sll^a,\sk(2i,\wt{E}_n(i)_X(D)))\\
& \simeq \Hom_{D}(C(\sll^a), E_n(i)_X(D)[2i]),
\end{align*}
where  $\Hom_{\sss}(,)$ refers to homomorphisms in the homotopy category of pointed sheaves of simplicial sets, $\sll^a$ is the complex of sheaves associated to $\sll$,
$C(\sll^a)$ is the sheaf of (normalized) chain complexes associated to $\sll^a$ that we see as a sheaf of cochain complexes by negating the degrees, and $\Hom_{D}(,)$ refers to homomorphisms in the derived category of complexes of sheaves of abelian groups on $X(D)_{\Nis}$. Similarly,
$$\Hom_{\spp}(\sll,  \sk(2i,\R j_*\wt{E}_n(i)_{X\setminus D}))  
\simeq \Hom_{D}(C(\sll^a), j_*\wt{E}_n(i)_{X\setminus D}[2i]).
$$

  Let $\scc(i)$ denote the mapping fiber of the map $\wt{E}_n^{\prime}(i)_{X}(D)\to  j_*\wt{E}_n^{\prime}(i)_{X\setminus D}$. 
We need to show that
  $\Hom_{D}(\cdot, \scc(i))$ is annihilated by $p^{Ni}$ for a constant $N$ described above. It suffices to show that $\scc(i)$ has cohomology annihilated by $p^{c_1ri}$ for the constant $c_1$ as above  and that $\scc(i)$ has cohomological length less than $c_2$ for a constants $c_2$ depending only on $d$.

 The first claim follows from  Corollary \ref{reduction-kolo}.
For the second claim, since Nisnevich topology of $X$ has cohomological dimension less than $d$, it suffices to show that $\scc(i)$ has cohomological length less than $c_3$ for a constant $c_3$ depending only on $d$.  But this is clear since the cohomological dimension of the \'etale topos of $X$ and the length of the (filtered) crystalline cohomology of $X/W(k)$ are bounded by $2d+3$.
 \end{proof}

 Let $m\geq N$. Set $C^{\synt,m}_{X\setminus D}:=[p^{mi}]C^{\synt}_{X\setminus D}(m)$. By the above lemma two different lifts $C^{\synt}_{X\setminus D}(m)$ yield the same class $C^{\synt,m}_{X\setminus D}$.  Similarly, we define the total syntomic-\'etale Chern class maps
$$C^{\synt}_{X\setminus D}:\quad j_*K_{X\setminus D}
\to
\prod_{p-2\geq i\geq 0}\sk(2i,\wt{E}_n(i)_{X}(  D)).
$$No additional constants are needed here. 

  We record the 
following properties of these Chern class maps.
\begin{theorem}
\label{augmented}
\begin{enumerate}
\item The Chern class maps  $C^{\synt,m}_{X\setminus D}$  are compatible with the reduction 
maps $E^{\prime}_n(i)_{X}(D)\to E^{\prime}_{n_1}(i)_X(D)$, $n\geq n_1$.
\item  The Chern class maps $C^{\synt,m}_{X\setminus D}$ are independent of the number $m$ chosen, that is, for two different numbers $m_1<m_2$  the following diagram commutes
\begin{equation}
\label{independencem}
\xymatrix{
G_X(D)  \ar[rd]_-{C^{\synt,m_2}_{X\setminus D}}
\ar[r]^-{C^{\synt,m_1}_{X\setminus D}} & 
\prod_{i\geq 0}\sk(2i, \tilde{E}^{\prime}_{n}(i)_X(D))\ar[d]^{[p^{2(m_2-m_1)i}]}\\
& \prod_{i\geq 0}\sk(2i, \tilde{E}^{\prime}_n(i)_X(D))
}
\end{equation}

\item The Chern class map $C^{\synt,m}_{X\setminus D}$ is functorial for morphisms of log-schemes as above $\pi: Y\to X$ such that $\pi^{-1}(D_X)\subset D_Y$, i.e., 
the following diagram commutes in the homotopy category
$$
\xymatrix{
j_*K_{X\setminus D_X}\ar[r]^-{C^{\synt,m}_{X\setminus D_X}}\ar[d]^{\pi^*} & \sk(2i,\wt{E}_n^{\prime}(i)_{X}( D_X))\ar[d]^{\pi^*}\\
\pi_*j^{\prime}_*K_{Y\setminus D_Y}\ar[r]^-{C^{\synt,m}_{Y\setminus D_Y}} & \R \pi_*\sk(2i,\wt{E}_n^{\prime}(i)_{Y}( D_Y)),
}
$$
where $j^{\prime}: Y\setminus D_Y\hookrightarrow Y$ is the natural open immersion.
 \item The total Chern class map $C^{\synt,m}_{X\setminus D}$ is compatible with addition.  The augmented total Chern class map $\wt{C}^{\synt,m}_{X\setminus D}$ is compatible with products. 
\item The   total Chern class map
$$C^{\synt,m}_{X\setminus D}: \quad j_*K_{X\setminus D}\to \prod_{i\geq 0}\sk(2i, \tilde{E}^{\prime}_n(i)_X(D)),\quad n\geq 1,
$$
is an extension of  the syntomic-\'etale Chern class map
$$[p^{2mi}]C^{\synt}_{X}: \quad K_X\to \prod_{i\geq 0}\sk(2i, \tilde{E}^{\prime}_n(i)_X),\quad n\geq 1.
$$
Two such extensions become equal after multiplication by $[p^{mi}]$.
\end{enumerate}

 Similarly for the Chern classes with values in $E_n(i)$-cohomology.
\end{theorem}
\begin{proof}
The first two claims are immediate from construction. The next two  follow  from Section \ref{truncated-kolo} and Lemma \ref{kolo1}; to control the constants we  use  the fact that the Whitney sum formula and the product formula involve homogenous polynomials of the right degrees. For the fifth claim it remains to show that an extension  of the Chern class maps
      $[p^{2mi}]C^{\synt}_{X}$ to $X\setminus D$ is unique up to multiplication by $[p^{mi}]$. To do this, consider  the localization homotopy cofiber sequence
 $$G(D)\stackrel{i_*}{\to} G(X) \to G(X\setminus D)
 $$
and apply Lemma \ref{keylemma} below  to $\sll=i_*G_D[1]$.
\end{proof}
\begin{lemma}
\label{keylemma}
Let $\sll$ be a complex of pointed presheaves of simplicial sets on $D_1$ with homotopy presheaves concentrated in nonnegative degrees. 
\begin{enumerate}
\item For $0\leq i\leq p-2$, we have 
$$\Hom_{\spp}(i_*\sll,\sk(2i,\wt{E}_n(i)_X(D)))=0.
$$
\item 
For $ 0\leq i$, the group
$$\Hom_{\spp}(i_*\sll,\sk(2i,\wt{E}^{\prime}_n(i)_X(D)))
$$
is annihilated by $p^{Ni}$, where $N$ is the constant from Lemma \ref{kolo1}.
\end{enumerate}
\end{lemma}
\begin{proof} We start with the first claim.
 Take $p-2\geq i\geq 0$.  We have 
\begin{align*}
\Hom_{\spp}(i_*\sll, & \sk(2i,\wt{E}_n(i)_X(D)))  \simeq \Hom_{\sss}(i_*\sll^a,\sk(2i,\wt{E}_n(i)_X(D)))\\
& \simeq \Hom_{D}(i_*C(\sll^a), E_n(i)_X(D)[2i])\simeq \Hom_{D}(C(\sll^a). \R i^!E_n(i)_X(D)[2i]).
\end{align*}

  We claim that $\R i^!E_n(i)_X(D)=0$. Indeed, for $X$ semistable over ${\so_K}^{\times}$, by Theorem \ref{keylemma11}, we have a quasi-isomorphism
  $E_n(i)_X(D)_{\Nis}\stackrel{\sim}{\to}\tau_{\leq i}\R \varepsilon_*\R j_*{\mathbf Z}/p^n(i)$, where $\varepsilon:X_{\eet}\to X_{\Nis}$ is the projection and 
  $j:X_K\setminus D_{K}\hookrightarrow X$ is the natural open immersion.
  But the Beilinson-Lichtenbaum conjecture implies that
\begin{align*}
\tau_{\leq i}\R \varepsilon_*\R j_*{\mathbf Z}/p^n(i)  \simeq \tau_{\leq i}\R j_*\R \varepsilon_*{\mathbf Z}/p^n(i)\simeq \tau_{\leq i}\R j_*\tau_{\leq i}\R \varepsilon_*{\mathbf Z}/p^n(i)
   & \simeq \tau_{\leq i}\R j_*{\mathbf Z}/p^n(i)_{\M}
\end{align*}
and we have 
$$\tau_{\leq i}\R j_*{\mathbf Z}/p^n(i)_{\M}\simeq \tau_{\leq i}j_*{\mathbf Z}/p^n(i)_{\M}\simeq j_*{\mathbf Z}/p^n(i)_{\M}\simeq \R j_*{\mathbf Z}/p^n(i)_{\M}.$$
Hence 
$$\R i^!E_n(i)_X(D)\simeq \R i^!\R j_*{\mathbf Z}/p^n(i)_{\M}=0,
$$
as wanted.

  If $X$ is a semistable scheme over $\so_K$ with a smooth special fiber,  by Theorem \ref{keylemma10}, we have a quasi-isomorphism
  $$E_n(i)_X(D)_{\Nis}\stackrel{\sim}{\to}\tau_{\leq i}\R \varepsilon_*\R j^{\prime}_*{\mathbf Z}/p^n(i)_{\M},$$
   where  
  $j^{\prime}:X\setminus D\hookrightarrow X$ is the natural open immersion.
  But by the Beilinson-Lichtenbaum conjecture in mixed characteristic \cite[Theorem 1.2.2]{GD} that states that we have a quasi-isomorphism
$$
\tau_{\leq i}\R \varepsilon_*{\mathbf Z}/p^n(i)_{\M}
    \simeq {\mathbf Z}/p^n(i)_{\M}
$$ 
and by the fact that 
$\tau_{\leq i}{\mathbf Z}/p^n(i)_{\M}\stackrel{\sim}{\to}{\mathbf Z}/p^n(i)_{\M}$  \cite[Cor. 4.2]{GD}, we have
\begin{align*}
\tau_{\leq i}\R \varepsilon_*\R j^{\prime}_*{\mathbf Z}/p^n(i) & \simeq \tau_{\leq i}\R j^{\prime}_*\R \varepsilon_*{\mathbf Z}/p^n(i)\simeq \tau_{\leq i}\R j^{\prime}_*\tau_{\leq i}\R \varepsilon_*{\mathbf Z}/p^n(i)
    \simeq \tau_{\leq i}\R j^{\prime}_*{\mathbf Z}/p^n(i)_{\M}\\
    & \simeq \tau_{\leq i}j^{\prime}_*{\mathbf Z}/p^n(i)_{\M}\simeq j^{\prime}_*{\mathbf Z}/p^n(i)_{\M}\simeq \R j^{\prime}_*{\mathbf Z}/p^n(i)_{\M}.
\end{align*}
 Hence 
$$\R i^!E_n(i)_X(D)\simeq \R i^!\R j^{\prime}_*{\mathbf Z}/p^n(i)_{\M}=0,
$$
as wanted.

  For the second claim of the lemma, we can assume that  $i\geq 1$ (the case of $i=0$ being treated above). As above we compute that
 $$\Hom_{\spp}(i_*\sll,  \sk(2i,\wt{E}^{\prime}_n(i)_X(D)))  \simeq \Hom_{D}(C(\sll^a), \R i^!E^{\prime}_n(i)_X(D)[2i]).
 $$
 We need to show that
  $\Hom_{D}(C(\sll^a), \R i^!E^{\prime}_n(i)_X(D)[2i])$ is annihilated by $p^{Ni}$ for a constant
  $N$ as above. It suffices to show that $\R i^!E^{\prime}_n(i)_X(D)$ has cohomology annihilated by $p^{Ni}$, $N$ as in Lemma \ref{kolo1},  and that $\R i^!E^{\prime}_n(i)_X(D)$ has cohomological length less than ${c_1}$ for a constant  $c_1$ depending only on $d$.

 Concerning the cohomological length of $\R i^!E^{\prime}_n(i)_X(D)$, since Nisnevich topology of $X$ has cohomological dimension less than $d$, it suffices to show that $E^{\prime}_n(i)_X(D)$ has cohomological length less than ${c_1}$. By (\ref{Niis}) we have a distinguished triangle (for $i_0: X_0\hookrightarrow X$)
$$j_{\Nis !}\R \varepsilon_*\R j^{\prime}_*{\mathbf Z}/p^n(i)^{\prime}\to \R \varepsilon_*E^{\prime}_n(i)_X(D)\to i_{0 *}\R \varepsilon_*S^{\prime}_n(i)_X(D).
$$ It suffices thus to show that both 
$\R \varepsilon_*\R j^{\prime}_*{\mathbf Z}/p^n(i)^{\prime}$ and 
$ \R \varepsilon_*S^{\prime}_n(i)_X(D)$ have cohomological length less than $c_1$. But this is clear since the first complex has length bounded by $2d$ - the cohomological dimension of the \'etale topos of $X$ and the second complex has length bounded by $2d+3$ ($2d+2$ being the length of the (filtered) crystalline cohomology of $X/W(k)$).

  It remains to show that $\R i^!E^{\prime}_n(i)_X(D)$ has cohomology annihilated by $p^{Ni}$. Let us look first at the case when $X$ is semistable  over $\so_K$ with a smooth special fiber. 
By  Theorem \ref{keylemma10} we have a distinguished triangle
\begin{equation}
\label{kwak}
\scc \to E^{\prime}_n(i)_X(D) \to \tau_{\leq i}\R j^{\prime}_*{\mathbf Z}/p^n(i)_{\M},
\end{equation}
where $\scc$ is a complex whose cohomology is annihilated by $p^{Ni}$.  Arguing as above we have that 
$\R i^!\tau_{\leq i}\R j^{\prime}_*{\mathbf Z}/p^n(i)_{\M}=0$. Hence $
\R i^!\scc \stackrel{\sim}{\to} \R i^!E^{\prime}_n(i)_X(D) $. Since Nisnevich topos of $X$ has cohomological dimension less than $d$ it suffices now to show that the complex $\scc$ has cohomological length less than $c_1$ for a constant $c_1$ depending only on $d$. By the distinguished triangle (\ref{kwak}), since we have shown this for $ E^{\prime}_n(i)_X(D) $,  it remains to show it for $ \tau_{\leq i}\R j^{\prime}_*{\mathbf Z}/p^n(i)_{\M}$. Hence it suffices to show that for a log-scheme $T$, smooth over ${\so_K}$
the cohomology groups  $H^*(T_{\tr,\Nis},{\mathbf Z}/p^n(i)_{\M})$ vanish in degree  larger than twice the dimension $d$ of $T$. But this is clear since Nisnevich topos has cohomological dimension $d$ and cohomology sheaves of ${\mathbf Z}/p^n(i)_{\M}$ are trivial above degree $d$ \cite[Theorem 1.2]{GD}.

  Let now $X$ be semistable over ${\so_K}^{\times}$. Arguing as above using Theorem \ref{keylemma11} we reduce to showing that $ \tau_{\leq i}\R j^{\prime}_*{\mathbf Z}/p^n(i)_{\M}$ has cohomological length less than $c_1$ - a constant depending only on $d$. This can be treated as above and we have finished the proof of our lemma.
\end{proof}

 Having the universal Chern class maps, as before, we get the (compatible) characteristic classes \begin{align*}
\overline{c}^{\synt,m}_{i,j}: \quad & K_j(X\setminus D,{\mathbf Z}/p^n) \to  H^{2i-j}(X,E^{*}_n(i)_X(D)),\quad
j\geq 2,E^*=E,E^{\prime}\\
c^{\synt,m}_{i,j}:\quad &  K_j(X\setminus D) \to  H^{2i-j}(X,E^{*}_n(i)_X(D)),\quad
j\geq 0, E^*=E,E^{\prime}.
\end{align*}
By composition with the maps $E^{\prime}_n(i)\to S^{\prime}_n(i)$ and $
  E_n(i)\to S_n(i)$ we obtain syntomic Chern classes with values in $S^{\prime}_n(i)$ and $S_n(i)$. Clearly, all these Chern classes satisfy analogs of Lemma \ref{prop}.
\begin{corollary}Set $U=X\setminus D$. 
The  syntomic Chern classes
 $$
 \overline{c}^{\synt,m}_{ij}:\quad K_j(U, \Z/p^n)\to H^{2i-j}_{\synt}(X,\sss^{\prime}_n(i)),\quad j\geq 2,
 $$
  are compatible, via the period maps $\alpha_{*,*}$ of Fontaine-Messing, 
  with \'etale Chern classes, i.e., the following diagram commutes
   $$
   \xymatrix{
   K_j(U, \Z/p^n)\ar[d]^{\jmath^*} \ar[r]^-{\overline{c}^{\synt,m}_{ij}}  & H^{2i-j}_{\synt}(X,\sss^{\prime}_n(i)) \ar[d]^{\alpha_{2i-j,i} } \\
    K_j(U_K, \Z/p^n) \ar[r]^-{p^{(m+1)i}\overline{c}^{\eet}_{ij}}  & H^{2i-j}_{\eet}(U_K,\Z/p^n(i)^{\prime}),   }
   $$
   where $\jmath:U_K\hookrightarrow U$ is the natural open immersion. We have analogous statement for $j\leq p-2$ and $\sss$-cohomology, where no twist is necessary.
  
\end{corollary}
\begin{proof}We will present the argument for $\sss^{\prime}$-cohomology, the one for $\sss$-cohomology being analogous.
It suffices to show that the total universal syntomic-\'etale and \'etale Chern classes are compatible, i.e.,  that the following diagram commutes
$$
\xymatrix@C=55pt{
j_*K_{X\setminus D}\ar[r]^-{C^{\synt,m}_{X\setminus D}}\ar[d] &  \prod_{i\geq 0}\sk(2i, \tilde{E}^{\prime}_n(i)_X(D))\ar[d]^{[\alpha_i]}\\
j_{\Nis,*}j_*K_{X_K\setminus D_K}\ar[r]^-{[p^{(m+1)i}]C^{\eet}_{X_K\setminus D_K}} & \R j_{\Nis,*}\R j_*\prod_{i\geq 0}\sk(2i, \R\varepsilon_*\Z/p^n(i)^{\prime}_{X_K\setminus D_K}).
}
$$
It follows from functoriality of  syntomic-\'etale Chern classes and period maps that it suffices to show this in the case when the divisor $D$ is trivial, i.e., in the case of classical Chern classes. But this reduces to showing that
 the period map $\alpha_i: H^{2i}(BGL_l,E_n^{\prime}(i)_{\Nis})\to H^{2i}_{\eet}(BGL_{l,F},\Z/p^n(i)^{\prime})$ maps the syntomic-\'etale universal class $x^{\synt}_i$ to the \'etale universal class $p^ix^{\eet}_i$. But this we proved in Corollary \ref{syntet}.
\end{proof}

\begin{remark}
Theorem \ref{augmented}  allows us to define logarithmic Chern class maps (having all the described above properties) with values in 
 log-crystalline cohomology $H^*_{\crr}(X_n,\sj^{[*]}_{X_n})$ for $X$ semistable over $\so_K$ with a smooth special fiber.
\end{remark}

\section{Chern class maps and  Gysin sequences}We will show in this section that the  syntomic universal Chern class maps are compatible with the canonical Gysin sequences.

Let $X$ be a semistable scheme  over $\so_K^{\times}$ or a semistable scheme over $\so_K$ with a smooth special fiber. Let $D^{\prime}$ be the horizontal divisor. Assume that  $D^{\prime}=\cup_{i=1}^{m}D_i$, $m\geq 1$,  is a union of $m$  irreducible components $D_i$.
 Note that each scheme $D_i$ with the induced log-structure is of the same type as $X$  (with  at most $m-1$  components in the divisor at infinity). 
 Set  
$D=\cup_{i=2}^{m}D_i$, $Y=D_1$, $i:Y\hookrightarrow X$. 
The pairs $(X,D)$ and $(Y,D_Y)$ are of the same type as the pair $(X,D^{\prime})$ we started with but  with at most $m-1$ irreducible divisors at infinity.  
\subsection{Basic properties of syntomic cohomology}
   In this section, we will list  several basic properties of syntomic and syntomic-\'etale cohomologies. 
\subsubsection{Gysin sequences}
We start with Gysin sequences. We have the following localization sequences.
\begin{lemma}
\label{localization}
 For $r\geq 1$, there exist the following 
distinguished triangles 
\begin{align*}
i_*S^{i+1}_n(r-1)_Y(D_Y)[-2] & \stackrel{i_!}{\to }
S^{i}_n(r)_X( D)\to S^{i}_n(r)_X(Y\cup D),\quad i\geq 0;\\
i_*S_n(r-1)_Y(D_Y)[-2] & \stackrel{i_!}{\to }
S_n(r)_X( D)\to S_n(r)_X(Y\cup D),\quad r\leq p-1.
\end{align*}
There are analogous distinguished triangles for the complexes $E_n(r)$, $r\leq p-2$, $E^{\prime}_{n}(r)$, and $E^{1}_n(r)$. Moreover, for $r\geq 1$, 
there exist the following 
distinguished triangles 
\begin{align*}
i_*F^1_n(r-1)_X(Y\cup D)[-2] & \stackrel{i_!}{\to }
E^{\prime}_n(r)_{{\Nis},X}( D)\to E^{\prime}_n(r)_{{\Nis},X}(Y\cup D),\quad i\geq 0;\\
i_*E_n(r-1)_{{\Nis},Y}(D_Y)[-2] & \stackrel{i_!}{\to }
E_n(r)_{{\Nis},X}( D)\to E_n(r)_{{\Nis},X}(Y\cup D),\quad r\leq p-1.
\end{align*}
Here $F^1_n(r-1)_X(Y\cup D)$ is a complex of ${\mathbf Z}/p^n$-sheaves on $Y$ equipped with a natural map $F^1_n(r-1)_X(Y\cup D)\to E^1_n(r-1)_{{\Nis},Y}(D_Y)$ that has an inverse up to $p^{N(r-1)}$, i.e., composition either way is a multiplication by $p^{N(r-1)}$ (in the derived category), where $N$ is the constant from Theorem \ref{input1}.
\end{lemma}
\begin{proof}
(1) {\em \'Etale complexes}

  We will start the proof in the case of complexes with undivided Frobenius.
Recall that we have the following short exact sequence of complexes
$$0\to 
S^{i}_n(r)_X( D)\to S^{i}_n(r)_X(Y\cup D)\lomapr{res_{Y}}i_*S^{i+1}_n(r-1)_Y(D_Y)[-1]\to 0\\
$$
This was shown in \cite[4.2,4.3]{Ts2} and follows relatively easily from two facts. First, that for a regular log-scheme $X_n$ log-smooth over $W_n(k)$, an irreducible regular divisor $D_1$ log-smooth over $W_n(k)$, and a divisor $D$ on $X_n$ such that $D_1\cup D$ is a relative simple normal crossing divisor on $X_n$ over $W_n(k)$ and $D_1\cap D$ a relative simple normal crossing divisor on $D_1$ over $W_n(k)$ the following sequence is exact 
$$
0\to \Omega^q_{X_n}(D_1)\to \Omega^q_{X_n}(D_1\cup D)\lomapr{res_{D_1}} \Omega^{q-1}_{D_{1,n}}(D_1\cap D)\to 0,\quad q\geq 1
$$
Second, that the corresponding divided power envelopes behave as expected.

  In the case of complexes $E^{\prime}_n(r)$, we want to construct the following distinguished triangle
\begin{equation}
\label{EGysin}
E^{\prime}_n(r)_X( D)\to E^{\prime}_n(r)_X(Y\cup D)\lomapr{res_{Y}}i_*E^{1}_n(r-1)_Y(D_Y)[-1],
\end{equation}
  where the first map is the natural map and the residue map is yet to be defined. Since we have Gysin sequences for syntomic and \'etale cohomologies this amounts to checking that they can be glued. Tsuji in \cite{Ts2} verified that these sequences are compatible. This is nontrivial since the period map does not behave well with respect to closed immersions. We will use his constructions to perform the gluing.

   To start, we replace the maps $Y(D_Y)\rightarrow X(D)\leftarrow X(Y\cup D)$ with the maps 
  $Y(D_Y)\stackrel{\id}{\to} Y(D_Y)\leftarrow Y(D_Y^0)$, where $Y(D_Y^0)$ is the scheme $Y$ endowed with the pullback of the log-structure of $X(Y\cup D)$.
  Consider the following commutative diagram. All the maps but the bottom rightmost map are genuine maps of complexes. 
\begin{equation}
\label{terror}
\xymatrix{
 E^{\prime}_n(r)_X( D)\ar[r]\ar@{=}[d] & E^{\prime}_n(r)_X(Y\cup D)\ar[d]^{\wr} & &\\
 E^{\prime}_n(r)_X( D)\ar@{^{(}->}[r]^-f \ar@{->>}[d]^{i^*}& E^{\prime}_n(r)_X(Y\cup D)_{\keet}\ar[d]^{i^*}\ar@{->>}^-{res}[r] & C(f,r-1)\ar[d]^{i^*}_{\wr} &\\
 i_*E^{\prime}_n(r)_Y( D_Y)\ar@{^{(}->}[r]^-{f^{\prime}} & i_*E^{\prime}_n(r)_Y( D^0_Y)_{\keet}\ar@{->>}^-{res}[r] & i_*C(f^{\prime},r-1)\ar@{-}[r]^-{\sim}& i_*E^{1}_n(r-1)_Y(D_Y)[-1]
  }
\end{equation}
The complex $E^{\prime}_n(r)_X(Y\cup D)_{\keet}$ is defined in an analogous way to $E^{\prime}_n(r)_X(Y\cup D)$ by (locally) using the map (assume that $U=\Spec(A)\to X$ is a strict \'etale map)
\begin{align*}
\alpha_{r,U}^{\prime}:\quad \Gamma(U, S^{\prime}_n(r)_{X^0,Z^0})\to  \Gamma(U^h,\theta_U & \tau_*G\overline{S}_n(r)_{U^0,Z^0})  \stackrel{\sim}{\leftarrow}\Gamma(U^h,\theta_U\tau_*G\Lambda_{U^{h}})\stackrel{\sim}{\leftarrow}\Gamma(U,\theta_U\tau_*G\Lambda_{U})\\
& \stackrel{\sim}{\leftarrow}\Gamma(U,\Cone(\eta(\tau_*G\Lambda_{U})))\to  \Gamma(U,j_{\eet !}j^{\prime}_*\tau_*G\Lambda_{U})[1]),
\end{align*}
 where $\tau: X(D)^0_{K,\tr,\keet}\to X(D)^0_{K,\tr,\eet}$ is the natural map from the Kummer \'etale topos to the \'etale topos and the subscripts $0$ refer to the log-structure induced by $Y\cup D$.
 We have a natural morphism $E^{\prime}_n(r)_X(Y\cup D)\stackrel{\sim}{\to }E^{\prime}_n(r)_X(Y\cup D)_{\keet}$ that is a quasi-isomorphism. Moreover the induced map of complexes 
 $f:E^{\prime}_n(r)_X(D)\to E^{\prime}_n(r)_X(Y\cup D)_{\keet}$ is injective \cite[Lemma 4.7.5]{Ts2}. We wrote $C(f,r-1)$ for the cokernel of this map.
 
  The complex $E^{\prime}_n(r)_Y(D^0_Y)_{\keet}$ is defined in an analogous way to $E^{\prime}_n(r)_X(Y\cup D)_{\keet}$. The local setting is as follows.
  We assume that $X$ (with the log-structure defined by $D$) is affine and we have a closed exact immersion $X\hookrightarrow Z$ over $W(k)$ into a fine log-scheme $Z$ that is log-smooth over $W(k)$ and has a compatible system of liftings of
  Frobenius $\{F_{Z_n}:Z_n\to Z_n\}$. Further we assume that $Z$ has a Cartier divisor $Z_1\hookrightarrow Z$
 defined by a global equation $g=0$ for $g\in\Gamma(Z,\so_Z)$ and such that $Y$ is a pullback of $Z_1$ and $Z_1$ equipped with the log-structure pullbacked from $Z$ is log-smooth over $W(k)$. 
 We denote by $Z^0$ the log-scheme $Z$ equipped with the log-structure induced by $Z_1$ and $M_Z$, and by $Z_1^0$ - the log-scheme equipped with the log-structure pullbacked from $Z^0$. 
 
 Let $U=\Spec(A)\to X$ be  a strict \'etale morphism. Write $U_1=Y\times_X U$. Under certain additional assumptions on $Z$ and $U$ (cf. \cite[4.7]{Ts2}) we get the following map 
\begin{align*}
\alpha_{r,U_1}^{\prime}:\quad \Gamma(U_1, S^{\prime}_n(r)_{Y^0,Z^0_1})\to  \Gamma(U^h_1,\theta_{U_1}^{\prime} & \tau_*^{\prime}G\overline{S}_n(r)_{U^0_1,Z^0_1})  \stackrel{\sim}{\leftarrow}\Gamma(U^h_1,\theta_{U_1}^{\prime}\tau_*^{\prime}G\Lambda_{U^{h}_1})\stackrel{\sim}{\leftarrow}\Gamma(U_1,\theta^{\prime}_{U_1}\tau_*^{\prime}G\Lambda_{U_1})\\
& \stackrel{\sim}{\leftarrow}\Gamma(U_1,\Cone(\eta^{\prime}(\tau_*^{\prime}G\Lambda_{U_1})))\to  \Gamma(U_1,\jmath_{\eet !}\jmath^{\prime}_*\tau_*^{\prime}G\Lambda_{U_1})[1]),
\end{align*}
 where $\tau^{\prime}: Y(D_Y)^0_{K,\tr,\keet}\to Y(D_Y)^0_{K,\tr,\eet}$ is the natural map of topoi and $\theta^{\prime}=\imath_*\imath^*\jmath_*G$ for $\imath:Y_0\hookrightarrow Y$, 
 $\jmath: Y_{K,\tr}\stackrel{\jmath^{\prime}}{\to}Y_K\stackrel{\jmath_{\eet}}{\to} Y$. We set $\eta^{\prime}(M)=\jmath_{\eet !}\jmath^{\prime}_*M\hookrightarrow \jmath_{\eet *}\jmath^{\prime}_*M$. Here the resolution 
 $\Lambda_{U_1}\stackrel{\sim}{\to}\overline{S}_n(r)_{U^0_1,Z^0_1}$ on the Kummer \'etale topos of $U_{1,K,\tr}^0$ is the one defined in \cite[4.7]{Ts2}. By globalizing, straightening, and taking cone of the morphisms $\alpha_{r,U_1}^{\prime}$ we get the complex $E^{\prime}_n(r)_Y(D^0_Y)_{\keet}$.
 
 Recall that the complex $E^{\prime}_n(r)_Y(D_Y)$ is defined locally by an analogous map
 \begin{align*}
\alpha_{r,U_1}^{\prime}:\quad \Gamma(U_1, S^{\prime}_n(r)_{Y,Z_1})\to  \Gamma(U^h_1,\theta_{U_1}^{\prime} & \overline{S}_n(r)_{U_1,Z_1})  \stackrel{\sim}{\leftarrow}\Gamma(U^h_1,\theta_{U_1}^{\prime}\Lambda_{U^{h}_1})\stackrel{\sim}{\leftarrow}\Gamma(U_1,\theta^{\prime}_{U_1}\Lambda_{U_1})\\
& \stackrel{\sim}{\leftarrow}\Gamma(U_1,\Cone(\eta(G\Lambda_{U_1})))\to  \Gamma(U_1,\jmath_{\eet !}\jmath^{\prime}_*G\Lambda_{U_1})[1]).
\end{align*}
By globalizing, straightening, and taking cone of the morphisms $\alpha_{r,U_1}^{\prime}$ we get the complex $E^{\prime}_n(r)_Y(D_Y)$.
 The induced map of complexes 
 $f^{\prime}:E^{\prime}_n(r)_Y(D_Y)\to E^{\prime}_n(r)_Y(D_Y^0)_{\keet}$ is injective \cite[Lemma 4.7.8]{Ts2}. We wrote $C(f^{\prime},r-1)$ for the cokernel of this map. It is supported on $Y$.
 
 We claim that the induced pullback map $i^*:C(f,r-1)\stackrel{\sim}{\to} C(f^{\prime},r-1)$ is a quasi-isomorphism. To verify this notice that locally
 $C(f,r-1)$ is defined as the cone of the  straightening of the following zigzag of maps
 \begin{align*}
 \Gamma(U, i_*S^1_n(r-1)_{Y,Z_1}[-1])\to  \Gamma(U^h,\theta_U & L_U)  \stackrel{\sim}{\leftarrow}\Gamma(U^h,\theta_UC_{U^{h}})\stackrel{\sim}{\leftarrow}\Gamma(U,\theta_UC_U)\\
& \stackrel{\sim}{\leftarrow}\Gamma(U,\Cone(\eta(GC_{U})))\to  \Gamma(U,j_{\eet !}j^{\prime}_*GC_{U})[1]),
 \end{align*}
 where $C$ and $L$ are the cokernels of the injective morphisms $\Lambda_{X}\to \tau_*G(\Lambda_{X^0})$ and $\overline{S}_n(r)_{U,Z}\to \tau_*G\overline{S}_n(r)_{U^0,Z^0}$, respectively.
 
 Similarly, $C(f^{\prime},r-1)$ is defined as the cone of the straightening of the following zigzag of maps (c.f, \cite[4.7.12]{Ts2})
 \begin{align*}
 \Gamma(U_1, S^1_n(r-1)_{Y,Z_1}[-1])\to  \Gamma(U^h_1,\theta_{U_1}^{\prime} & L^{\prime}_{U_1})  \stackrel{\sim}{\leftarrow}\Gamma(U^h_1,\theta_{U_1}^{\prime}C^{\prime}_{U^{h}_1})\stackrel{\sim}{\leftarrow}\Gamma(U_1,\theta^{\prime}_{U_1}C^{\prime}_{U_1})\\
& \stackrel{\sim}{\leftarrow}\Gamma(U_1,\Cone(\eta^{\prime}(GC^{\prime}_{U_1})))\to  \Gamma(U_1,\jmath_{\eet !}\jmath^{\prime}_*GC^{\prime}_{U_1})[1]),
 \end{align*}
 where $C^{\prime}$ and $L^{\prime}$ are the cokernels of the injective morphisms $\Lambda_{Y}\to \tau_*G(\Lambda_{Y^0})$ and $\overline{S}_n(r)_{U_1,Z_1}\to \tau_*G\overline{S}_n(r)_{U_1^0,Z_1^0}$, respectively. 
 There exists a natural map ($i^*$) from the first zigzag to the second one. Hence the two maps are equal in the derived category and we obtain a  genuine map of complexes that yields a quasi-isomorphism $i^*:C(f,r-1)\stackrel{\sim}{\to} C(f^{\prime},r-1)$, as wanted. 
 
 It remains to construct a quasi-isomorphism $C(f^{\prime},r-1)\simeq E^{1}_n(r-1)_Y(D_Y)[-1]$. Consider the following  two commutative diagrams of maps of complexes. Here we set
  $\Gamma^h(M)=\Gamma(U_1^h,\theta^{\prime}_{U_1}M)$, $\Gamma(M)=\Gamma(U_1,\theta^{\prime}_{U_1}M)$, $\Lambda(a)={\mathbf Z}/p^n(a)^{\prime}$.
 \begin{equation}
 \label{first-diagram}
 \xymatrix{
 \Gamma^h(\Lambda(r-1)[-1])\ar[d]_{\wr}\ar[r]^-{h} & \Gamma^h(\Lambda(r)(-1)[-1])& \Gamma^h(\Lambda(r)(-1)[-1])\ar@{=}[l] & \Gamma(\Lambda(r)(-1)[-1]) 
 \ar[l]^-{\sim}\\
 \Gamma^h(\sh^1(\overline{S}_n(r-1)_{Y,Z_1})[-2])\ar[r] & \Gamma^h(\sh^1(L^{\prime}_{U_1})[-1])\ar[u]^{\wr} & \Gamma^h(\sh^1(C^{\prime}_{U_1^h})[-1])\ar[u]^{\wr}\ar[l]^-{\sim} & \Gamma(\sh^1(C^{\prime}_{U_1})[-1])\ar[l]^-{\sim}\ar[u]^{\wr}\\
 \Gamma^h(\tau_{\leq 1}\overline{S}_n(r-1)_{Y,Z_1}[-1])\ar[d]^{\wr}\ar[r]\ar[u]^{\wr} & \Gamma^h(\tau_{\leq 1}L^{\prime}_{U_1})\ar[d]^{\wr}\ar[u]^{\wr} & \Gamma^h(\tau_{\leq 1}C^{\prime}_{U_1^h})\ar[d]^{\wr}\ar[u]^{\wr}\ar[l]^{\sim} & \Gamma(\tau_{\leq 1}C^{\prime}_{U_1})\ar[d]^{\wr}\ar[u]^{\wr}\ar[l]^{\sim}\\
 \Gamma^h(\overline{S}_n(r-1)_{Y,Z_1}[-1])\ar[r]& \Gamma^h(L^{\prime}_{U_1}) & \Gamma^h(C^{\prime}_{U_1^h})\ar[l]^{\sim} & \Gamma(C^{\prime}_{U_1})\ar[l]^{\sim}
 }
 \end{equation}
 \begin{equation}
 \label{second-diagram}
 \xymatrix{
 \Gamma(\Lambda(r)(-1)[-1]) & \Gamma(\Cone(\eta^{\prime}(G\Lambda(r)(-1)[-1]))) \ar[l]^-{\sim}\ar[r] & \Gamma(\jmath_{\eet !}\jmath^{\prime}_*G\Lambda(r)(-1))\\
 \Gamma(\sh^1(C^{\prime}_{U_1})[-1])\ar[u]^{\wr} & \Gamma(\Cone(\eta^{\prime}(G\sh^1(C^{\prime}_{U_1})[-1])))\ar[l]^-{\sim}\ar[r]\ar[u]^{\wr} & \Gamma(\jmath_{\eet !}\jmath^{\prime}_*G\sh^1(C^{\prime}_{U_1}))\ar[u]^{\wr}\\
 \Gamma(\tau_{\leq 1}C^{\prime}_{U_1})\ar[u]^{\wr} \ar[d]^{\wr} & \Gamma(\Cone(\eta^{\prime}(G\tau_{\leq 1}C^{\prime}_{U_1})))\ar[l]^-{\sim}\ar[r]\ar[u]^{\wr}\ar[d]^{\wr} & \Gamma(\jmath_{\eet !}\jmath^{\prime}_*G\tau_{\leq 1}C^{\prime}_{U_1}[1])\ar[u]^{\wr}\ar[d]^{\wr}\\
 \Gamma(C^{\prime}_{U_1}) & \Gamma(\Cone(\eta^{\prime}(GC^{\prime}_{U_1})))\ar[l]^-{\sim}\ar[r] & \Gamma(\jmath_{\eet !}\jmath^{\prime}_*GC^{\prime}_{U_1}[1])
 }
\end{equation}
 The map $h$ is defined to make the upper left corner square of the first diagram commute. The bottom leftmost map in the same diagram is that defined in \cite[4.7.10]{Ts2}. By \cite[Prop. 4.8.4]{Ts2} it is equal to the natural map. Take the two maps
 $\Gamma(U_1,\overline{S}_n(r-1)_{Y,Z_1}[-1])\to \Gamma(U_1,\Lambda(-1)[-1])$ obtained by composing the maps from the the left and upper sides of the first diagram and the upper side of the second diagram, respectively, from the lower and right sides.
  Composing these maps with the map $\Gamma(U_1, S^1_n(r-1)_{Y,Z_1}[-1])\to \Gamma(U_1^h,\overline{S}_n(r-1)_{Y,Z_1}[-1])$ \cite[4.7.11]{Ts2} we obtain maps that define the complexes $E^{1}_n(r-1)_Y(D_Y)[-1]$ and $C(f^{\prime},r-1)$, respectively \cite[4.7.10]{Ts2}. Now our diagrams describe a specific quasi-isomorphism $C(f^{\prime},r-1)\simeq E^{1}_n(r-1)_Y(D_Y)[-1]$, as wanted.
        
 We finish by setting the residue map  $res_{Y}: E^{\prime}_n(r)_X(Y\cup D)\to i_*E^{1}_n(r-1)_Y(D_Y)[-1]$ to be equal to the composition of the maps from the diagram (\ref{terror})
 $$ E^{\prime}_n(r)_X(Y\cup D)\stackrel{\sim}{\to} E^{\prime}_n(r)_X(Y\cup D)_{\keet}\to C(f,r-1) \stackrel{i^*}{\to} i_*C(f^{\prime},r-1)\simeq  i_*E^{1}_n(r-1)_Y(D_Y)[-1],
 $$
 and by pushing everything down to the Nisnevich site using $\varepsilon_*G$ for $\varepsilon: X_{\eet}\to X_{\Nis}.$
 The same diagram shows now that we have the distinguished triangle (\ref{EGysin}). By construction, the above residue map is compatible with the syntomic and \'etale residue maps \begin{align}
\label{residues}
S^{\prime}_n(r)_X(Y\cup D) & \to i_*S^1_n(r-1)_Y(D_Y)[-1],\\
j_{\Nis !}\R^r(j^{\prime}\varepsilon)_*\Z/p^n(r)_{X_{\tr}} &  \to   i_*\jmath_{\Nis !}\R^{r-1} \varepsilon_*\Z/p^n(r-1)_{Y_{\tr}}=  j_{\Nis !}i_*\R^{r-1} \varepsilon_*\Z/p^n(r-1)_{Y_{\tr}}.\notag
\end{align}
 
  The constructions for the complexes $S_n(r)$, $r\leq p-1$, and  $E_n(r)$, $r\leq p-2$ are analogous. So is the construction for the complexes $E^1_n(r)$ after one notices that the residue maps are compatible with the maps $\omega_0$ and $\omega_1$. 
 
  (2) {\em Nisnevich complexes}\\
  (2a) {\em The case of $r\geq p-2$}
  
 In the case of complexes $E_n(r)_{\Nis}$, $r\leq p-2$, we  construct  the triangle
\begin{equation}
\label{ETGysin}
E_n(r)_{{\Nis},X}( D)\to E_n(r)_{{\Nis},X}(Y\cup D)\lomapr{res_{Y}}i_*E_n(r-1)_{{\Nis},Y}(D_Y)[-1]
\end{equation}
by truncating the complexes in the analog of the distinguished triangle (\ref{EGysin}). To show that the obtained triangle  is distinguished it suffices to show that the residue map induces a surjection
$$res_Y: \sh^r(E_n(r)_{{\Nis},X}(Y\cup D))\to i_*\sh^{r-1}(E_n(r-1)_{{\Nis},Y}(D_Y)).
$$
To simplify the notation we will assume that the divisor $D$ is trivial. Consider the following commutative diagram (see (\ref{residues}))
$$
\xymatrix{\sh^r(S_n(r)_X(Y))\ar@{->>}[r]^-{res_Y} & i_*\sh^{r-1}(S_n(r-1)_Y)\\
 \sh^r(E_n(r)_{{\Nis},X}(Y))\ar@{->>}[u]\ar[r]^-{res_Y} &  i_*\sh^{r-1}(E_n(r-1)_{{\Nis},Y})\ar@{->>}[u]\\
  j_{\Nis !}\R^r(j^{\prime}\varepsilon)_*\Z/p^n(r)_{X_{\tr}} \ar[u]\ar@{->>}[r]^-{\res_Y} &  i_*\jmath_{\Nis !}\R^{r-1} \varepsilon_*\Z/p^n(r-1)_{Y_K}=  j_{\Nis !}i_*\R^{r-1} \varepsilon_*\Z/p^n(r-1)_{Y_K}\ar[u]
}
$$

   It suffices to show that the bottom and the top residue maps in the above diagram are surjective. For that we can assume that $X$ is affine and $Y=(T)$. Assume first that $X$ is semistable over $\so_K^{\times}$. Since the period map is compatible with Gysin sequences by \cite[Prop. 4.5.3]{Ts2}, Theorem \ref{input0} implies that for the top map it suffices to show surjectivity of the map 
\begin{equation}
\label{paryz}
  \R^r(j\varepsilon)_*\Z/p^n(r)_{X_{\tr}} \to  i_*\R^{r-1}(\jmath\varepsilon)_*\Z/p^n(r-1)_{Y_K}. \end{equation}
For the bottom map, it suffices to show   surjectivity of the map
\begin{equation}
\label{surjectivity}
  \R^r(j^{\prime}\varepsilon)_*\Z/p^n(r)_{X_{\tr}} \to  i_*\R^{r-1}\varepsilon_*\Z/p^n(r-1)_{Y_K}. \end{equation}

  Set $U:=X_{\tr}$. We have the symbol map
$$
M^{\gp}_X=j_*\so^*_U\to \R^1j_*\Z/p^n(1)_U=\R^1j_{\eet *}\tau_{\leq 1}\R j^{\prime}_*\Z/p^n(1)_U
$$
and the following commutative diagram (cf. \cite[3.4.6]{Ts}). We set here $\wt{M}^{\gp}_X:=j^{\prime}_*\so_{X_K}^*$.
$$\xymatrix{ (M^{\gp}_X)^{\otimes r}\ar[r] & \R^rj_{\eet *}\tau_{\leq r}\R j^{\prime}_*\Z/p^n(r)_U \ar[r]\ar[dd] & \R^r j_*\Z/p^n(r)_U\ar@/^23pt/[dddl]^{res_Y}\\
M^{\gp}_X\otimes (\wt{M}^{\gp}_X)^{\otimes(r-1)}\ar[ru]\ar[u] & \\
(\wt{M}^{\gp}_X)^{\otimes(r-1)}\ar[d]\ar[u]^{\kappa} & \R^{r-1}j_{\eet*}\R^1j^{\prime}_*\Z/p^n(r)_U\ar[d]^{\wr}\\
(\wt{M}^{\gp}_Y)^{\otimes(r-1)}\ar[r] & \R^{r-1}\jmath_{\eet *}\Z/p^n(r-1)_{Y_K}
}
$$
The map $\kappa$ is defined by sending the section $m$ to $T \otimes m$. We claim that the symbol map
$$
(\wt{M}^{\gp}_Y)^{\otimes(r-1)}\to \R^{r-1}(\jmath_{\eet }\varepsilon)_*\Z/p^n(r-1)_{Y_K}
$$
is surjective. Since we have Gersten conjecture for Milnor K-theory \cite[Theorem 7.1]{KZ} and \'etale cohomology, this follows from Bloch-Kato conjecture and \cite[Lemma 7.2]{KZ}.
Since the map $\wt{M}^{\gp}_X\to  \wt{M}^{\gp}_Y$ is surjective in the Zariski topology on $X$, the above diagram shows  surjectivity of the map (\ref{paryz}) we wanted. The argument for the map (\ref{surjectivity}) is analogous.

  Assume now that $X$ is semistable  over $\so_K$ with a smooth special fiber. Consider the following diagram
$$
\xymatrix{
\sh^r(\wt{S}_n(r)_X)\ar@{^(->}[d]\ar[r] & \sh^r(\wt{S}_n(r)_X(Y))\ar@{^(->}[d]\ar[r]^-{res_Y} & i_*\sh^{r-1}(\wt{S}_n(r-1)_Y)\ar@{^(->}[d]\\
\sh^r(S_n(r)_X)\ar@{->>}[d]^{\beta}\ar[r] & \sh^r(S_n(r)_X(Y))\ar@{->>}[d]^{\beta}\ar@{->>}[r]^-{res_Y} & i_*\sh^{r-1}(S_n(r-1)_Y)\ar@{->>}[d]^{\beta}\\
i_*W_n\Omega^{r-1}_{X_0,\log} \ar@{^(->}[r] & i_*W_n\Omega^{r-1}_{U_0,\log}\ar@{->>}[r]^{res_Y}&  i_*W_n\Omega^{r-2}_{Y_0,\log}
}
$$
The horizontal sequences are Gysin sequences. The syntomic cohomology $\wt{S}$ uses the vertical log-structure, the syntomic cohomology $S$ uses the trivial log-structure. The vertical sequences are exact \cite[1]{Kur} (the map $\beta$ is the residue at $\varpi$ map). We claim that the diagram commutes. Indeed, for the top part this follows from functoriality. For the bottom part, the left square commutes by functoriality. For the right square,  it suffices to look at symbols. It follows from the definition that the map $\beta$ sends the symbol $\{ f_1,\ldots, f_{r}\}$, $f_i\in i^*_0\so_{X}^{*}$,  $i_0:X_0\hookrightarrow X$, to zero and the symbol $\{f_1,\ldots,f_{r-1},\varpi\}$, $f_i\in i^*_0\so_{X}^{*}$, to $\dlog [\overline{f}_1]\wedge\cdots\dlog [\overline{f}_{r-1}]$. The definition of the residue map $\res_Y$ is similar with $\varpi$ replaced by $T$. The commutativity we want is now immediate. 

  (2b) {\em The general case.} 
  
 In the case of complexes $E_n(r)^{\prime}_{\Nis}$, we  construct  the triangle
\begin{equation}
\label{ETTGysin}
E_n(r)^{\prime}_{{\Nis},X}( D)\to E_n(r)^{\prime}_{{\Nis},X}(Y\cup D)\lomapr{res_{Y}}i_*F^1_n(r-1)_{X}(Y\cup D)[-1]
\end{equation}
by truncating the complexes in the Nisnevich  version of the distinguished triangle (\ref{EGysin}) and by setting $F^1_n(r-1)_{X}(Y\cup D)$ equal to the image of
the composition
$$ (\tau_{\leq r}\varepsilon_*GE^{\prime}_n(r)_X(Y\cup D)_{\keet})[1]\to (\tau_{\leq r}\varepsilon_*GC(f,r-1))[1].
$$
Clearly  the obtained triangle  is distinguished. It remains to show that the natural injection
$$F^1_nr-1)_{X}(Y\cup D)\hookrightarrow (\tau_{\leq r}\varepsilon_*GC(f,r-1))[1]
$$
has an inverse up to $p^{N(r-1)}$. Since $C(f,r-1)\simeq i_*E^1_n(r-1)_{Y}(D_Y)[-1]$,  we have 
\begin{align*}
\tau_{\leq r-2}F^1_n(r-1)_{X}(Y\cup D) & \simeq \tau_{\leq r-1}\varepsilon_*GC(f,r-1),\\
\sh^r\varepsilon_*GC(f,r-1)& \simeq i_*\sh^{r-1} E^1_n(r-1)_{Y}(D_Y).
\end{align*}
It suffices thus to show that the map
$$res_Y: \sh^r(E_n^{\prime}(r)_{{\Nis},X}(Y\cup D))\to i_*\sh^{r-1}(E^1_n(r-1)_{{\Nis},Y}(D_Y))
$$
has a cokernel annihilated  by $p^{N(r-1)}$.  But this follows just as in the case $r\leq p-2$ treated above using Theorem \ref{input1} and, in the good reduction case,  \cite[Theorem 3.2]{EN}.
\end{proof}
Gysin distinguished triangles are functorial for certain morphisms of log-schemes.
\begin{lemma}
\label{torind1}
For $Y$, $X$, and $D$ as above, let $f:X^{\prime}\to X$ be a flat morphism
or a closed immersion such that $X^{\prime}$ is regular and syntomic over $W(k)$ and $f^{-1}(Y\cup D)$ 
is a relative simple normal crossing
divisor on $X^{\prime}$ over $W(k)$. Set $Y^{\prime}=f^{-1}(Y)$ and $D^{\prime}=f^{-1}(D)$.
 Then, for $r\geq 1$, $i\geq 0$, we have a map of distinguished triangles
$$
\begin{CD}
i_*S^{i+1}_n(r-1)_Y(D_Y)[-2]@>i_! >>
S^{i}_n(r)_X( D)@>>> S^{i}_n(r)_X (Y\cup D)\\
@VVf^* V @VVf^* V @VVf^* V \\
\R f_*i_*S^{i+1}_n(r-1)_{Y^{\prime}}(D^{\prime}_{Y^{\prime}})[-2]@> i_! >>
\R f_*S^{i}_n(r)_{X^{\prime}}( D^{\prime})@>>>
\R f_*S^{i}_n(r)_{X^{\prime}} (Y^{\prime}\cup D^{\prime})
\end{CD}
$$
Similarly, for the complexes $S_n(r)_X( D)$, $r\leq p-1$.
  
  There are analogous maps of Gysin distinguished triangles for the complexes $E_n(r)$ and $E_n(r)_{\Nis}$ (and $r\leq p-2$) as well as for the complexes $E^{\prime}_n(r)$. For the complexes $E^{\prime}_n(r)_{\Nis}$ we have the following maps of Gysin distinguished triangles
$$
\begin{CD}
i_*F^{1}_n(r-1)_X(Y\cup D)[-2]@>i_! >>
E^{\prime}_n(r)_{\Nis,X}( D)@>>> E^{\prime}_n(r)_{\Nis,X} (Y\cup D)\\
@VVf^* V @VVf^* V @VVf^* V \\
\R f_*i_*F^{1}_n(r-1)_{X^{\prime}}({Y^{\prime}}\cup D^{\prime})[-2]@> i_! >>
\R f_*E^{\prime}_n(r)_{\Nis,X^{\prime}}( D^{\prime})@>>>
\R f_*E^{\prime}_n(r)_{\Nis,X^{\prime}} (Y^{\prime}\cup D^{\prime})
\end{CD}
$$
\end{lemma}
\begin{proof}
The case of the map $f$ being flat follows from the flat base change. The case of $f$ being a closed immersion follows from the fact that the morphisms
$X^{\prime}\to X$ and $Y\to X$  are Tor-independent.
\end{proof}
\subsubsection{Projection formula}
We have the following projection formula for a closed immersion of codimension one.
\begin{lemma}
\label{projsynt}
 Let $X$, $Y$, and $D$ be as above.  Then, for $r_1,r_2,i,j\geq 0$, the following diagram commutes.
$$
\xymatrix{
S_n^{i}(r_1)_X(D)\otimes^{{\mathbb L}} i_*S_n^{j+1}(r_2-1)_Y(  D_Y)[-2] \ar[r]^-{1\otimes i_!} 
\ar[d]^{i^*\otimes 1}
& S_n^{i}(r_1)_X(D)\otimes^{{\mathbb L}} S_n^{j}(r_2)_X( D) \ar[dd]^{\cup}\\
i_*S_n^{i}(r_1)_Y( D_Y)\otimes^{{\mathbb L}} i_*S_n^{j+1}(r_2-1)_Y(D_Y)[-2]\ar[d]^{\cup} &\\
i_*S_n^{i+j+1}(r_1+r_2-1)_Y( D_Y)[-2]\ar[r]^-{i_!} & S_n^{i+j}(r_1+r_2)_{X}( D)
}
$$
Similarly, for the complexes $S_n(r)_X( D)$, $r\leq p-1$. 

 There are analogous projection formulas  for the complexes $E_n(r)$ and $E_{n}(r)_{\Nis}$ (and $r\leq p-2$) as well as for the complexes
 $E^{\prime}_n(r)$. For the complexes   $E^{\prime}(r)_{\Nis}$  we have the following commutative diagram (the constant $N$ is the one from Theorem \ref{input1})
\begin{equation}
\label{projZar}
\xymatrix@C=40pt{
E_n^{\prime}(r_1)_{\Nis,X}(D)\otimes^{{\mathbb L}} i_*F^1_n(r_2-1)_X(Y\cup D)[-2] \ar[r]^-{p^{N(r_1+r_2-1)}\otimes i_!} 
\ar[d]^{i^*\otimes 1}
& E_n^{\prime}(r_1)_{\Nis,X}(D)\otimes^{{\mathbb L}} E_n^{\prime}(r_2)_{\Nis,X}( D) \ar[dd]^{\cup}\\
i_*E_n^{\prime}(r_1)_{\Nis,Y}( D_Y)\otimes^{{\mathbb L}} i_*F_n^{1}(r_2-1)_X(Y\cup D)[-2]\ar[d]^{\cup} &\\
i_*F_n^{1}(r_1+r_2-1)_{X}( Y\cup D)[-2]\ar[r]^{i_!} & E_n^{\prime}(r_1+r_2)_{\Nis,X}( D)
}
\end{equation}
\end{lemma}
\begin{proof}For the first claim of the lemma, it suffices to show that the localization short exact sequence
$$0\to 
S^{j}_n(r_2)_X( D)\stackrel{f}{\to}S^{j}_n(r_2)_X(Y\cup D)\stackrel{res}{\to}i_*S^{j+1}_n(r_2-1)_Y(D_Y)[-1]\to 0
$$
is compatible with the action of $S^{i}_n(r_1)_X( D)$. For the localization map $f$ this follows from the functoriality of products; for the residue map
it will follow if we show that the following diagram of complexes commutes.
$$
\xymatrix{
S_n^{j}(r_2)_X(D\cup Y)\otimes S_n^{i}(r_1)_X(  D) \ar[r]^-{1\otimes f} 
\ar[d]^{res\otimes 1} 
& S_n^{j}(r_2)_X(D\cup Y)\otimes S_n^{i}(r_1)_X( D\cup Y) \ar[d]^{\cup}\\
i_*S_n^{j+1}(r_2-1)_Y( D_Y)[-1]\otimes S_n^{i}(r_1)_X(D)\ar[d]^{1\otimes i^*} &S_n^{i+j}(r_2+r_1)_X(D\cup Y)\ar[d]^{res}\\
i_*S_n^{j+1}(r_2-1)_Y( D_Y)[-1]\otimes i_*S_n^{i}(r_1)_Y( D_Y)\ar[r]^-{\cup}& i_*S_n^{i+j+1}(r_1+r_2-1)_Y( D_Y)[-1]
}
$$
But this can be checked locally (in the lifted situation) where it follows immediately from the formulas for product. 
The argument for the complexes $S_n(r)_X( D)$ is basically the same. 

  For the complexes $E^{\prime}_n(r)$, it suffices to argue on the \'etale site. Assume thus that $X$ is equipped with the \'etale topology. Consider the following commutative diagram of short exact sequences that appears in  the proof of Lemma \ref{localization} (diagram (\ref{terror})). Below
 we will use the notation from that proof freely.
$$
\begin{CD}
0@>>> E^{\prime}_{n}(r_2)_X( D) @> f >> E^{\prime}_{n}(r_2)_X(Y\cup D)_{\keet} @>res >> C(f,r_2-1) @>>> 0\\
@. @VV i^*V @VV i^*V @V\wr V i^*V \\
0@>>>  i_*E^{\prime}_{n}(r_2)_Y( D_Y) @>f^{\prime} >> i_*E^{\prime}_{n}(r_2)_Y(D_Y^0)_{\keet}@> res >> i_*C(f^{\prime},r_2-1)@>>> 0
\end{CD}
$$ 
Proceeding as in the definition of cup product on the complexes $E_n^{\prime}(r)$, we define  an action of $E^{\prime}_{n}(r_1)_X( D)$ on the first two terms of the top exact sequence; hence an action on $C(f,r_2-1)$. In a similar way we can define a compatible action of $i_*E^{\prime}_{n}(r_2)_Y( D_Y)$ on the first two terms of the bottom exact sequence; hence an action on $i_*C(f^{\prime},r_2-1)$. It follows immediately that these two actions are compatible with the above map between exact sequences. 

  It follows that if we define the product
$$\cup: i_*E^{\prime}_{n}(r_1)_Y( D_Y)\otimes^{\mathbb L}C(f,r_2-1)\stackrel{1\otimes i^*}{\to}i_*E^{\prime}_{n}(r_1)_Y( D_Y)\otimes^{\mathbb L}C(f^{\prime},r_2-1)\stackrel{\cup}{\to}
C(f^{\prime},r_1+r_2-1)\stackrel{i^*}{\simeq}C(f,r_1+r_2-1),
$$
  from the action of $E^{\prime}_{n}(r_1)_X( D)$ on the top exact sequence, we obtain the following projection formula
$$
\xymatrix{
E^{\prime}_{n}(r_1)_X( D)\otimes^{\mathbb L}C(f,r_2-1)[-1]\ar[r]^{1\otimes \partial}\ar[d]^{i^*\otimes 1} & E^{\prime}_{n}(r_1)_X( D)\otimes^{\mathbb L}E^{\prime}_{n}(r_2)_X( D)\ar[dd]^{\cup}\\
i_*E^{\prime}_{n}(r_1)_Y( D_Y)\otimes^{\mathbb L}C(f,r_2-1)[-1]\ar[d]^{\cup} &\\
C(f,r_1+r_2-1)[-1]\ar[r]^{\partial} & E^{\prime}_{n}(r_1+r_2)_X( D)
}
$$
Recall that we have the quasi-isomorphisms
$$\gamma^{\prime}: C(f,r_2-1)\stackrel{i^*}{\simeq}i_*C(f^{\prime},r_2-1)
\stackrel{\gamma}{\simeq}i_*E^{1}_{n}(r_2-1)_Y( D_Y)[-1]
$$
Thus to obtain the projection formula in the statement of the lemma it suffices to show that the following diagram commutes in the derived category.
$$
\begin{CD}
i_*E^{\prime}_{n}(r_1)_Y( D_Y)\otimes^{\mathbb L}C(f,r_2-1)@>\cup >> C(f,r_1+r_2-1)\\
@V\wr V 1\otimes \gamma^{\prime}  V @V\wr V \gamma^{\prime} V\\
i_*E^{\prime}_{n}(r_1)_Y( D_Y)\otimes^{\mathbb L}i_*E^{1}_{n}(r_2-1)_Y( D_Y)[-1]@>\cup >> i_*E^{1}_{n}(r_1+r_2-1)_Y( D_Y)[-1]
\end{CD}
$$
Or, unwinding the action of $i_*E^{\prime}_{n}(r_1)_Y( D_Y)$ on $C(f,r_2-1)$, that the following diagram commutes in the derived category.
\begin{equation}
\label{pairing}
\begin{CD}
E^{\prime}_{n}(r_1)_Y( D_Y)\otimes^{\mathbb L}C(f^{\prime},r_2-1)@>\cup >> C(f^{\prime},
r_1+r_2-1)\\
@V\wr V 1 \otimes \gamma  V @V\wr V \gamma V\\
E^{\prime}_{n}(r_1)_Y( D_Y)\otimes^{\mathbb L}E^{1}_{n}(r_2-1)_Y( D_Y)[-1]@>\cup >> E^{1}_{n}(r_1+r_2-1)_Y( D_Y)[-1]
\end{CD}
\end{equation}

  To see this consider the  commutative diagrams (\ref{first-diagram}) and (\ref{second-diagram}) that define the map $\gamma$. Set $r=r_2$. Consider also their analogs built on top of the map
\begin{align*}
\Gamma^h(\overline{S}_n(r_1)_{Y,Z_1}[-1]) & \stackrel{ \id}{\to} \Gamma^h(\overline{S}_n(r_1)_{Y,Z_1}[-1])\stackrel{\sim}{\leftarrow}\Gamma^h(\Lambda(r_1)))\stackrel{\sim}{\leftarrow}\Gamma(\Lambda(r_1)))\stackrel{\sim}{\leftarrow}\\
 & \Gamma(\Cone(\eta^{\prime} G\Lambda(r_1))) \to \Gamma(\jmath_{\eet !}\jmath^{\prime}_*G\Lambda(r_1))
\end{align*}
as well as the extension
$$\Gamma(S^{\prime}_n(r_1)_{Y,Z_1}[-1])\to \Gamma^h(\overline{S}_n(r_1)_{Y,Z_1}[-1])
$$
We claim that there exists a pairing of the above two sets of extended diagrams into the extended diagrams (\ref{first-diagram}) and (\ref{second-diagram}) for $r=r_1+r_2$ that matches corresponding nodes 
and all the maps are compatible with  pairings.
Indeed, by functoriality, it suffices to construct such a pairing for the bottom rows. And there it can be checked easily using the fact that the 
complexes $L^{\prime}$ and $C^{\prime}$ are defined as cokernels of maps that are compatible with certain obvious pairings. It is easy to see  that  this pairing of 
 diagrams  induces a quasi-isomorphism of pairings (\ref{pairing}) that we wanted.
 
 For the complexes $E^{\prime}_n(r)_{\Nis}$, consider the following commutative diagram of short exact sequences that is a truncated Zariski version of a diagram that appears in  the proof of Lemma \ref{localization} (diagram (\ref{terror})). 
$$
\begin{CD}
0@>>> E^{\prime}_{n}(r_2)_{\Nis,X}( D) @> f >> \tau_{\leq r_2}\pi_*GE^{\prime}_{n}(r_2)_X(Y\cup D)_{\keet} @>res >> i_*F^{1}_{n}(r_2-1)_{X}(Y\cup D) @>>> 0\\
@. @VV i^*V @VV i^*V @VV i^*V \\
0@>>>  i_*E^{\prime}_{n}(r_2)_{\Nis,Y}( D_Y) @>f^{\prime} >> \tau_{\leq r_2}\pi_*Gi_*E^{\prime}_{n}(r_2)_Y(D_Y^0)_{\keet}@> res >> i_*\wt{F}^{1}_{n}(r_2-1)_{X}(Y\cup D)@>>> 0
\end{CD}
$$ 
Here $i_*\wt{F}^{1}_{n}(r_2-1)_{X}(Y\cup D)$ is defined as the cokernel of the map $f^{\prime}$.
Proceeding as above we define  an action of $E^{\prime}_{n}(r_1)_{\Nis,X}( D)$ on the first two terms of the top exact sequence; hence an action on $i_*F^{1}_{n}(r_2-1)_{X}(Y\cup D)$. In a similar way we can define a compatible action of $i_*E^{\prime}_{n}(r_2)_{\Nis,Y}( D_Y)$ on the first two terms of the bottom exact sequence; hence an action on 
$i_*\wt{F}^{1}_{n}(r_2-1)_{X}(Y\cup D)$. Clearly these two actions are compatible with the above map between exact sequences and with the action on $i_*E_n^{1}(r_2-1)_{\Nis,X}(Y\cup D)[-2]$. 

  We set the product
  $$\cup: i_*E_n^{\prime}(r_1)_{\Nis,Y}( D_Y)\otimes^{{\mathbb L}} i_*F_n^{1}(r_2-1)_X(Y\cup D)[-2] \to 
i_*F_n^{1}(r_1+r_2-1)_{X}( Y\cup D)[-2]
$$
 to be  equal to the composition
  \begin{align*}
 i_*E_n^{\prime} (r_1)_{\Nis,Y}( D_Y)  & \otimes^{{\mathbb L}} i_*F_n^{1}(r_2-1)_X(Y\cup D)[-2] \to
  i_*E_n^{\prime}(r_1)_{\Nis,Y}( D_Y)\otimes^{{\mathbb L}} i_*E_n^{1}(r_2-1)_{\Nis,X}(Y\cup D)[-2] \\
   & \stackrel{\cup}{\to} 
 i_*E_n^{1}(r_1+r_2-1)_{\Nis,X}(Y\cup D)[-2]\stackrel{\beta}{\to}i_*F_n^{1}(r_1+r_2-1)_X(Y\cup D)[-2],
  \end{align*}
  where $\beta$ is an inverse (in the derived category and up to $p^{N(r_1+r_2-1)}$) of the natural map $F_n^{1}(r_2-1)_X(Y\cup D)[-2]\to E_n^{1}(r_1+r_2-1)_{\Nis,X}(Y\cup D)[-2]$. 
  
  To show that the diagram (\ref{projZar}) commutes it suffices to show it with the lower right-hand term $E_n^{\prime}(r_1+r_2)_{\Nis,X}( D)$
 replaced by $E_n^{\prime}(r_1+r_2)_{X}( D)$. This is by degree reason: cohomology of $E_n^{\prime}(r_1)_{\Nis,X}(D)\otimes^{{\mathbb L}} i_*F^1_n(r_2-1)_X(Y\cup D)[-2]$ is concentrated in degrees 
 $[0,r_1+r_2+1]$ and cohomology of $E_n^{\prime}(r_1+r_2)_{\Nis,X}( D)$ is concentrated in degrees $[0,r_1+r_2]$; hence the map 
 \begin{align*}
\Hom_{\sd}(E_n^{\prime}(r_1)_{\Nis,X}(D) & \otimes^{{\mathbb L}} i_*F^1_n(r_2-1)_X(Y\cup D)[-2], E_n^{\prime}(r_1+r_2)_{\Nis,X}( D))\\
& \to \Hom_{\sd}(E_n^{\prime}(r_1)_{\Nis,X}(D)\otimes^{{\mathbb L}} i_*F^1_n(r_2-1)_X(Y\cup D)[-2],E_n^{\prime}(r_1+r_2)_{X}( D))
\end{align*}
 between homomorphism groups in the derived category is injective. Now a simple diagram chase reduces this projection formula diagram  to the one for $E_n^{\prime}(r)_{X}$, which we have just proved.
\end{proof}

\begin{lemma}
\label{Chern}
Let $X$, $Y$, $D$ be as above. Then there exists a map of distinguished triangles 
$$
\begin{CD}
j_*\so^*_{X\setminus D}[-1] @>>> j^{\prime}_*\so^*_{X\setminus (D\cup Y)}[-1] 
@>\ord_Y >> i_*{\mathbf Z}_Y[-1]\\
@VV {c}^{\synt}_1 V @VV {c}^{\synt}_1 V @VV c_0 V\\
S^{\prime}_n(1)_X(D) @>>>  S^{\prime}_n(1)_X( Y\cup D) @>\res_Y >> 
 i_*S^{1}_n(0)_Y(D_Y)[-1]  
\end{CD}
$$
Here $j: X\setminus D\hookrightarrow X$ and 
$j^{\prime}: X\setminus (D\cup Y)\hookrightarrow X$ 
are the natural open immersions and the map $c_0$ on the right is defined by the natural morphism
${\mathbf Z}_Y\to \sh^0(S^{1}(0)_Y(D_Y))$. Similarly for the  Chern classes
with values in $S_n(1)_X(D)$.

  We have an analogous map of distinguished  triangles for $E_n(1)_X(D)$ and $E^{\prime}_n(1)_X(D)$ as well as the induced map of distinguished triangles
    $$
\begin{CD}
j_*\so^*_{X\setminus D}[-1] @>>> j^{\prime}_*\so^*_{X\setminus (D\cup Y)}[-1] 
@>\ord_Y >> i_*{\mathbf Z}_Y[-1]\\
@VV {c}^{\synt}_1 V @VV {c}^{\synt}_1 V @VV c_0V\\
E^{\prime}_n(1)_{\Nis,X}(D) @>>>  E^{\prime}_n(1)_{\Nis,X}( Y\cup D) @>\res_Y >> 
 i_*F^{1}_n(0)_X(Y\cup D)[-1]  
\end{CD}
$$
\end{lemma}
\begin{proof}
The statements for $S^{\prime}_n$ and $ S_n$ follow immediately from the above formulas. Since $\sh^0(E^1_n(0)_Y(D_Y))\hookrightarrow \sh^0(S^1_n(0)_Y(D_Y))$ the statements for $E^{\prime}_n$ and $ E_n$ follow as soon as we define the map $c_0:{\mathbf Z}_Y\to \sh^0(E^1_n(0)_Y(D_Y)))$ that is compatible with the map ${\mathbf Z}_Y\to \sh^0(S^{1}(0)_Y(D_Y))$. For that notice that the composition
$${\mathbf Z}_Y\stackrel{c_0}{\to} S^{1}(0)_Y(D_Y)\stackrel{\alpha}{\to}i_{Y*}i^*_Y\R j_{Y*}{\mathbf Z}/p^n(1)^{\prime}(-1)
$$
is compatible with the natural map $pc_0^{\eet}:{\mathbf Z}_Y\to \R j_{Y*}{\mathbf Z}/p^n(1)^{\prime}(-1)$. Hence the composition
$${\mathbf Z}_Y\stackrel{c_0}{\to} S^{1}(0)_Y(D_Y)\stackrel{\alpha}{\to}i_{Y*}i^*_Y\R j_{Y*}{\mathbf Z}/p^n(1)^{\prime}(-1)\to j_{Y\eet !}\R j^{\prime}_{Y*}{\mathbf Z}/p^n(1)^{\prime}(-1)[1]
$$
is homotopic to zero and we obtain its unique factorization $c_0:{\mathbf Z}_Y\to E^1_n(0)_Y(D_Y)$.  It is clear that this map factors through $F^1_n(0)_X(Y\cup D)$ hence the last statement of the lemma.
\end{proof} 
\subsubsection{Projective space theorem and homotopy invariance}
   We will now discuss certain versions of projective space theorem for syntomic cohomology. 
For large primes $p$, we have the following projective space theorem.
\begin{lemma}
\label{projective}Let $X$ and $D:=D^{\prime}$ be as above. 
Let $\se$ be a locally free sheave of rank $d+1$, $d\geq 0$. Consider the associated projective bundle $\pi:{\mathbf P}(\se)\to X$. Then, we have the following quasi-isomorphism
\begin{align*}
\bigoplus_{i=0}^d{c}_1^{\synt}(\so(1))^i\cup\pi^*: \quad 
\bigoplus_{i=0}^d S_n(r-i)_X(D)[-2i] \stackrel{\sim}{\to} \R \pi_*
S_n(r)_{{\mathbf P}(\se)}(\pi^{-1}(D)), \quad 0\leq d \leq r\leq p-1.
\end{align*}
Here, the class ${c}_1^{\synt}(\so(1))\in H^2({\mathbf P}(\se), S_n(1)(\pi^{-1}(D)))$ refers to the class of the tautological bundle on ${\mathbf P}(\se)$.

 We have an analogous quasi-isomorphism for complexes $E_n(r)$.
\end{lemma}
\begin{proof}
In the absence of the divisor $D$, for the complexes $S_n(r)$ this  follows easily from the (filtered) projective space theorem for crystalline cohomology  and the compatibility of the actions of syntomic Chern classes with the cohomology long exact sequence associated to  (\ref{exact}) \cite[I.4.3]{Gr}. For complexes $E_n(r)$, this follows from the above and the projective space theorem for $j_{\eet !}\R j^{\prime}_*{\mathbf Z}/p^n(r)$.

 In general, write $D=\cup_{i=1}^{i=m}D_i$, $Y=D_1$, $D^{\prime}=\cup_{i=2}^{i=m}D_i$, $D_Y=D^{\prime}\cap Y$, where $D_i$ is an  irreducible component of $D$. We will argue by induction on $m$; the case of $m=0$
being known.  
Using Lemma \ref{localization}, we get the following map of distinguished triangles
$$
\xymatrix@C=40pt{\bigoplus_{i=0}^d i_*S_n(r-1-i)_Y(D_Y)[-2i-2]\ar[d]^{i_!}\ar[rr]^{\bigoplus_{i=0}^d\xi_Y(D_Y)^i\cup\pi^*}_{\sim} & &i_*\R \pi_*
S_n(r-1)_{{\mathbf P}(\se)_Y}(\pi^{-1}(D_Y))[-2]\ar[d]^{i_!}\\
\bigoplus_{i=0}^d S_n(r-i)_X(D)[-2i] \ar[d]\ar[rr]^{\bigoplus_{i=0}^d\xi_X(D)^i\cup\pi^*} && \R \pi_*
S_n(r)_{{\mathbf P}(\se)}(\pi^{-1}(D))\ar[d]\\
\bigoplus_{i=0}^d S_n(r-i)_X(D^{\prime})[-2i] \ar[rr]^{\bigoplus_{i=0}^d\xi_X(D^{\prime})^i\cup\pi^*}_{\sim} && \R \pi_*
S_n(r)_{{\mathbf P}(\se)}(\pi^{-1}(D^{\prime}))
}
$$
where $i:Y\hookrightarrow X$ is the natural immersion and we wrote $\xi_X(D),\xi_Y(D_Y),\xi_X(D^{\prime})$ for 
the Chern classes of the tautological bundle in $H^2({\mathbf P}(\se)_Y, S_n(1)(\pi^{-1}(D_Y)))$,
$H^2({\mathbf P}(\se), S_n(1)(\pi^{-1}(D)))$, and $H^2({\mathbf P}(\se), S_n(1)(\pi^{-1}(D^{\prime})))$, respectively. The upper square commutes by the projection formula from Lemma \ref{projsynt} and Tor-independence from Lemma \ref{torind1}. 

  By the inductive assumption the top and bottom horizontal arrows are quasi-isomorphisms. So is then the middle one, as wanted.

  The argument for $E_n(r)$ is analogous.
\end{proof}

  As a corollary we get the following syntomic cohomology version of homotopy invariance.
\begin{lemma}
\label{homotopy}Let $X$, $D$ be as in Lemma \ref{projective}. 
 Let $\se$ be a locally free sheaf of rank $d+1$, $d\geq 0$. Consider its projectivization $\pi:{\mathbf P}(\se\oplus \so_X)\to X$ and its hyperplane at infinity $H_{\infty}$. Then, for $a\geq 0$, we have the following quasi-isomorphisms.
\begin{align*}
\pi^*: & \quad S_n^{a}(r)_X( D) \stackrel{\sim}{\to} \R \pi_*
S^{a}_n(r)_{{\mathbf P}(\se\oplus \so_X)}(\pi^{-1}(D)\cup H_{\infty}),\quad 0\leq r,\\
 \pi^*: & \quad S_n(r)_X( D) \stackrel{\sim}{\to} \R \pi_*
S_n(r)_{{\mathbf P}(\se\oplus \so_X)}( \pi^{-1}(D)\cup H_{\infty}), \quad 0\leq d \leq r\leq p-1.
\end{align*}
  We have analogous quasi-isomorphisms for the complexes $E_n(r)$, $E^{\prime}_n(r)$, and $E^1_n(r)$.

\end{lemma}
\begin{proof}
 The second quasi-isomorphism  follows  from the projective space theorem from Lemma \ref{projective}. Indeed, consider the following commutative diagram of distinguished triangles.
$$\xymatrix{
\R \pi_*i_*S_n(r-1)_{H_{\infty}}(\pi^{-1}(D)\cap H_{\infty})[-2]\ar[r]^-{i_!} & \R \pi_*S_n(r)_{{\mathbf P}}( \pi^{-1}(D))\ar[r] \ar@{=}[d]
&
\R \pi_*
S_n(r)_{{\mathbf P}}( \pi^{-1}(D)\cup H_{\infty})\ar@{=}[d]\\
\R q_*S_n(r-1)_{{\mathbf P}(\se)}(q^{-1}(D))[-2]\ar[r]^-{\xi\cup \wt{\pi}^*}\ar[u]^{\wr}_{i^*\pi^*} & \R {\pi}_*S_n(r)_{{\mathbf P}}( \pi^{-1}(D))\ar[r] &
\R {\pi}_*
S_n(r)_{{\mathbf P}}( \pi^{-1}(D)\cup H_{\infty})\\
\oplus_{i=0}^{i=d-1}S_n(r-1-i)_{X}(D)[-2-2i]\ar[r]^-{\xi^{i+1}\cup \pi^*}\ar[u]^{\oplus_i \xi_1^i\cup q^*}_{\wr} & \R \pi_*S_n(r)_{{\mathbf P}}( \pi^{-1}(D))\ar[r]^{p_1}\ar@{=}[u] &
S_n(r)_{X}(D)\ar[u]^{\pi^*},}
$$
where we put $\xi=c_1^{\synt}(\so(H_{\infty}))$,  $\xi_1=c_1^{\synt}(\so(1))\in H^2({\mathbf P}(\se),S_n(1)(q^{-1}(D))$, ${\mathbf P}={\mathbf P}(\se\oplus \so_X)$. We have $\xi=\wt{\pi}^*(\xi)$.  We named the maps as in the following diagram
$$
\xymatrix{
H_{\infty}\ar[r]^-{i}\ar[rd]^{\sim} & {\mathbf P}(\se\oplus \so_X)\ar[d]^{\wt{\pi}}\\
& {\mathbf P}(\se)\ar[r]^{q} & X
}
$$

The top row in the big diagram is the Gysin sequence.  The upper left square commutes
because by the projection formula from Lemma \ref{projsynt} and by Lemma \ref{Chern} we have 
$$i_!i^*\pi^*=i_!(1)\cup \pi^* = c_1^{\synt}(\so(H_{\infty}))\cup \pi^*.$$
In the bottom row, the left vertical map is a quasi-isomorphism by the projective space theorem for $q: {\mathbf P}(\se)\to X$ (Lemma \ref{projective}) and the map $p_1$ comes from the same projective space theorem. The pullback map $\pi^*$ is its right inverse. 
The second quasi-isomorphism of the lemma follows now easily.

  For the first quasi-isomorphism consider the following commutative diagram of distinguished triangles.
$$\xymatrix@C=42pt{
\R \pi_*S^{a}_n(r)_{{\mathbf P}}( \pi^{-1}(D)\cup H_{\infty})\ar[r] &
\R \pi_*\sj^{[r]}_{{\mathbf P},n}( \pi^{-1}(D)\cup H_{\infty})\ar[r]^{p^a(p^{r}-1)\phi} & \R \pi_*\so^{\crr}_{{\mathbf P},n}( \pi^{-1}(D)\cup H_{\infty})\\
S^{a}_n(r)_{X}(D) \ar[r] \ar[u]^{\pi^*} & \sj^{[r]}_{X,n}(D) \ar[u]^{\pi^*}_{\wr} \ar[r]^{p^a(p^{r}-1)\phi}& \so^{\crr}_{X,n}(D) \ar[u]^{\pi^*}_{\wr}.
}
$$
Here we put $\sj_n^{[r]}=\R \varepsilon_*\sj_n^{[r]}$ and $\so_n^{\crr}=\R \varepsilon_*\so_n^{\crr}$, 
where $\varepsilon$ is the natural projection from the syntomic to the Zariski site. The two marked maps are quasi-isomorphisms because we have the (filtered) projective space theorem for crystalline cohomology and we can use  the same argument as we did above to prove the second quasi-isomorphism of the lemma. The first quasi-isomorphism of the lemma follows.

  The arguments for the complexes $E_n(r)$, $E^{\prime}_n(r)$, and $E^1_n(r)$ are analogous (but use also the proper base change theorem and homotopy property of \'etale cohomology).
\end{proof}
For a general prime $p$ we still have a cell decomposition of the cohomology of the projective bundle.
We will prove a first step of it.
\begin{lemma}
\label{projective1}Let $X$, $D$ be as in Lemma \ref{projective}. Let $\se$ be a locally free sheaf of rank $d+1$, $d\geq 0$. Consider its projectivization $\pi:{\mathbf P}(\se\oplus \so_X)\to X$ and its hyperplane at infinity $H_{\infty}$.  Then, for $a\geq 0$, we have the following quasi-isomorphism
$$
i_!\oplus \pi^*:\quad 
\R \pi_*i_*S^{a+1}_n(r-1)_{H_{\infty}}(\pi^{-1}(D)\cap H_{\infty})[-2]\oplus  S^{a}_n(r)_X(D)\stackrel{\sim}{\to} 
\R \pi_*S^{a}_n(r)_{{\mathbf P}(\se\oplus \so_X)}(\pi^{-1}(D)), \quad 1\leq r.
$$
Here $i: H_{\infty }\hookrightarrow {\mathbb P}(\se\oplus \so_X)$ is the natural closed immersion.

 In particular,  we have the following quasi-isomorphism
$$
i_!(\pi i)^*\oplus \pi^*:\quad 
S^{a+1}_n(r-1)_X(D)[-2]\oplus  S^{a}_n(r)_X(D)\stackrel{\sim}{\to} 
\R \pi_*S^{a}_n(r)_{{\mathbf P}(\se\oplus \so_X)}(\pi^{-1}(D)), \quad 1\leq r.
$$
Similarly for the complexes $E^{\prime}_n(r)$ and $E^1_n(r)$.
\end{lemma}
\begin{proof}
We have the following commutative diagram of distinguished triangles.
$$\xymatrix{
\R \pi_*i_*S^{a+1}_n(r-1)_{H_{\infty}}(D_{\infty})[-2]\ar[r]^{i_!} & \R \pi_*S^{a}_n(r)_{{\mathbf P}}(\pi^{-1}(D))\ar[r] & \R \pi_*S^{a}_n(r)_{{\mathbf P}}(\pi^{-1}(D)\cup H_{\infty})\\
\R \pi_*i_*S^{a+1}_n(r-1)_{H_{\infty}}(D_{\infty})[-2]\ar@{=}[u]\ar[r] & \R \pi_*i_*S^{a+1}_n(r-1)_{H_{\infty}}(D_{\infty})[-2]\oplus  S^{a}_n(r)_X(D)\ar[u]^{i_!\oplus \pi^*}\ar[r] &   S^{a}_n(r)_X(D)\ar[u]^{\pi^*}_{\wr}
}
$$
where we put ${\mathbf P}={\mathbf P}(\se\oplus \so_X)$, $D_{\infty}=\pi^{-1}(D)\cap H_{\infty}$. The top triangle is induced by the Gysin sequence.
The right vertical arrow is a quasi-isomorphism by Lemma \ref{homotopy}. The first statement of the lemma follows. 
 
  For the second, notice that  the natural projection 
$\pi i: H_{\infty}\to X$ is an isomorphism and we have a quasi-isomorphism
$$
(\pi i)^*:\quad S^{a+1}_n(r-1)_X(D)\stackrel{\sim}{\to}
\R \pi_*i_*S^{a+1}_n(r-1)_{H_{\infty}}(D_{\infty}).
$$
The arguments for the complexes $E^{\prime}_n(r)$ and $E^1_n(r)$ are analogous.
\end{proof}

\subsection{Compatibility with Gysin sequences}
 We will show in this section that the  syntomic universal Chern class maps are compatible with Gysin sequences.
We will present the arguments for  $E^{\prime}$-syntomic-\'etale cohomology. In the case of $E$-syntomic-\'etale cohomology the arguments are very similar. 

    Let $X$ be a semistable scheme  over $\so_K^{\times}$ or a semistable scheme $\so_K$ with smooth special fiber. Let $D$ be the horizontal divisor. Assume that  $D=\cup_{i=1}^{m}D_i$, $m\geq 1$,  is a union of $m$  irreducible components $D_i$.
 Note that each scheme $D_i$ with the induced log-structure is of the same type as $X$  (with  at most $m-1$  components in the divisor at infinity). 
 Set  
$D^{\prime}=\cup_{i=2}^{m}D_i$, $i:D_1\hookrightarrow X$, $D^{\prime}_1=D_1\cap D^{\prime}$. 
The pairs $(X,D^{\prime})$ and $(D_1,D_1^{\prime})$ are of the same type as the pair $(X,D)$ we started with but  with at most $m-1$ irreducible divisors at infinity.  Fix $N$ as in Lemma \ref{kolo1}. 
Consider the following Gysin-type diagram 
with columns being homotopy cofiber sequences
\begin{equation}
\label{induction}
\begin{CD}
i_*G_{D_1}( D^{\prime}_1) @>[p^{Ni}][p^{N(i+1)}]\gamma^{\prime}[p^{2N}]P(D_1/X;{\wt{C}}^{\synt}_{D_1\setminus D^{\prime}_1})>>
\{0\}\times \prod_{i\geq 1}\sk(2i-2,i_*\wt{F}_n^{1}(i-1)_{X}(D^{\prime}))\\
@VVi_* V @VVi_! V\\
G_X( D^{\prime})@>[p^{Ni}][p^{2Ni}]\wt{C}^{\synt}_{X\setminus D^{\prime}}>>{\mathbf Z}\times \prod_{i\geq 1}\sk(2i,\wt{E}_n^{\prime}(i)_X( D^{\prime}))\\
@VVV @VVV\\
G_X( D)@>[p^{Ni}][p^{2Ni}]\wt{C}^{\synt}_{X\setminus D}>>{\mathbf Z}\times \prod_{i\geq 1}\sk(2i,\wt{E}_n^{\prime}(i)_X(D)).
\end{CD}
\end{equation}
Here the  Chern classes are associated to $N$. We skip this index from the notation. 
The polynomials $P(D_1/X;\wt{C}^{\synt}_{D_1\setminus D^{\prime}_1})$ are given in degree $k$, $k\geq 0$, by certain  polynomials  
$$
P_k(D_1/X;\wt{C}^{\synt}_{D_1\setminus D^{\prime}_1}):\quad 
G_{D_1}(D^{\prime}_1)\to \sk(2k,\wt{E}_n^{\prime}(k)_{D_1}(D^{\prime}_1))
$$
that are equal to $0$ for $k=0$ and, for  $k\geq 1$, are obtained by applying the universal polynomials 
$P_k(U_1;T_0,T_1,\ldots, T_{k-1})$ with integer coefficients \cite[1.3]{Jou}, \cite[II.4]{FLL} to the augmented universal Chern classes
$$\wt{C}^{\synt}_{D_1\setminus D^{\prime}_1,i}:\quad G_{D_1}(  D^{\prime}_1)\to
\sk(2i,\wt{E}_n^{\prime}(i)_{D_1}( D^{\prime}_1)), \quad 0\leq i\leq k-1,
$$
and to the first Chern class of the conormal sheaf
$c_1^{\synt}(\sn^{\vee}_{D_1/X})\in H^2(D_1,E^{\prime}_n(1)_{D_1})$. We set $\gamma^{\prime}=\beta\omega^{\prime}$, where $\beta$ is the map defined in the proof of Lemma \ref{projsynt}. 
\begin{theorem}
\label{logChern}
Diagram \ref{induction} commutes in the homotopy category.
\end{theorem}
\begin{proof}
  We start with the upper square.
\begin{lemma}(Grothendieck-Riemann-Roch)
\label{deformation}
There exists a  homotopy  $$[p^{2Ni}]\wt{C}^{\synt}_{X\setminus D^{\prime}}i_*\simeq  
[p^{N(i+1)}]i_!\gamma^{\prime}[p^{2N}]P(D_1/X;\wt{C}^{\synt}_{D_1\setminus D^{\prime}_1}).
$$
\end{lemma}
\begin{proof}
The construction of the homotopy is implicit in Gillet's proof of Theorem 3.1 in \cite{Gi}. Just like in the classical situation 
one argues by deforming
the given  closed immersion to a regular zero section of a vector bundle of the right rank.\\
{\em  (1) The case of zero section}\\
 So we start with a special case of a closed immersion $i:Y\hookrightarrow X$, $Y=D_1$, which is the zero section of a projective bundle $X={\mathbb P}(\sn\oplus \so_Y)$, where $\sn$ is an invertible sheaf on $Y$. The sheaf $\sn$ is the conormal sheaf of the immersion. Let $\pi: X\to Y$ be the projection. Write $D_Y=D^{\prime}\cap Y$ and $D_X=D^{\prime}$. We assume that
$\pi^{-1}(D_Y)=D_X$.
Functoriality of log-$K$-theory and the projection formula in log-$K$-theory (Lemma \ref{Kproj}) yield the following homotopies
$$ \wt{C}^{\synt}_{X\setminus D_X}i_*\simeq \wt{C}^{\synt}_{X\setminus D_X}i_*i^*\pi^*\simeq
\wt{C}^{\synt}_{X\setminus D_X}([i_*\so_Y]\wedge \pi^*),$$
where $[i_*\so_Y]\wedge $ refers to the action of the class of $i_*\so_Y$ in $K_0(X)$ on $G_X( D_X)$.
The exact sequence
\begin{equation}
 \label{Koszul}
0\to \pi^*\sn\to \so_X\to i_*\so_Y\to 0
\end{equation}
yields a homotopy $\id\simeq ([\pi^*\sn]+ [i_*\so_Y])$. Hence by functoriality and  by compatibility of Chern classes with  the action of $K$-theory, as well as  by Lemma \ref{Zariskiprop} 
we have 
\begin{align*}
\wt{C}^{\synt}_{X\setminus D_X}i_* & \simeq \wt{C}^{\synt}_{X\setminus D_X}((\id-[\pi^*\sn])\wedge \pi^*)\simeq
\wt{C}^{\synt}_{X\setminus D_X}(\pi^*\lambda_{-1}(\sn)\wedge\pi^*)
  \simeq \pi^*\wt{C}^{\synt}_{Y\setminus D_Y}(\lambda_{-1}(\sn)\wedge)\\
& \simeq\pi^*(\wt{C}^{\synt}_{Y\setminus D_Y}(\lambda_{-1}(\sn))\star \wt{C}^{\synt}_{Y\setminus D_Y})
\simeq\pi^*(\lambda_{-1}(\wt{C}^{\synt}_{Y\setminus D_Y}(\sn))\star \wt{C}^{\synt}_{Y\setminus D_Y}),
\end{align*}
where for a locally free sheaf $\sll$ on a scheme $Z$ we wrote 
$\lambda_{-1}(\sll)=[\so_Z]-[\sll]\in K_0(Z)$. The equlity
$$\wt{C}^{\synt}_{Y\setminus D_Y}(\lambda_{-1}(\sn))\simeq \lambda_{-1}(\wt{C}^{\synt}_{Y\setminus D_Y}(\sn))
$$ holds since the classes $\wt{C}^{\synt}_{Y\setminus D_Y}$ and $[p^{2Ni}]\wt{C}^{\synt}_{Y}$ are compatible and we have the Whitney sum formula for the (classical) syntomic-\'etale classes $C^{\synt}_{Y}$.

  We claim that
\begin{equation}
\label{Jouanolou}
\lambda_{-1}(\wt{C}^{\synt}_{Y\setminus D_Y}(\sn))\star \wt{C}^{\synt}_{Y\setminus D_Y}\simeq P(Y/X;\wt{C}^{\synt}_{Y\setminus D_Y})[p^{2N}]c^{\synt}_1(\sn^{\vee}).
\end{equation}
This will reduce to a classical result as soon as we put this equation in the right context.
Consider the graded ring
$$
\prod _{i\geq 1}[G_Y( D_Y), \sk(2i,\tilde{E}^{\prime}_n(i){}_{Y}( D_Y))]
$$
with addition and multiplication induced from $\sk(2i,\tilde{E}^{\prime}_n(i){}_{Y}( D_Y))$.
Consider the following abelian groups
\begin{align}
\label{ring}
H^{2i}(Y,G_Y( D_Y),E^{\prime}_n(i)_Y(D_Y)): & =
[G_Y(D_Y), \sk(2i,\tilde{E}^{\prime}_n(i){}_{Y}( D_Y))],\\
\wt{H}^*(Y,G_Y(D_Y),E^{\prime}_n(*)_Y( D_Y)): & =
{\mathbf Z}\times
\{1\}\times  \prod _{i\geq 1}[G_Y( D_Y), \sk(2i,\tilde{E}^{\prime}_n(i){}_{Y}( D_Y))].\notag
\end{align}
Recall that \cite[Exp.0]{S6} $\wt{H}^*(Y,G_Y( D_Y),E^{\prime}_n(*)_Y(D_Y))$ has a structure of a  $\lambda$-ring with involution, defined as follows. 
 Write an element of 
$ \wt{H}^*(Y,G_Y(D_Y),E^{\prime}_n(*)_Y(D_Y))$ as $[n,x]$ for 
$x=1+x_1+\ldots +x_k+\ldots.$
Then
\begin{align*}
[n,x] & + [m,y]=[n+m,xy],\\
[n,x]&\star [m,y]
=[nm, x^my^nx\ast y],
\end{align*}
where
$$(1+\sum_{i\geq 1}x_i)\ast (1+\sum_{i\geq 1}y_i)=1+\sum_{k\geq 0}Q_k(x_1,\ldots,x_k;y_1,\ldots,y_k)
$$
for certain universal polynomials $Q_k$ with integral coefficients. 

  The $\lambda$-structure on $\wt{H}^*(Y,G_Y( D_Y),E^{\prime}_n(*)_Y( D_Y))$ is given by the following operations $\lambda^k$ 
\begin{equation}
\label{formulas}
\begin{cases}
& \lambda^k[0,x] =[0,\lambda^kx], \quad (\lambda^kx)_n=Q_{k,n}(x_1,\ldots,x_n),\\
& \lambda^k[n,1] =[\lambda^kn,1],
\end{cases}
\end{equation}
where $Q_{k,n}$ are certain universal polynomials with integral coefficients and $\lambda^kn$ on ${\mathbf Z}$ is the canonical one.
The involution is $\psi^{-1}[n,1+\sum_{i>0}x_i]=[n,1+\sum_{i>0}(-1)^ix_i]$.

 Now we can rephrase the equality (\ref{Jouanolou}) as the following equality in the $\lambda$-ring $\wt{H}^*(Y,G_Y( D_Y),E^{\prime}_n(*)_Y( D_Y))$ (where we skipped the $Y\setminus D_Y$ subscript).
$$
(\lambda_{-1}([1,1+C^{\synt}_1(\sn^{\vee})])\star [\wt{C}^{\synt}_0,1+C^{\synt}_1+C^{\synt}_2\ldots ])_k=
-P_k((C^{\synt}_1(\sn^{\vee})),(\wt{C}^{\synt}_0,C^{\synt}_1,\ldots, C^{\synt}_k))C^{\synt}_1(\sn^{\vee}).
$$
But this is classical \cite[1.3]{Jou}.

   By Lemma \ref{Chern} we have $c_1^{\synt}(\pi^*\sn^{\vee})= c_1^{\synt}(\so(Y))=i_!(1)$
for $1\in H^0(Y,F^1_n(0)_X(D_X))$. Note that the projection formula in syntomic-\'etale cohomology from Lemma \ref{projsynt} 
implies that 
$$
[p^{N(i-1)}]i_!(1)=i_!\beta \omega^{\prime}i^*:\quad E^{\prime}_n(i)_X(D_X)\to E^{\prime}_n(i)_X(D_X)[2].
$$
Using this and functoriality we get
\begin{align*}
[p^{N(i-1)}]\wt{C}^{\synt}_{X\setminus D_X}i_* 
& \simeq 
\pi^*(P(Y/X;\wt{C}^{\synt}_{Y\setminus D_Y}))[p^{2N}][p^{N(i-1)}]i_!(1)
\simeq
i_!\gamma^{\prime}i^*\pi^*[p^{2N}]P(Y/X;\wt{C}^{\synt}_{Y\setminus D_Y})\\
  & \simeq i_!\gamma^{\prime}[p^{2N}]P(Y/X;\wt{C}^{\synt}_{Y\setminus D_Y}),
\end{align*}
as wanted.\\
{\em (2) General case}

  We pass to the general case by "deformation to the normal cone". We will now recall some aspects of this construction \cite[IV.5]{FLL}. Consider the blow-up $\pi: W\to  {\mathbb P}^1_X$ along $Y\times\{\infty\}$, $Y=D_1$. Then we have the following deformation  diagram.
$$
\xymatrix{
 Y_0=Y \ar[rr]^{i_0=i}\ar[d]^{j_0} && X=W_0\ar[d]^{j_0} \ar[rr]^{\id_X} && X\ar[d]^{j_0}\\
{\mathbb P}^1_Y \ar[rr]^{\overline{i}} && W \ar[rr]^{\pi} && {\mathbb P}^1_X\\
Y_{\infty}=Y \ar[u]^{j_{\infty}}\ar[r]^-{i_{\infty}} & {\mathbb P}(\sn_{Y/X}^{\vee}\oplus\so_Y)\ar[ur]^t &+& \wt{X}\ar[ul]^{l} \ar[r]^{\id_X} & X\ar[u]^{j_{\infty}}.
}
$$
This diagram is commutative and all the squares are cartesian.
Here $\wt{X}$ is the blow-up of $X$ at $Y$, hence $\wt{X}=X$. The sheaf $\sn_{Y/X}$ is the normal sheaf of $Y$ in $X$. We write $N_{Y/X}$ for the associated vector bundle which is canonically isomorphic to the normal cone of $Y$ in $X$. Note that ${\mathbb P}(\sn_{Y/X}^{\vee}\oplus\so_Y)$ is the projectivization of $N_{Y/X}$. 
The embedding 
$i_{\infty}: Y \hookrightarrow  W_{\infty}=\wt{X}+{\mathbb P}(\sn_{Y/X}^{\vee}\oplus\so_Y)$ is the zero-section embedding of $Y$ in $N_{Y/X}$, followed by the canonical open embedding of $N_{Y/X}$ in ${\mathbb P}(\sn_{Y/X}^{\vee}\oplus\so_Y)$. In general, the divisors ${\mathbb P}(\sn_{Y/X}^{\vee}\oplus\so_Y)$ and $\wt{X}$ intersect in the scheme ${\mathbb P}(\sn_{Y/X}^{\vee})$, which is embedded as the hyperplane at infinity in ${\mathbb P}(\sn_{Y/X}^{\vee}\oplus\so_Y)$, and as the exceptional divisor in $\wt{X}$. Since our embedding $Y\hookrightarrow X$ is of codimension $1$, the scheme ${\mathbb P}(\sn_{Y/X}^{\vee})$ is just $\Spec(W(k))$. Let $\rho:{\mathbb P}^1_Y\to Y$ be the projection and $q:W\to X$ the canonical map. In the above diagram we put log-structures on schemes via the following divisors
\begin{align*}
 D_X=D^{\prime},\quad D_Y=Y\cap D^{\prime},\quad
 D_{{\mathbb P}^1_Y}=\rho^{-1}(D_Y)=D_Y\times {\mathbb P}^1,\quad D_W=q^{-1}(D_X),
\quad D_{{\mathbb P}(\sn_{Y/X}^{\vee}\oplus\so_Y)}=t^{-1}(D_W).
\end{align*}
They make all the maps in the squares in the above diagram strict, i.e., the log-structures pullback to log-structures.

   First, we claim that there is a homotopy 
\begin{equation}
 \label{deform}
j_{0!}\gamma^{\prime}i_!\gamma^{\prime}[p^{2N}]P(Y_0/W_0;\wt{C}^{\synt}_{Y_0\setminus D_Y})\simeq j_{0!}\gamma^{\prime}
[p^{N(i-1)}]\wt{C}_{X\setminus D_X}i_*.
\end{equation} 
We will show this by deforming both sides to infinity  and using the special case treated above.
By functoriality of syntomic-\'etale cohomology and Lemma \ref{torind1} we have
\begin{align*}
j_{0!}\gamma^{\prime}i_!\gamma^{\prime}P(Y_0/W_0;\wt{C}^{\synt}_{Y_0}) & \simeq j_{0!}\gamma^{\prime}i_!j^*_0\rho^*\gamma^{\prime}P(Y_0/W_0;\wt{C}^{\synt}_{Y_0}) 
\simeq j_{0!}\gamma^{\prime}j_0^*\overline{i}_!\rho^*\gamma^{\prime}P(Y_0/W_0;\wt{C}^{\synt}_{Y_0}).
\end{align*}
Here and below we skip the divisors from the notation of Chern classes if understood.
The projection formula in syntomic-\'etale cohomology (Lemma \ref{projsynt}) and functoriality imply that
\begin{align*}
j_{0!}\gamma^{\prime}i_!\gamma^{\prime}P(Y_0/W_0;\wt{C}^{\synt}_{Y_0}) &\simeq  [p^{N(i-1)}]\overline{i}_!\rho^*\gamma^{\prime}P(Y_0/W_0;\wt{C}^{\synt}_{Y_0})c^{\synt}_1(j_{0*}\so_{W_0})\simeq  
[p^{N(i-1)}]\overline{i}_!\gamma^{\prime}P({\mathbb P}^1_Y/W;\wt{C}^{\synt}_{{\mathbb P}^1_Y})c^{\synt}_1(j_{0*}\so_{W_0}).
\end{align*}
Since the divisors $W_0$ and $W_{\infty}$ on $W$ are linearly equivalent we have $[j_{0*}\so_{W_0}]= [j_{\infty*}\so_{W_{\infty}}]=[l_*\so_X]+[t_*\so_{{\mathbb P}(\sn^{\vee}_{Y/X}\oplus\so_Y)}]$ in $K_0(W)$. This implies
\begin{align*}
j_{0!}\gamma^{\prime}i_!\gamma^{\prime}P(Y_0/W_0;\wt{C}^{\synt}_{Y_0}) &\simeq 
[p^{N(i-1)}]\overline{i}_!\gamma^{\prime}P({\mathbb P}^1_Y/W;\wt{C}^{\synt}_{{\mathbb P}^1_Y})c^{\synt}_1(j_{\infty*}\so_{W_{\infty}})\\
& \simeq
[p^{N(i-1)}]\overline{i}_!\gamma^{\prime}P({\mathbb P}^1_Y/W;\wt{C}^{\synt}_{{\mathbb P}^1_Y})(c^{\synt}_1(l_*\so_{X})+c^{\synt}_1(t_*\so_{{\mathbb P}(\sn^{\vee}_{Y/X}\oplus\so_Y)}))
\end{align*}
The projection formula in syntomic-\'etale cohomology (Lemma \ref{projsynt}), the fact that $l^*\overline{i}_!=0$, Lemma \ref{torind1}, and functoriality of Chern classes yield
\begin{align*}
j_{0!}\gamma^{\prime}i_!\gamma^{\prime}P(Y_0/W_0;\wt{C}^{\synt}_{Y_0})& \simeq l_!
\gamma^{\prime}l^*\overline{i}_!\gamma^{\prime}P({\mathbb P}^1_Y/W;\wt{C}^{\synt}_{{\mathbb P}^1_Y})
\vee t_!\gamma^{\prime}t^*\overline{i}_*\gamma^{\prime}P({\mathbb P}^1_Y/W;\wt{C}^{\synt}_{{\mathbb P}^1_Y})\\ &\simeq 
t_!\gamma^{\prime}i_{\infty!}j^*_{\infty}\gamma^{\prime}P({\mathbb P}^1_Y/W;\wt{C}^{\synt}_{{\mathbb P}^1_Y})
\simeq
 t_!\gamma^{\prime}i_{\infty!}\gamma^{\prime}P(Y/{\mathbb P}(\sn^{\vee}_{Y/X}\oplus\so_Y);\wt{C}^{\synt}_{Y})
\end{align*}
We can now apply the computations we did in the special case to the embedding 
$i_{\infty}: Y\to {\mathbb P}(\sn^{\vee}_{Y/X}\oplus\so_Y)$
to conclude that
\begin{align*}
j_{0!}\gamma^{\prime}i_!\gamma^{\prime}[p^{2N}]P(Y_0/W_0;\wt{C}^{\synt}_{Y_0})& \simeq t_!\gamma^{\prime}
[p^{N(i-1)}]\wt{C}^{\synt}_{{\mathbb P}(\sn^{\vee}_{Y/X}\oplus\so_Y)}i_{\infty,*}
\end{align*}

  Similarly, by functoriality in log-$K$-theory and Lemma \ref{torindep} we get
\begin{align*}
j_{0!}\gamma^{\prime}[p^{N(i-1)}]\wt{C}^{\synt}_Xi_* & \simeq 
j_{0!}\gamma^{\prime}[p^{N(i-1)}]\wt{C}^{\synt}_{W_0}i_*j^*_0\rho^*\simeq j_{0!}\gamma^{\prime}[p^{N(i-1)}]\wt{C}^{\synt}_{W_0}j_0^*\overline{i}_*\rho^* 
\end{align*}
Next, by functoriality of Chern classes and the projection formula in syntomic-\'etale cohomology from
 Lemma \ref{projsynt} we get 
\begin{align*}
j_{0!}\gamma^{\prime}[p^{N(i-1)}]\wt{C}^{\synt}_Xi_* & \simeq j_{0!}\gamma^{\prime}j^*_{0}[p^{N(i-1)}]\wt{C}^{\synt}_{W}\overline{i}_*\rho^*
\simeq [p^{2N(i-1)}]\wt{C}_{W}^{\synt}\overline{i}_*\rho^*c^{\synt}_1(j_{0*}\so_{W_0})\\
 & \simeq [p^{2N(i-1)}]\wt{C}_{W}^{\synt}\overline{i}_*\rho^*c^{\synt}_1(j_{\infty*}\so_{W_{\infty}})
 \simeq 
[p^{2N(i-1)}]\wt{C}_{W}^{\synt}\overline{i}_*\rho^*(c^{\synt}_1(l_*\so_{X})\vee c^{\synt}_1(t_*\so_{{\mathbb P}(\sn^{\vee}_{Y/X}\oplus\so_Y)}))\\
 & \simeq t_!\gamma^{\prime}t^*[p^{N(i-1)}]\wt{C}^{\synt}_{W}\overline{i}_*\rho^*\vee l_!\gamma^{\prime}l^*[p^{N(i-1)}]\wt{C}^{\synt}_{W}\overline{i}_!\rho^*
\end{align*}
Since $l^*\overline{i}_*=0$, functoriality of Chern classes and of log-$K$-theory together with Lemma \ref{torindep} yield
\begin{align*}
j_{0!}\gamma^{\prime}[p^{N(i-1)}]\wt{C}^{\synt}_Xi_* &\simeq t_!\gamma^{\prime}[p^{N(i-1)}]\wt{C}^{\synt}_{{\mathbb P}(\sn^{\vee}_{Y/X}\oplus\so_Y)}t^*\overline{i}_*\rho^*\vee l_!\gamma^{\prime}[p^{N(i-1)}]\wt{C}^{\synt}_{X}l^*\overline{i}_*\rho^*\\
& \simeq t_!\gamma^{\prime}[p^{N(i-1)}]\wt{C}^{\synt}_{{\mathbb P}(\sn^{\vee}_{Y/X}\oplus\so_Y)}i_{\infty*}j^*_{\infty}\rho^*
\simeq t_!\gamma^{\prime}[p^{N(i-1)}]\wt{C}^{\synt}_{{\mathbb P}(\sn^{\vee}_{Y/X}\oplus\so_Y)}i_{\infty*}
\end{align*}
So we got that $j_{0!}\gamma^{\prime}i_!\gamma^{\prime}[p^{2N}]P(Y_0/W_0;\wt{C}^{\synt}_{Y_0})\simeq j_{0!}\gamma^{\prime}[p^{N(i-1)}]\wt{C}^{\synt}_Xi_*$, as wanted.

    Lemma \ref{finish} below and the projective space theorem (Lemma (\ref{projective1})) imply that there is a map
$$q_*:\quad\prod_{i\geq 1}\sk(2i+2,\R q_*\wt{E}^{\prime}_{n}(i+1)_W(D_W))\to 
\prod_{i\geq 1}\sk(2i,\wt{E}^{1}_{n}(i)_X( D_X))
$$
such that $q_*j_{0!}\beta\simeq [p^{Ni}]$. 
Hence, applying first $q_*$ then $\omega_0$ to both sides of the equation (\ref{deform}) we get 
$$[p^{N(i+1)}]i_!\omega^{\prime}[p^{2N}]P(Y_0/W_0;\wt{C}^{\synt}_{Y_0})\simeq [p^{2Ni}]\wt{C}^{\synt}_Xi_*, 
$$
as wanted.
\end{proof}
\begin{lemma}
\label{finish}
Set $h_{\infty }:Y_{\infty }=Y\times\{\infty\}\hookrightarrow {\mathbb P}^1_X$ and 
$h_{0}:Y_{0}=Y\times \{0\}\hookrightarrow {\mathbb P}^1_X$.
Then there exists a  quasi-isomorphism
$$
E^{\prime}_{n}(r)_{{\mathbb P}^1_X}( D_{{\mathbb P}^1_X})\oplus 
h_{\infty *}E^{1}_{n}(r-1)_{Y_{\infty }}( D_{Y_{\infty }})[-2]
\simeq \R \pi_*E^{\prime}_{n}(r)_W( D_W),\quad r\geq 2,
$$
where the map $\pi^*: E^{\prime}_{n}(r)_{{\mathbb P}^1_X}( D_{{\mathbb P}^1_X})\to
\R \pi_*E^{\prime}_{n}(r)_W( D_W)$ is the map induced by $\pi:W\to
{\mathbb P}^1_X$ and the complexes $E^{\prime}_n$ and $E^1_n$ are not truncated.
\end{lemma}
\begin{proof}Just to simplify the notation we will assume that the divisor  $D^{\prime}$ is trivial. 
We claim that we have the following map of distinguished triangles
$$
\xymatrix{(h_{0*}E^{1}_{n}(r-1)_{Y_{0}}\oplus j_{\infty *}E^{1}_{n}(r-1)_{X_{\infty }})[-2]\ar[d]\ar[r] & (h_{0*}E^{1}_{n}(r-1)_{Y_{0}}\oplus j_{\infty *}E^{1}_{n}(r-1)_{X_{\infty }}\oplus h_{\infty *}E^{1}_{n}(r-1)_{Y_{\infty }})[-2]\ar[d]\\
E^{\prime}_{n}(r)_{{\mathbb P}^1_X}\ar[d]\ar[r]^{\pi^*} & \R \pi_*E^{\prime}_{n}(r)_W\ar[d]\\
E^{\prime}_{n}(r)_{{\mathbb P}^1_X}({\mathbb P}^1_Y\cup X_{\infty})\ar[r]^{\pi^*}_{\sim} & 
\R \pi_*E^{\prime}_{n}(r)_W({\mathbb P}^1_Y\cup X_{\infty }\cup {\mathbb P}(\sn^{\vee}_{Y/X}\oplus\so_Y)).
}
$$
where the bottom row is a quasi-isomorphism because log-blow-ups do not change syntomic cohomology
\cite[Theorem 5.10]{N7} and \'etale cohomology. 

  Indeed, by gluing the Gysin distinguished triangles we get the following two compatible (via the map $\pi_*$) weight ``exact'' sequences.
\begin{align*}
h_{\infty *}E^{2}_{n}(r-2)_{Y_{\infty } }[-4]& \stackrel{f}{\to }i_*E^{1}_{n}(r-1)_{{\mathbb P}^1_Y}[-2]\oplus j_{\infty *}E^{1}_{n}(r-1)_{X_{\infty }}[-2]\to E^{\prime}_{n}(r)_{{\mathbb P}^1_X}\stackrel{g}{\to} E^{\prime}_{n}(r)_{{\mathbb P}^1_X}({\mathbb P}^1_Y\cup X_{\infty })\\
\overline{i}_*j_{\infty *}E^{2}_{n}(r-2)_{Y_{\infty } }[-4] & \oplus t_*s_*E^{2}_{n}(r-2)_{H_{\infty } }[-4]\stackrel{f}{\to }
\overline{i}_*E^{1}_{n}(r-1)_{{\mathbb P}^1_Y}[-2]\oplus l_*E^{1}_{n}(r-1)_{\wt{X}}[-2]\oplus 
t_*E^{1}_{n}(r-1)_{{\mathbb P}(\sn^{\vee}_{Y/X}\oplus\so_Y)}[-2]\to \\
 & \to E^{\prime}_{n}(r)_W \stackrel{g}{\to }E^{\prime}_{n}(r)_W({\mathbb P}^1_Y\cup X_{\infty }\cup {\mathbb P}(\sn^{\vee}_{Y/X}\oplus\so_Y)),
\end{align*}
where $s: H_{\infty }\hookrightarrow {\mathbb P}(\sn^{\vee}_{Y/X}\oplus\so_Y)$ is the hyperplane at infinity. That is $\cofiber(f)\simeq \fiber(g)$. 

  To compute the cofiber of the top map $f$
we use the Gysin distinguished triangle
$$j_{\infty *}E^{2}_{n}(r-2)_{Y_{\infty } }[-2]\to E^{1}_{n}(r-1)_{{\mathbb P}^1_Y}\to E^{1}_{n}(r-1)_{{\mathbb P}^1_Y}(Y_{\infty })
$$
and the projective space theorem  from Lemma (\ref{projective1})
$$E^{1}_{n}(r-1)_{Y_{0} }\oplus E^{2}_{n}(r-2)_{Y_{\infty} }[-2]\simeq
\R p_*E^{1}_{n}(r-1)_{{\mathbb P}^1_Y},
$$
where $p:{\mathbb P}^1_Y\to Y$, 
to get that
$$\cofiber(f)\simeq h_{0*}E^{1}_{n}(r-1)_{Y_{0} }[-2]\oplus h_{\infty *}E^{1}_{n}(r-1)_{X_{\infty }}[-2].
$$

 To compute the cofiber of the bottom map $f$ we use the above to get
$$\cofiber(f)\simeq \overline{i}_*j_{0*}E^{1}_{n}(r-1)_{Y_0}[-2]\oplus l_*E^{1}_{n}(r-1)_{\wt{X}}[-2]\oplus t_*\cofiber (s_* E^{2}_{n}(r-2)_{H_{\infty } }[-2]\to 
E^{1}_{n}(r-1)_{{\mathbb P}(\sn^{\vee}_{Y/X}\oplus\so_Y)}
)[-2].
$$
Next, we use the Gysin distinguished triangle
$$
 s_*E^{2}_{n}(r-2)_{H_{\infty } }[-2]\to 
E^{1}_{n}(r-1)_{{\mathbb P}(\sn^{\vee}_{Y/X}\oplus\so_Y)}\to
E^{1}_{n}(r-1)_{{\mathbb P}(\sn^{\vee}_{Y/X}\oplus\so_Y)}(H_{\infty })
$$
and the homotopy quasi-isomorphism
$$E^{1}_{n}(r-1)_{Y_{\infty } }\stackrel{\sim}{\to }\R p^{\prime}_*E^{1}_{n}(r-1)_{{\mathbb P}(\sn^{\vee}_{Y/X}\oplus\so_Y)}(H_{\infty }),
$$
where $p^{\prime}:{\mathbb P}(\sn^{\vee}_{Y/X}\oplus\so_Y)\to Y$,
to conclude that
$$\cofiber(f)\simeq \overline{i}_*j_{0*}E^{1}_{n}(r-1)_{Y_{0}}[-2]\oplus l_*E^{1}_{n}(r-1)_{\wt{X}}[-2]\oplus t_*i_{\infty *}E^{1}_{n}(r-1)_{Y_{\infty }}[-2],
$$
as wanted.
\end{proof}

   Having  Lemma \ref{deformation}, we can define a total Chern class map
$$C^{\synt}_{X(D)}:\quad G_X( D)\to \prod_{i\geq 0}\sk(2i,\wt{E}_n^{\prime}(i)_X( D))
$$
to get a map of homotopy cofiber sequences in the diagram (\ref{induction}) but without the $[p^{Ni}]$ factor. We set $C^{\synt}_{X(D),0}$ 
equal to the composition of the rank map with the map $ {\mathbf Z}\to E^{\prime}_n(0)$.
It suffices now to show that
 the  Chern class map $C^{\synt}_{X(D)}$, multiplied by $[p^{Ni}]$, is homotopic to $[p^{Ni}][p^{2Ni}]C^{\synt}_{X\setminus D}$. Note that, by construction, the  map 
 $C^{\synt}_{X(D)} $ makes the lower square in the diagram (\ref{induction}) (without the $[p^{Ni}]$ factor) commute and, by functoriality, the same is true of the map $[p^{2Ni}]C^{\synt}_{X\setminus D}$. 
  Since, 
  by Lemma \ref{keylemma}, the group 
$$\Hom_{\spp}(i_*G_{D_1}( D^{\prime}_1)[1],\sk(2i,\wt{E}^{\prime}_{n}(i)_X(D)))
$$
is annihilated by $[p^{Ni}]$,  the maps $[p^{Ni}]C^{\synt}_{X(D)}, [p^{Ni}][p^{2Ni}]C^{\synt}_{X\setminus D}$ are  homotopic. This finishes the proof of our theorem.

\end{proof}

\begin{remark} 
  Theorem \ref{logChern} is also true  for log-\'etale cohomology $H^*(X(D)_{\eet},{\mathbf Z}/p^n(*))$. Same proof works. Here $X(D)$ is a Zariski log-scheme over $K$ or $\ovk$. Since $H^*(X(D)_{\eet},{\mathbf Z}/p^n(*))\simeq H^*(U_{\eet},{\mathbf Z}/p^n(*))$, for  $U=X\setminus D$, no additional constants are needed. The key Lemma \ref{keylemma}
  follows easily from purity in \'etale cohomology. Indeed, 
we have 
\begin{align*}
\Hom_{\spp} & (i_*\sll,\sk(2i,\R \ve_*{\mathbf Z}/p^n(i)_{X(D)})) =\Hom_{\spp}(i_*\sll, \sk(2i,\R j_*\R \ve_*{\mathbf Z}/p^n(i)_U))\\
& =
  \Hom_{\spp}(\sll, \sk(2i,\R i^!\R j_*\R \ve_*{\mathbf Z}/p^n(i)_{U})) =\Hom_{\spp}(\sll, \sk(2i,\R (i^!j_*)\R \ve_*{\mathbf Z}/p^n(i)_{U}))=0,
\end{align*}
$\varepsilon$ is the projection from the log-\'etale site to the Nisnevich site,  $j: U\hookrightarrow X$, and $i: D_1\hookrightarrow X$.
\end{remark}
\subsubsection{Chern classes and Gysin sequences}The above  computations imply that  Chern classes from motivic cohomology are compatible with  Gysin sequences. To state this compatibility we need first to evaluate the twisted Chern classes from the diagram (\ref{induction}) on motivic cohomology. 

  Let $i:Y\hookrightarrow X$ be a closed immersion of regular syntomic schemes over $W(k)$ of codimention one and let $D$ be a divisor on $X$ such that the divisors $D_X=Y\cup D$ and $D_Y=D\cap Y$ have  relative simple normal crossings over $W(k)$ and all the irreducible components of $D_X$, $D_Y$ are regular. Fix $m$ as in Section \ref{6.1}. Consider the map
$$\overline{p}^{\synt}_{i,j}:\quad  K_j(Y\setminus D_Y ,{\mathbf Z}/p^n) 
\to  H^{2i-j-2}(Y,S^{\prime}_n(i-1)_Y(D_Y)),\quad
j\geq 2, 
$$
equal to the composition 
\begin{align*}
[P^j(Y\setminus D_Y),K_{Y\setminus D_Y}] & =[P^jY,j_*K_{Y\setminus D_Y}]
 \veryverylomapr{P_i(Y/X;\wt{C}^{\synt}_{Y\setminus D_Y})} [P^jY,\sk(2i-2,\tilde{S}^{\prime}_n(i-1){}_{Y}(D_Y))]\\
& =
H^{-j}(Y,\sk(2i-2,\tilde{S}^{\prime}_n(i-1){}_{Y}(D_Y);{\mathbf Z}/p^n)
\stackrel{f}{\rightarrow}  H^{2i-j-2}(Y,S^{\prime}_n(i-1)_Y(D_Y)),
\end{align*}
where $j:Y\setminus D_Y\hookrightarrow Y$ is the natural immersion.  We skipped the index $m$ in the notation for Chern classes. Similarly, we get the map
$$\overline{p}^{\synt}_{i,j}:\quad  K_j(Y\setminus D_Y ,{\mathbf Z}/p^n) 
\to  H^{2i-j-2}(Y,S_n(i-1)_Y(D_Y)),\quad
j\geq 2, i\leq p.
$$
\begin{lemma} 
\label{comput}Let $j\geq 2$ for $p$ odd and let $j\geq 3$ for $p=2$.
\begin{enumerate}
\item The map $\overline{p}^{\synt}_{i,j}$ restricts to zero on 
$F^{i}_{\gamma}K_j(Y\setminus D_Y ,{\mathbf Z}/p^n)$, $i\geq 2$.
\item We have the equality
$$
\overline{p}^{\synt}_{i,j}=-(i-1)\overline{c}^{\synt}_{i-1,j}:\quad 
 F^{i-1}_{\gamma}K_j(Y\setminus D_Y ,{\mathbf Z}/p^n)\to H^{2i-j-2}(Y,S^{\prime}_n(i-1)_Y(D_Y)).
$$
\end{enumerate}
\end{lemma}
\begin{proof}
 We have 
$$\overline{p}^{\synt}_{i,j}=P_i(c_1^{\synt,m}(\sn^{\vee}_{Y/X});\overline{c}^{\synt}_{0,j},
\overline{c}^{\synt}_{1,j},\ldots,\overline{c}^{\synt}_{i-1,j}),
$$
where $P_i(U_1;T_0,T_1,\ldots,T_{i-1})$ is a certain universal polynomial 
with integral coefficients \cite[1.3]{Jou}. By Lemma
\ref{prop}, $\overline{c}^{\synt}_{k,j}(x)=0$, $k\leq i-1$, for $x\in F^{i}_{\gamma}K_j(Y\setminus D_Y ,{\mathbf Z}/p^n)$. Hence $$\overline{p}^{\synt}_{i,j}(x)=P_i(c_1^{\synt,m}(\sn^{\vee}_{Y/X});0,\ldots,0)=0,\quad x\in F^{i}_{\gamma}K_j(Y\setminus D_Y ,{\mathbf Z}/p^n).
$$
This proves the first statement of the lemma. 

  For the second, by the same argument, we have 
$$\overline{p}^{\synt}_{i,j}(x)=P_i(c_1^{\synt,m}(\sn^{\vee}_{Y/X});0,\ldots,0,
\overline{c}^{\synt}_{i-1,j}(x)),\quad x\in F^{i-1}_{\gamma}K_j(Y\setminus D_Y ,{\mathbf Z}/p^n).
$$
Now the computation is quite formal. By the defining property of the polynomials $P_i$ \cite[1.1.7]{Jou}
and the product formula \cite[Exp.V, 6.6.1]{S6}, we have
\begin{align*}
P_i(c_1^{\synt,m}(\sn^{\vee}_{Y/X}) & ;0,\ldots,0,
\overline{c}^{\synt}_{i-1,j}(x))c^{\synt,m}_1(\sn^{\vee}_{Y/X})=
-[(1+\overline{c}^{\synt}_{i-1,j}(x))\star \lambda_{-1}(c_1^{\synt,m}(\sn^{\vee}_{Y/X}))]^{(i)}\\
& =-[(1+\overline{c}^{\synt}_{i-1,j}(x))\star (1-c_1^{\synt,m}(\sn^{\vee}_{Y/X})+\ldots )]^{(i)}\\
& =-(1+(i-1)!/(i-2)!\overline{c}^{\synt}_{i-1,j}(x)c_1^{\synt,m}(\sn^{\vee}_{Y/X})+\ldots )^{(i)}
\\ &
=
-(i-1)\overline{c}^{\synt}_{i-1,j}(x)c_1^{\synt,m}(\sn^{\vee}_{Y/X})
\end{align*}
Since we can see this as an equality in the ring of polynomials with integral coefficients, we get that
$$P_i(c_1^{\synt,m}(\sn^{\vee}_{Y/X});0,\ldots,0,
\overline{c}^{\synt}_{i-1,j}(x))=-(i-1)\overline{c}^{\synt}_{i-1,j}(x),
$$
as wanted.
\end{proof}

  And here is the compatibility of Chern classes (from motivic cohomology) with Gysin sequences we wanted.
\begin{corollary}Let $2b-a\geq 3$, $p^n >2$. There exists a constant $A=A(d,a,b)$ such that we 
 have  the following
maps of partial Gysin sequences
$$
\xymatrix{
K^{b-1}_{2b-a}(Y(D_Y)) \ar[r]^-{Ai_*}\ar[d]^{\gamma^{\prime}p^{m(2b+3)}(1-b)\overline{c}^{\synt,m}_{b-1,2b-a}} &
K^b_{2b-a}(X(D))\ar[r]^-{Aj^*}\ar[d]^{p^{3bm}\overline{c}^{\synt,m}_{b,2b-a}} & K^b_{2b-a}(X(D_X))\ar[r]^-{A\partial} \ar[d]^{p^{3bm}\overline{c}^{\synt,m}_{b,2b-a}} &
 K^{b-1}_{2b-a-1}(Y(D_Y))\ar[d]^{\gamma^{\prime}p^{m(2b+3)}(1-b)\overline{c}^{\synt,m}_{b-1,2b-a-1}}\ar[r]^-{Ai_*} & K^b_{2b-a-1}(X(D_X))\ar[d]^{p^{3bm}\overline{c}^{\synt,m}_{b,2b-a-1}}\\
H^{a-2}_1(Y(D_Y),b-1)\ar[r]^-{Ai_!} & H^a(X(D),b)
\ar[r]^-{Aj^*} & H^a(X(D),b)\ar[r]^-{A\partial} & H^{a-1}_1(Y(D_Y),b-1)\ar[r]^-{Ai_!} & H^{a+1}(X(D^{\prime}),b)}
$$
Here we set 
\begin{align*}
& K^i_{j}(T(D_T)) =\gr^i_{\gamma}K_{j}(T(D_{T}),\mathbf{Z}/p^n),\quad
  H^*(T(D_T), *) =H^*(T,S^{\prime}_n(*)_T(D_T)),\\
& H^*_{1}(T(D_T), *) =H^*(T,S^{1}_n(*)_T(D_T)).
\end{align*}
\end{corollary}
\begin{proof}
The top sequence exists because of Lemma \ref{exactseq1}. The constant $A=N(d,b,2b-a)N(d,b+1,2b-a)$. The diagram itself is obtained from Theorem \ref{logChern} and 
Lemma \ref{comput} (which gives us that
$\overline{p}^{\synt,m}_{b,2b-a}=(1-b)\overline{c}^{\synt,m}_{b-1,2b-a}$ and $\overline{p}^{\synt,m}_{b,2b-a-1}=(1-b)\overline{c}^{\synt,m}_{b-1,2b-a-1}$). 
\end{proof}

  The above lemma and corollary hold for the $S$, $E$, and $E^{\prime}$  cohomologies as well.


\begin{thebibliography}{BEK}
\bibitem{BDR} M.~Bertolini, H.~Darmon, V.~Rotger, {\em 
Beilinson-Flach elements and Euler systems II: the Birch and Swinnerton-Dyer conjecture for Hasse-Weil-Artin L-series,}
Journal of Algebraic Geometry 24 (2015) 569--604.
\bibitem{Bl} S.~Bloch, {\em  Algebraic cycles and higher $K$-theory}, Adv. in Math. 61 (1986), no. 3, 267--304. 
\bibitem{BEK} S.~Bloch, H.~Esnault, M.~Kerz, {\em $p$-adic deformation of algebraic cycle classes},  Invent. Math. 195 (2014), 673--722.
\bibitem{BK} S.~Bloch, K.~Kato, {\em $p$-adic \'etale cohomology}. Publ. IHES No. 63 (1986), 107-152. 
\bibitem{BKa} A.K.~Bousfield, D.~ Kan, {\em Homotopy limits, completions and localizations}.
 Lecture Notes in Mathematics, Vol. 304. Springer-Verlag, Berlin-New York, 1972. 
 \bibitem{CN} P.~Colmez, W.~Nizio{\l}, {\em Syntomic complexes and $p$-adic nearby cycles},  arXiv:1505.06471. 
 \bibitem {EN} V.~Ertl, W.~Nizio{\l}, {\em Syntomic cohomology and $p$-adic motivic cohomology}, arXiv:1603.01705.
\bibitem{FM} J.-M.~Fontaine and  W.~Messing, {\em  $p$-adic periods and $p$-adic \'{e}tale
cohomology}, Current Trends in Arithmetical Algebraic Geometry
(K.~Ribet, ed.), Contemporary
Math., vol. 67, Amer. Math. Soc., Providence, 1987, 179--207.
\bibitem{FLL} W.~Fulton, S.~Lang, {\em Riemann-Roch algebra}. Grundlehren der Mathematischen Wissenschaften [Fundamental Principles of Mathematical Sciences], 277. Springer-Verlag, New York, 1985. 
\bibitem{GD} T.~Geisser, {\em Motivic cohomology over Dedekind rings}, Math. Z. 248 (2004), no. 4, 773--794. 
\bibitem{GL} T.~Geisser, M.~ Levine, {\em  The $K$-theory of fields in characteristic $p$},
 Invent. Math. {\bf 139} (2000), no. 3,
  459--493. 
\bibitem{Gi} H.~Gillet, {\em Riemann-Roch theorems for 
higher algebraic K-theory}, Adv. Math. {\bf 40} (1981), 203--289.
\bibitem{GS1} H.~ Gillet, C.~ Soul\'e, {\em Descent, motives and $K$-theory.} J. Reine Angew. Math. 478 (1996), 127--176. 
\bibitem{GS} H.~Gillet, C.~ Soul\'e, {\em Filtrations on higher algebraic $K$-theory},
Algebraic $K$-theory (Seattle, WA, 1997), 89--148,
  Proc. Sympos. Pure Math., {\bf 67}, Amer. Math. Soc., Providence, RI, 1999.
\bibitem{Gr} M.~Gros, {\em R\'egulateurs syntomiques et valeur de 
fonctions L p-adiques I},
Invent. Math. {\bf 99} (1990), 293--320.
\bibitem{S6}{\em Th\'eorie des intersections et th\'eor\'eme de Riemann-Roch. S\'eminaire de G\'eom\'etrie Alg\'ebrique du Bois-Marie 1966--1967 (SGA 6)}.
 Dirig\'e par P. Berthelot, A. Grothendieck et L. Illusie. Avec la collaboration de D. Ferrand, J. P. Jouanolou, O. Jussila, S. Kleiman, M. Raynaud et J. P. Serre. Lecture Notes in Mathematics, Vol. 225. Springer-Verlag, Berlin-New York, 1971.
 \bibitem{HW} A.~Huber, J.~Wildeshaus, {\em Classical motivic polylogarithm according to Beilinson and Deligne}. Doc. Math. 3 (1998), 27--133 (electronic).
 \bibitem{Il} L.~Illusie, {\em Complexe de de Rham-Witt et cohomologie cristalline}, Ann. ENS, 4e s\'erie, tome 12, no 4 (1979), 501--661.
\bibitem{J1} J.~F.~Jardine, {\em Simplicial presheaves},
 J. Pure Appl. Algebra {\bf 47} (1987), no. 1, 35--87.
\bibitem{Jeu} R.~De Jeu, {\em Zagier's conjecture and wedge complexes in algebraic $K$-theory}. 
 Compositio Math.  96  (1995),  no. 2, 197--247.
\bibitem{Jou} J.P.~Jouanolou, {Riemann-Roch sans d\'enominateurs.}  Invent. Math.  11  (1970), 15--26. 
\bibitem{Kah} B.~Kahn, {\em $K$-theory of semi-local rings with finite coefficients and \'etale cohomology.}  $K$-Theory   25  (2002),  no. 2, 99--138.
\bibitem{K} K.~Kato, {\em On p-adic vanishing cycles (application of ideas
of Fontaine-Messing)}, Algebraic Geometry, Sendai, 1985, 
Adv. Stud. Pure Math. {\bf 10},
North-Holland, Amsterdam-New York, 1987, 207--251.
\bibitem{KZ} M.~Kerz, {\em The Gersten conjecture for Milnor K-theory.} Invent. Math. 175,  (2009), no. 1,  1-33.
\bibitem{Kur}  M.~Kurihara, {\em A note on $p$-adic \'etale cohomology}. Proc. Japan Acad. Ser. A Math. Sci. 63 (1987), no. 7, 275--278. 
\bibitem{Ls} M.~Levine, {\em $K$-theory and motivic cohomology of schemes}, preprint, 2004.
\bibitem{N0} W.~Nizio{\l}, {\em  Cohomology of crystalline representations}. Duke Math. J. 71 (1993), no. 3, 747--791.
\bibitem{N4} W.~Nizio{\l}, {\em Crystalline Conjecture via K-theory}, 
 Ann. ENS (4) {\bf 31} (1998), no. 5, 659--681.
\bibitem{N7} W.~Nizio{\l}, {\em Toric singularities: log-blow-ups and global resolutions},
 J. Algebraic Geom. 15 (2006), 1-29.
\bibitem{N2} W.~Nizio{\l}, {\em Semistable Conjecture via K-theory}, Duke Math. J. 141 (2008), no. 1, 151-178. 
\bibitem{N10} W.~Nizio{\l}, {\em On uniqueness of $p$-adic period morphisms},  Pure Appl. Math. Q.  5  (2009),  no. 1, 163-212. 
\bibitem{PO} O.~Pushin, {\em Higher Chern classes and Steenrod operations in motivic cohomology}. $K$-Theory 31 (2004), no. 4, 307--321. 
 \bibitem{Sa} K.~Sato, {\em  Cycle classes for $p$-adic \'etale Tate twists and the image of $p$-adic regulators}. Doc. Math. 18 (2013), 177--247. 
 \bibitem{Sw} M.~Somekawa, {\em Log-syntomic regulators and $p$-adic polylogarithms.}
  $K$-Theory  17  (1999),  no. 3, 265--294. 
\bibitem{So} C.~Soul\'{e}, {\em Operations on \'etale K-theory. Applications},
Algebraic K-theory I, Lect. Notes Math. 
{\bf 966}, Springer-Verlag, Berlin, Heidelberg and New York, 1982, 271--303.
\bibitem{T} R.~Thomason, {\em Algebraic K-theory and \'etale cohomology},
Ann. ENS  {\bf 18} (1985), 437--552.
\bibitem{T2} R.~Thomason, Th.~Trobaugh, {\em Higher algebraic $K$-theory of schemes and of derived categories}. The Grothendieck Festschrift, Vol. III, 247--435, Progr. Math., 88, 
Birkh\"{a}user Boston, Boston, MA, 1990. 
\bibitem{Ts1} T.~Tsuji, {\em $p$-adic Hodge theory in the semi-stable reduction case}, Proceedings of the International Congress of Mathematicians, Vol. II (Berlin, 1998). Doc. Math. 1998, Extra Vol. II, 207--216 (electronic). 
\bibitem{Ts} T.~Tsuji, {\em p-adic \'etale and crystalline cohomology
in the semistable reduction case}, Invent. Math. {\bf 137} (1999), no. 2, 233--411. 
\bibitem{Ts1} T.~Tsuji, {\em On p-adic nearby cycles of log smooth families.} Bull. SMF  128 (2000), no. 4, 529--575. 
\bibitem{Ts2} T.~Tsuji, {\em On the maximal unramified quotients of $p$-adic \'etale cohomology groups and logarithmic Hodge-Witt sheaves}. Kazuya Kato's fiftieth birthday.  Doc. Math.  2003,  Extra Vol., 833--890 (electronic).
\end{thebibliography}
\end{document}